\documentclass[12pt]{amsart}
\usepackage{amsmath}
\usepackage{amssymb}
\usepackage{amsbsy}
\usepackage{amscd}
\usepackage{amsfonts}
\usepackage{mathrsfs}
\usepackage{verbatim}
\usepackage{version}
\usepackage[all]{xy}
\usepackage{pstricks}
\usepackage{pst-all}

\makeatletter
\def\cal{\mathcal}

\def\frak{\mathfrak}

\newenvironment{pf*}[1]{\proof[#1]}{\endproof}
\makeatother

\makeatletter
\@ifclasslater{amsart}{1999/11/24}{}{
\renewcommand*\subjclass[2][1991]{%
  \def\@subjclass{#2}%
  \@ifundefined{subjclassname@#1}{%
    \ClassWarning{\@classname}{Unknown edition (#1) of Mathematics
      Subject Classification; using '1991'.}%
  }{%
    \@xp\let\@xp\subjclassname\csname subjclassname@#1\endcsname
  }%
}
\renewcommand{\subjclassname}{%
  \textup{1991} Mathematics Subject Classification}
\@xp\let\csname subjclassname@1991\endcsname \subjclassname
\@namedef{subjclassname@2000}{%
  \textup{2000} Mathematics Subject Classification}
}
\makeatother
\newtheorem{Theorem}[equation]{Theorem}
\newtheorem{Corollary}[equation]{Corollary}
\newtheorem{Lemma}[equation]{Lemma}
\newtheorem{Proposition}[equation]{Proposition}

\theoremstyle{definition}
\newtheorem{Definition}[equation]{Definition}

\newtheorem{Notation}[equation]{Notation}

\newtheorem{ConjQues}[equation]{Conjecture/Question}

\theoremstyle{remark}
\newtheorem{Remark}[equation]{Remark}

\numberwithin{equation}{section}
\newcommand{\thmref}[1]{Theorem~\ref{#1}}
\newcommand{\secref}[1]{\S\ref{#1}}
\newcommand{\lemref}[1]{Lemma~\ref{#1}}
\newcommand{\propref}[1]{Proposition~\ref{#1}}
\newcommand{\corref}[1]{Corollary~\ref{#1}}

\newcommand{\defref}[1]{Definition~\ref{#1}}
\newcommand{\remref}[1]{Remark~\ref{#1}}
\newcommand{\C}{{\Bbb C}}
\newcommand{\Z}{{\Bbb Z}}
\newcommand{\Q}{{\Bbb Q}}
\newcommand{\R}{{\Bbb R}}
\newcommand{\la}{{\frak \lambda}}

\newcommand{\Hom}{\operatorname{Hom}}
\newcommand{\Ext}{\operatorname{Ext}}

\newcommand{\Tor}{\underline{Tor}}
\newcommand{\D}{{\mathcal D}}

\newcommand{\ra}{\rightarrow}%%%%%%%% added by Yuan Yao
\newcommand{\ls}{|L|}
\newcommand{\z}{\Theta}%%%%%%   added by Yuan Yao

\newcommand{\Pic}{\operatorname{Pic}}

\newcommand{\bbeta}{\boldsymbol\beta}

\newcommand{\res}{\mathop{\text{\rm res}}}
\newcommand{\rk}{\mathop{{\rm rk}}}
\newcommand{\coeff}{\mathop{\text{\rm Coeff}}}

\newcommand{\Coeff}{\mathop{\text{\rm Coeff}}}
\def\oo{{\cal O}}
\def\I{{\cal I}}

\def\P{{\mathbb P}}
\def\<{\langle}
\def\>{\rangle}

\def\F{{\mathcal F}}

\def\G{{\mathcal G}}
\def\E{{\mathcal E}}
\def\cc{{\mathcal C}}
\def\WR{{\widetilde R}}
\def\WS{{\widetilde S}}
\def\WT{{\widetilde \theta}}

\def\th{\theta}

\makeatletter
\newcommand{\vechatom}{
    {\Vec{\omega}}
    \,\smash[b]{\hbox{\lower2\ex@\hbox{$\m@th\hat{\null}$}}}
}
\makeatother

\def\R{{\mathbb R}}
\def\RR{{\mathcal R}}
\def\Q{{\mathbb Q}}

\def\P{{\mathbb P}}
\def\Z{{\mathbb Z}}
\def\C{{\mathbb C}}
\def\H{{\mathcal H}}
\def\oo{{\mathcal O}}
\def\L{{\Lambda}}

\begin{document}

\title[Generating functions for K-theoretic  Donaldson invariants]
{Generating functions for K-theoretic  Donaldson invariants and Le Potier's strange duality}
\author{Lothar G\"ottsche,~Yao Yuan}
\address{International Centre for Theoretical Physics, Strada Costiera 11,
34014 Trieste, Italy}
\email{gottsche@ictp.it}
\address{Yau Mathematical Sciences Center, Tsinghua University, 100084, Beijing, P. R. China}
\email{yyuan@math.tsinghua.edu.cn}

%\date: \today
\begin{abstract} For a projective algebraic surface $X$, with an ample line bundle $H$, let $M_H^X(c)$ be the moduli space $H$-semistable sheaves $\E$ of class $c$ in the Grothendieck group $K(X)$.   We write $c=(r,c_1,c_2)$, or $c=(r,c_1,\chi)$ with $r$ the rank, $c_1,c_2$, the Chern classes and $\chi$ the holomorphic Euler characteristic. We also write $M_H^X(2,c_1,c_2)=M_X^X(c_1,d)$, with $d=4c_2-c_1^2$.
%For  line bundles $L$ on $X$, let $\mu(L)$ be the corresponding determinant line bundles  on $M_H^X(c_1,d)$.
The $K$-theoretic Donaldson invariants are the holomorphic Euler characteristics $\chi(M_H^X(c_1,d),\mu(L))$, where $\mu(L)$ is the determinant line bundle associated to a line bundle on $X$. More generally for suitable classes $c^*\in K(X)$ there is a determinant line bundle $\D_{c,c^*}$ on $M^X_H(c)$.
%Under suitable conditions they coincide with the dimensions of the corresponding spaces of sections $H^0(M_H^X(c_1,d),\mu(L))$.
We first compute some generating functions for $K$-theoretic Donaldson invariants on $\P^2$ and rational ruled surfaces, using the
wallcrossing formula of \cite{GNY}.

Then we show that Le Potier's strange duality conjecture relating $H^0(M^X_H(c),\D_{c,c^*})$ and $H^0(M^X_H(c^*),\D_{c^*,c})$ holds for the cases $c=(2,c_1=0,c_2>2)$ and $c^{*}=(0,L,\chi=0)$ with $L=-K_X$ on $\P^2$, and $L=-K_X$ or $-K_X+F$ on $\P^1\times\P^1$ and $\widehat \P^2$ with $F$ the fiber class of the ruling, and also the case $c=(2,H,c_2)$ and $c^*=(0,2H,\chi=-1)$ on $\P^2$.

\end{abstract}

\maketitle
\tableofcontents

\section{Introduction}
In this whole paper let $X$ be a  simply connected nonsingular projective surface over $\C$, with its anti-canonical divisor $-K_X$ ample.  In particular $X$ is a rational surface.
For  $H$ ample on $X$, $c_1\in H^2(X,\Z)$, $c_2\in \Z$, let   $M:=M^X_H(c_1,d)$ with $d=4c_2-c_1^2$ be the moduli space of $H$-semistable sheaves with Chern classes $c_1,c_2$.
For a line bundle $L$ on $X$ there exists a corresponding determinant line bundle $\mu(L)$ on $M$. The Donaldson invariants are given as the top self intersection numbers
$\int_M c_1(\mu(L))^{\dim M}$ on the moduli spaces. In \cite{GNY} the $K$-theoretic Donaldson invariants are introduced as the holomorphic Euler characteristics $\chi(M,\mu(L))$.
Like for the usual Donaldson invariants it is an interesting problem to understand the $K$-theoretic Donaldson invariants and determine their generating functions. In this paper we will do this in a number of cases for $\P^2$ and rational ruled surfaces. This result is then applied to prove some cases of Le Potier's strange duality conjecture, which is a duality between spaces of sections of determinant bundles on different moduli spaces of sheaves on surfaces.

  \subsection{$K$-theoretic Donaldson invariants}\label{backDong}
Let $K(X)$ be the Grothendieck group of coherent sheaves over $X$.
Let $c$ be an element in $K(X)$, which is the class of a coherent rank $2$ sheaf 
with Chern classes $c_1,c_2$.  We write $M^X_H(c)=M^X_H(c_1,d)$ with $d=4c_2-c_1^2$ for the moduli space of $H$-semistable sheaves in class $c$.
%\begin{LG} Introduce class $c$.\end{LG}\begin{YY}Done, see the words in italic font.\end{YY}
%Let $v\in K_c$, where $c$ is the class of a coherent rank $2$ sheaf with Chern classes $c_1,c_2$.
Let $L$ be a line bundle on $X$ and assume that $\<c_1(L),c_1\>$ is even with $\<-,-\>$ the intersection form on $H^2(X,\mathbb{Z})$.
Then %for $c$ the class of a rank $2$ coherent sheaf with Chern classes $c_1,c_2$, 
we put
\begin{equation}\label{eq:uL} v(L):=(1-L^{-1})+\<\frac{c_1(L)}{2}, (c_1(L)+K_X-c_1)\>[\oo_x].
%~~\begin{NB}~K_X+c_1~or~K_X-c_1\end{NB}
\end{equation}
Note that $v(L)$ is independent of $c_2$.
Assume that $H$ is $c$-general. Then we have a well-defined \emph{determinant line bundle} 
$\mu(L):=
\lambda(v(L))\in \Pic(M^X_H(c_1,d))$ associated to $v(L)$.
The {\it $K$-theoretic Donaldson invariant\/} of $X$ with respect to $L,c_1,d,H$ is
$\chi(M^X_H(c_1,d),\mu(L))$.

\subsection{Results for rational ruled surfaces.}
We denote by $\Sigma_e=\P(\oo_{\P_1}\oplus \oo_{\P^1}(e))$ the $e$-th rational ruled  surface. We will restrict to the cases $e=0$, i.e.
$X=\P^1\times \P^1$ and $e=1$, i.e. $X=\widehat \P^2$, the blowup of $\P^2$ in a point.
In the case $X=\P^1\times \P^1$ we denote $G$ the class of the fibre of the second projection to $\P^1$.
In the case $X=\widehat \P^2$ let $H$ be the pullback of the hyperplane class on $\P^2$ and $E$ the exceptional divisor.
We write $F=H-E$ and $G=(H+E)/2$. Note that $G$ only lies in $\frac{1}{2}H^2(\widehat \P^2,\Z).$

For power series $f(\Lambda)=\sum_{d\ge 0} f_n \Lambda^d, \quad g(\Lambda)=\sum_{d\ge 0} g_n \Lambda^d\in \Q[[\Lambda]]$, we write
$f(\Lambda)\equiv g(\Lambda)$ if there exists a $d_0\ge 0$ with $f_d=g_d$ for all $d\ge d_0$.
\begin{Theorem}\label{p11t}
For $X=\P^1\times\P^1$ and $X=\widehat \P^2$ and
for $\frac{a}{b}\ge \frac{n+2}{4}$ the following hold.
\begin{enumerate}
\item For $n\in \Z$ we have
\begin{align*}
1+\sum_{d>0}\chi(M_{aF+bG}^{X}(F,d),\mu(nF))\Lambda^d&=\frac{1}{(1-\Lambda^4)^{n+1}},\\
1+(n+1)\Lambda^4+\sum_{d>4}\chi(M_{aF+bG}^{X}(0,d),\mu(nF))\Lambda^d&=\frac{1}{(1-\Lambda^4)^{n+1}}.
\end{align*}
\item For $n\in \Z$ in case $X=\P^1\times\P^1$ and for $n\in\Z+\frac{1}{2}$ in case $X=\widehat \P^2$ we have
$$1+(2n+2)\Lambda^4+\sum_{d>4}\chi(M_{aF+bG}^{X}(0,d),\mu(nF+G))\Lambda^{d}=
\frac{1}{(1-\Lambda^4)^{2n+2}}.$$
\item For $n\in \Z$ we have
\begin{align*}
\sum_{d>0}\chi(M_{aF+bG}^{X}(F,d),\mu(nF+2G))\Lambda^{d}&=
\frac{1}{2}\frac{(1+\Lambda^4)^n-(1-\Lambda^4)^n}{(1-\Lambda^4)^{3n+3}},\\
1+(3n+3)\Lambda^4+\sum_{d>4}\chi(M_{aF+bG}^{X}(0,d),\mu(nF+2G))\Lambda^{d}&=
\frac{1}{2}\frac{(1+\Lambda^4)^n+(1-\Lambda^4)^n}{(1-\Lambda^4)^{3n+3}}.
\end{align*}
\item The formulas of (1), (2), (3) above hold for all ample classes $aF+bG$ on $X$ with $=$ replaced by $\equiv$.
\end{enumerate}
\end{Theorem}

\subsection{Results for the projective plane}

Combining \thmref{p11t}
and blowup formulas \lemref{blowsimple}, \lemref{1blow} relating the invariants of a surface and its blowup in a point we get the following formulas for $\P^2$.
\begin{Theorem}\label{P22}
\begin{align*}
\tag{1}1+3\Lambda^4+\sum_{d>4}\chi(M_{H}^{\P^2}(0,d),\mu(H))\Lambda^{d}&=
\frac{1}{(1-\Lambda^4)^{3}},\\
\tag{2}1+6\Lambda^4+\sum_{d>4}\chi(M_{H}^{\P^2}(0,d),\mu(2H))\Lambda^{d}&=
\frac{1}{(1-\Lambda^4)^{6}},\\
\tag{3}\sum_{d}\chi(M_{H}^{\P^2}(H,d),\mu(2H))\Lambda^{d}&=
\frac{\Lambda^3}{(1-\Lambda^4)^{6}},\\
\tag{4}1+10\Lambda^4+\sum_{d>4}\chi(M_{H}^{\P^2}(0,d),\mu(3H))\Lambda^{d}&=
\frac{1+\Lambda^8}{(1-\Lambda^4)^{10}}.
\end{align*}

\end{Theorem}

\subsection{Results on strange duality}\label{insd}
We choose two elements $c,c^*\in K(X)$, such that both moduli spaces $M^X_H(c^*)$  and $M^X_H(c)$ are non-empty and the determinant line bundles $\lambda(c)$ and $\lambda(c^*)$ are well-defined  over $M^X_H(c^*)$  and $M^X_H(c)$, respectively.  Under suitable conditions, (see \secref{strd}),
%\begin{LG} is it sensible to indicate here what the conditions  are or make a reference, e.g. (see ... below)\end{LG}
there is a canonical map
\[ SD_{c,c^*}:H^0(M^X_H(c),\lambda(c^*))^{\vee}\ra H^0(M^X_H(c^*),\lambda(c^*)).\]
The strange duality conjecture says that $SD_{c,c^*}$ should be an isomorphism.
%\begin{LG} under suitable conditions? Whenever it exists?\end{LG} \begin{YY} I think the conjecture itself does not ask for other conditions except the existence of the map and also of course both the moduli spaces are non-empty, 
%but maybe it is better to use "should be " instead of "is".\end{YY}

The strange duality conjecture was first formulated for $X$ a smooth curve in the 1990s (see \cite{Bea} and \cite{Donagi}) and in this case been proved around 2007 (see \cite{Bel1}, \cite{MO1} and \cite{Bel2}).
%\begin{LG} may give citations\end{LG}\begin{YY}Done\end{YY}.  
For $X$ a surface, there does not exist until now a general formulation of the strange duality conjecture.   There is a formulation for some special cases due to Le Potier (see \cite{LPst} or \cite{Da2}). 
%\begin{LG} maybe already here mention K3 surfaces\end{LG}
We will prove the following cases of Le Potier's strange duality conjecture. 

\begin{Theorem}\label{sdmain} Let the polarization $H$ be both $c$-general and $c^*$-general.  Then the strange duality conjecture is true, i.e. the map $SD_{c,c^*}$ is an isomorphism in the following three cases.
\begin{enumerate}
\item $X=\P^2$, $\P^1\times\P^1$ or $\widehat \P^2$.  $c=(2,0,c_2)$ with $c_2>2$ and $c^*=(0,-K_X,\chi=0)$, moreover if $X=\P^1\times\P^1$ or $\widehat \P^2$, we chose the polarization of the form $H=aF+bG$ with $a\geq b$.

\item $X=\P^1\times\P^1$ or $\widehat \P^2$ with $H=aF+bG$ and $\frac ab\geq \frac 54$, $c=(2,0,c_2)$ with $c_2>2$ and $c^*=(0,2G+3F,\chi=0)$.

\item $X=\P^2$ with $H$ the hyperplane class, $c=(2,H,c_2)$ with $c_2>0$ and $c^*=(0,2H,\chi=-1)$.
\end{enumerate}

\end{Theorem}
Actually we will show that essentially the strange duality conjecture holds for any polarization.  But if $H$ is not $c$-general, the formulation needs a slight modification (see Theorem \ref{vsdmain} and Remark \ref{esd}).  

The strange duality for surfaces is a very interesting problem and many other people  worked on it.  For instance in the case  $\P^2$, Danila proves that Le Potier's strange duality holds for $c=(2,0,c_2),c^*=(0,dH,\chi=0)$ with $c_2$ small and $d=1,2,3$ (\cite{Da1} and \cite{Da2});  Abe shows that it holds for all $c=(2,0,c_2),c^*=(0,dH,\chi=0)$ with $d=1,2$ (\cite{Abe});  and the second author shows that it holds for all $c=(1,0,c_2),c^*=(0,dH,0)$ (see Section 4.3 in \cite{Yuan}), and also for all $c=(2,0,c_2=2),c^*=(0,dH,\chi=0)$ (\cite{Yuan2}).  Marian and Oprea, and their collaborators have proven many results on the strange duality for K3 and abelian surfaces (e.g. \cite{MOY}, \cite{MO} and \cite{BMOY}).
However, in general still very few results on this conjecture are known. 

%We will introduce some background material in Section 2.  Section 3 to Section 5 are devoted to computing the generating functions of the K-theoretic Donaldson invariants, using the wallcrossing formula in \cite{GNY} and also some technics in modular forms.  In the last section, Section 6, we prove the strange duality for Case (1)-(3).

\subsection{Acknowledgements}
The first named author wants to thank Don Zagier for many useful discussions and explanations over the course of several years, without which this project could not have succeeded. 
The second-named author was supported by NSFC grant 11301292.

\section{Background Material}\label{sec:background}

%In this whole paper $X$ will be a simply connected nonsingular projective surface over $\C$.
%We will in addition assume that $-K_X$ is ample on $X$. In particular $X$ is a rational surface.
For a class $\alpha\in H^*(X)$, denote $\<\alpha\>:=\int_X\alpha$.
For $\alpha,\beta\in H^2(X)$ we write $\<\alpha,\beta\>:=\int_X\alpha\wedge\beta$ and $\beta^2:=\<\beta,\beta\>$.

Let $H$ be an ample divisor on $X$.  A class $c$ in the Grothendieck group $K(X)$ of coherent sheaves on $X$ is  determined by its rank and the Chern classes $c_1$, $c_2$. Therefore we will also denote $c=(r,c_1,c_2)$ the class of rank $r$ coherent sheaves on $X$ with first and second Chern class $c_1,c_2$.  We also may write $c=(r,c_1,\chi)$ with $\chi$ standing for the holomorphic Euler characteristic.  
%\begin{LG} explain better, introduce $K$-group\end{LG}
 %For $r\ge 0$, $c_1\in Pic(X)$, $c_2\in H^4(X,\Z)$ let
Let $M_H^X(c)=M^X_H(r,c_1,c_2)$ be the moduli space of $H$-semistable sheaves in class $c$,  and 
let $M^X_H(c)^s$ be the open
subset consisting of stable sheaves.  If $r=0$ and $c_1=c_1(L)$ with $L$ nontrivial and effective, then $M^X_H(r,c_1,c_2)$ is a moduli space of 1-dimensional semistable sheaves supported on curves in the linear system $|L|$.   
%pure sheaves in the sense of \cite{??}\begin{NB}give references\end{NB}.
We will write $M^X_H(c_1,d)$ with $d:=4c_2-c_1^2$ instead of $M^X_H(c)$ if $c=(2,c_1,c_2)$.

%For polynomials $f_1,\ldots,f_n\in \C[x_1,\ldots,x_m]$ we denote their zero set
%as $Z(f_1,\ldots,f_n)\subset \C^m$.

\subsection{Determinant line bundles}\label{sec:detbun}
We briefly review the determinant line bundles on the moduli space
\cite{DN},\cite{LP1}, for more details we refer to \cite[Chap.~8]{HL}.

For a Noetherian scheme $Y$ we denote by $K(Y)$ and $K^0(Y)$ the Grothendieck groups of coherent sheaves and locally free sheaves on $Y$ respectively.
%Then $K^0(Y)$ is a commutative ring with $1=[\oo_Y]$, with the multiplication given
%by the tensor product of locally free sheaves. 
If $Y$ is nonsingular and quasiprojective, then $K(Y)=K^0(Y)$.
In particular we have $K(X) = K^0(X)$ for the smooth projective
surface $X$. 
%We will identify $K^0(X)$ with $K(X)$ hereafter.
If we want to distinguish a sheaf $\F$ and its class in $K(Y)$, we
denote the latter by $[\F]$, but we may also write $\F$ for the class in $K(Y).$
For a proper morphism $f\colon Y_1\to Y_2$ we have the pushforward
homomorphism
\(f_!\colon K(Y_1)\to K(Y_2);  [\F] \mapsto\sum_i (-1)^i [R^if_*\F].\)
%When $Y_2 = \mathrm{pt}$, this is the Euler characteristic of $\F$
%under the identification of $K(\mathrm{pt}) \cong\Z$:
%\(  f_!([\F]) = \chi(Y_1,\F) = \sum_i (-1)^i \dim H^i(Y_1,\F).\)
%We also have a pushforward homomorphism $K^0(Y_1)\to K^0(Y_2)$ when
%$f$ is a locally complete intersection morphism (see \cite[\S4.4]{BFM}).
For any morphism $f\colon Y_1\to Y_2$ we have the pullback
homomorphism
\(   f^*\colon K^0(Y_2)\to K^0(Y_1);  [\F] \mapsto[f^*\F] \) for a locally free sheaf $\F$ on $Y_2$.
Let $\E$ be a flat family of coherent sheaves of class $c$ on $X$ parametrized by a scheme $S$, then $\E\in K^0(X\times S)$.
Let $p:X\times S\to S$, $q:X\times S\to X$ be the projections.
Define $\lambda_\E:K(X)\to \Pic(S)$ as the composition of the following homomorphisms:
\begin{equation}\label{dlb}
\xymatrix@C=0.3cm{
  K(X)=K^0(X) \ar[rr]^{~~q^{*}} && K^0(X\times S) \ar[rr]^{.[\E]} && K^0(X\times
S) \ar[rr]^{~~~p_{!}} && K^0(S)\ar[rr]^{det^{-1}} &&
\Pic(S),}\end{equation}
where $q^*$ is the pull-back morphism, $[\F].[\G]:=\sum_i (-1)^i[\Tor_i(\F,\G)]$ with $[\F]$ the class of sheaf $\F$ in $K(X)$, 
and $p_{!}([\F])=\sum_i(-1)^i[R^ip_{*}\F].$  Notice that $p_{!}([\F])\in K^0(S)$ for $\F$ $S$-flat by Proposition 2.1.10 in \cite{HL}.
% [\F]\to \det\big(p_!(q^*([\F])\otimes[\E])\big)^{-1}.$$(see also \cite[(2.1.10), (2.1.11)]{HL}).
%\begin{equation}\label{eq:lambdaE}
%\lambda_\E(u):=\det(p_!((q^*u)\otimes \E)).
%\end{equation}

The following elementary facts are important for working with these line bundles:
\begin{enumerate}
\item $\lambda_\E$ is a homomorphism, i.e. $\lambda_\E(v_1+v_2)=\lambda_\E(v_1)\otimes \lambda_{\E}(v_2)$.
\item If $\mu\in \Pic(S)$ is a line bundle, then $\lambda_{\E\otimes p^*\mu}(v)=
\lambda_{\E}(v)\otimes \mu^{\chi(c\otimes v)}$.
\item $\lambda_\E$ is compatible with base change: if
$\phi:S'\to S$ is a morphism, then $\lambda_{\phi^*\E}(v)=\phi^*\lambda_{\E}(v)$.
\end{enumerate}

However, in general there is no universal sheaf $\E$ over $X\times M^X_H(c)$, and even if it exists, there is ambiguity caused by tensoring with the pull-back of a line bundle on $M^X_H(c)$.  Define $K_c:=c^\perp=\big\{v\in K(X)\bigm|
\chi(v\otimes c)=0\big\}$,~
%\begin{NB} Changed: we did not introduce the notation $(v,c)$ 27.11.L\end{NB}
and $K_{c,H}:=c^\perp\cap\{1,h,h^2\}^{\perp\perp}$, where $h=[\oo_H]$.  Then we have a well-defined morphism $\lambda\colon K_c\to \Pic(M_H^X(c)^s)$,  and $\lambda\colon K_{c,H}\to \Pic(M_H^X(c))$ satisfying the following properties:
 \begin{enumerate}
 \item The $\lambda$ commute with the  inclusions $K_{c,H}\subset K_c$ and $\Pic(M_H^X(c))\subset \Pic(M_H^X(c)^s)$.
 \item If $\E$ is a flat family of semistable sheaves  on $X$ of class $c$ parametrized by $S$,  then  we have
$\phi_{\E}^*(\lambda(v))=\lambda_{\E}(v)$ for all $v\in K_{c,H}$ with $\phi_\E:S\ra M^X_H(c)$ the classifying morphism.
 \item If $\E$ is a flat family of stable sheaves, the same statement to (2) holds with  $K_{c,H}$, $M^X_H(c)$ replaced by  $K_{c}$, $M^X_H(c)^s$.
 \end{enumerate}
 
%On $K(X)$ we have a quadratic form $(u,v)\mapsto \chi(X,u\otimes v)
%\equiv \chi(u\otimes v)$. (We denote $\chi(X,u\otimes v)$ by
%$\chi(u\otimes v)$ for brevity hereafter.)
%We say that $u,v\in K(X)$ are numerically equivalent if $u-v$ is in the radical of
%this quadratic form, and denote $K(X)_{{\rm num}}$ the set of numerical equivalence classes.
%As $X$ is simply connected, we have that $K(X)_{\rm num}=H^*(X,\Z)$.
%\begin{NB}Explain what you  mean by this\end{NB} Therefore we also write write
%$M^X_H(c)$ for $M^X_H(r,c_1,c_2)$ if $c$ is corresponding class in $K(X)_{num}$.
%Let $c\in K(X)_{{\rm num}}$.

Since $X$ is a simply connected surface, both the moduli space $M^X_H(c)$ and the determinant line bundle $\lambda(c^*)$ only depend on the images of $c$ and $c^*$ in $K(X)_{num}.$  Here $K(X)_{num}$ is the Grothendieck group modulo numerical equivalence.  We say that $u,v\in K(X)$ are numerically equivalent if $u-v$ is in the radical of the quadratic form $(u,v)\mapsto \chi(X,u\otimes v)\equiv \chi(u\otimes v)$

%Let $H$ be a very ample divisor on $X$.
%For a class $c\in K(X)_{{\rm num}}$ we denote by $K_c:=c^\perp=\big\{v\in K(X)\bigm|
%\chi(v\otimes c)=0\big\}$.%
%\begin{NB} Changed: we did not introduce the notation $(v,c)$
%27.11.LG\end{NB}
%We denote by $K_{c,H}:=c^\perp\cap\{1,h,h^2\}^{\perp\perp}$, where $h=[\oo_H]$.
%Now let $c\in K(X)_{{\rm num}}$ be the class of  positive rank.
%There are  homomorphisms
%$\lambda\colon K_c\to \Pic(M_H^X(c)_s)$,  and $\lambda\colon K_{c,H}\to \Pic(M_H^X(c))$,
 %such that $\lambda$ commute with the  inclusions $K_{c,H}\subset K_c$ and $\Pic(M_H^X(c))\subset \Pic(M_H^X(c)_s)$.
%We call $H$ {\it general} with respect to
%$c$ if all the strictly semistable sheaves in $M_H^X(c)$
%are strictly semistable with respect to all ample divisors on $X$ in a neighbourhood of $H$ \begin{NB}I think this is not quite correct in case $r=0$, check with Le Potiers paper \end{NB}
% (the ample cone has the topology induced from the
%Euclidean topology on $H^2(X,\R)$). In this case $\lambda\colon K_{c,H}\to \Pic(M_H^X(c))$ can be
%extended to $K_c$.

%If $\E$ is a flat family of semistable sheaves  on $X$ of class $c$ parametrized by $S$,  then  we have
%$\phi_{\E}^*(\lambda(v))=\lambda_{\E}(v)$ for all $v\in K_{c,H}$ for $\phi_\E^*:\Pic(M^X_H(c))\to \Pic(S)$ the pullback  by the classifying morphism.
%If $H$ is general with respect to $c$, the same statement holds
%with $K_{c,H}$ replaced by $K_c$.
%If $\E$ is a flat family of stable sheaves, the same statement holds with  $K_{c,H}$, $M^X_H(c)$ replaced by  $K_{c}$, $M^X_H(c)_s$.
Often $\lambda\colon K_{c,H}\to \Pic(M_H^X(c))$ can be extended.  For instance let $c=(2,c_1,c_2)$, then $\lambda(v(L))$ is well-defined over $M^X_H(c)$ if $\<L,\xi\>=0$ for all $\xi$ a class of type $(c_1,d)$ (see \secref{walls}) with $\<H,\xi\>=0$.  This can be seen easily from the construction of $\lambda(v(L))$ (e.g. see the proof of Theorem 8.1.5 in \cite{HL}), we will also explain more in details in Remark \ref{ppd} in \secref{lpsd}.
%\begin{LG} we have not introduced classes of type $(c_1,d)$ and we have to give some kind of reference for this.
%\end{LG}\begin{YY} I think now it is ok.\end{YY}

\subsection{Walls}\label{walls}
Denote by $\cc$ the ample cone of $X$.
Then $\cc$ has a chamber structure:
For a class $\xi\in H^2(X,\Z)\setminus \{0\}$, let $W^\xi:=\big\{ x\in \cc\bigm| \< x,\xi\>=0\big\}$.
Assume $W^\xi\ne \emptyset $. Let $c_1\in \Pic(X)$, $d\in \Z$ and $d\equiv -c^2_1~(4)$. %\begin{NB}congruent to $-c_1^2$ modulo 4\end{NB}
Then we call $\xi$ a {\it class of type} $(c_1,d)$ and call
$W^\xi$ a {\it wall of type} $(c_1,d)$ if the following conditions hold
\begin{enumerate}
\item
$\xi+c_1$ is divisible by $2$ in $H^2(X,\Z)$,
\item $d+\xi^2\ge 0$.
% \begin{NB} I think it is better to be $(d+\xi^2)/8\in\Z_{\ge 0}$. \end{NB}
%\begin{LG} No, this is not the definition. What you want would imply that there are strictly semistable sheaves, when $\<H,\xi\>=0$, but a wall means that
%that when $H$ passes through a wall the stability changes, which is a different condition. With your definition it is not true that $M_H^X(c_1,d)$ depends only on the chamber of type $(c_1,d)$ of $H$
%
%Your definition is really false, and as you sometimes use it in your part I am a bit worried about it.
%\end{LG}
%\begin{YY} I now understand your meaning. My part also has been fixed.  Don't worry.\end{YY}
\end{enumerate}
We call $\xi$ a {\it class of type} $c_1$, if $\xi+c_1$ is divisible by $2$ in $H^2(X,\Z)$.
We say that $H\in \cc$ lies on the wall $W^\xi$ if $H\in W^\xi$.
The {\it chambers of type} $(c_1,d)$ are the connected components of the complement
of the walls of type $(c_1,d)$ in $\cc$.
Then $M_H^X(c_1,d)$ depends only on the chamber of type $(c_1,d)$ of $H$.

Let $c\in K(X)$ be the class of an sheaf $\F\in M_H^X(c_1,d)$. We call $H$ {\it general} with respect to
$c$ if all the strictly semistable sheaves in $M_H^X(c)$
are strictly semistable with respect to all ample divisors on $X$ in a neighbourhood of $H$.  It is easy to see
that $H$ is general with respect to $c$ if and only if $H$ does not lie on a wall $W_{\xi}$ of
type $(c_1,d)$ such that $(d+\xi^2)/8\in\mathbb{Z}_{\geq0}$.

\subsection{$K$-theoretic Donaldson invariants}\label{backDong}
%We write $M^X_H(c_1,d)$ for $M^X_H(2,c_1,c_2)$ with $d=4c_2-c_1^2$.
%Let $v\in K_c$, where $c$ is the class of a coherent rank $2$ sheaf
%with Chern classes $c_1,c_2$.
Let $L$ be a line bundle on $X$ and assume that $\<c_1(L), c_1\>$ is even.
Then for $c=(2,c_1,c_2)$, we put
\begin{equation}\label{eq:uL} v(L):=(1-L^{-1})+\<\frac{c_1(L)}{2},(c_1(L)+K_X-c_1)\>[\oo_x]\in K_c.\end{equation}
%\begin{LG} this formula is correct \end{LG}
Note that $v(L)$ is independent of $c_2$.
Assume that $H$ is general with respect to $c$. Then we denote
$\mu(L):=
\lambda(v(L))\in \Pic(M^X_H(c_1,d))$.
The {\it $K$-theoretic Donaldson invariant\/} of $X$, with respect to $L,c_1,d,H$ is
$\chi(M^X_H(c_1,d),\mu(L))$.

We recall the following blowup relation for the  $K$-theoretic Donaldson invariants from  \cite[Cor.~1.8]{GNY}.
Let $(X,H)$ be a polarized rational surface. Let $\widehat X$ be the
blowup of $X$ in a point and $E$ the exceptional divisor.
In the
following we always denote a class in $H^*(X,\Z)$ and its pullback by
the same letter.
Let $Q$ be an open subset of a suitable quot-scheme
such that $M^X_H(c_1,d)=Q/GL(N)$. Assume that $Q$ is smooth \textup(e.g.\ $\langle -K_X, H\rangle>0$\textup).
We choose $\epsilon>0$ sufficiently small so that  $H-\epsilon E$ is ample on $\widehat X$ and  there is no class $\xi$ of type $(c_1,d)$ or of type $(c_1+E,d+1)$ on $\widehat X$ with $\<\xi, H\><0<\<\xi, (H-\epsilon E)\>.$
In case $c_1=0$ assume $d>4$.

\begin{Lemma}  \label{blowsimple}
We have
\begin{align*}
\chi({M}^{\widehat X}_{H-\epsilon E}(c_1,d),\mu(L))&
=\chi(M_H(c_1,d),\mu(L)), \\
 \chi({M}^{\widehat X}_{H-\epsilon E}(c_1+E,d+1),\mu(L))&
=\chi(M_H(c_1,d),\mu(L))
\end{align*}
for any line bundle  $L$ on $X$ such that $\<L, c_1\>$ is even and
$\<L, \xi\>=0$ for $\xi$ any class of
type $(c_1,d)$ on $\widehat X$ with $\<H, \xi\>=0$.
\end{Lemma}

\begin{Remark}\label{rem:canonical}
If $H$ is a general polarization, then
$\mu(2K_X)$ is a line bundle on $M^X_H(c)$
which coincides with the dualizing sheaf
on the locus of stable sheaves $M_H^X(c)^s$.
If $\dim (M_H^X(c) \setminus M_H^X(c)^s) \leq \dim M_H^X(c)-2$,
then $\omega_{M_H^X(c)}=\mu(2K_X)$.
\end{Remark}

We introduce the generating function of the $K$-theoretic Donaldson invariants.
\begin{Definition} \label{KdonGen} Let $c_1\in H^2(X,\Z)$. Let $H$ be ample on $X$ not on a wall of type $(c_1)$.
\begin{enumerate}
\item
If $c_1\not \in 2 H^2(X,\Z)$, let
\begin{equation}\label{eq:Kdon}
\begin{split}
\chi_{c_1}^{X,H}(L)&:=\sum_{d>0}
\chi(M^X_H(c_1,d),\mu(L))\Lambda^d.
\end{split}
\end{equation}
\item In case $c_1=0$ let
$\widehat X$ be the blowup of $X$ in a point. Let $E$ be the exceptional divisor. Let $\epsilon>0$ be sufficiently small so that there is no class
$\xi$ of type $(E,d+1)$ on $\widehat X$ with  $\<\xi , H\> <0 <\<\xi , (H-\epsilon E)\>$.
We put
\begin{equation}\label{eq:Kdon0}
\begin{split}
\chi_{0}^{X,H}(L)&:=\sum_{d>4}
\chi(M^X_H(0,d),\mu(L))\Lambda^d+\Big(\chi(M^{\widehat X}_{H-\epsilon E}(E,5) ,\mu(L))+\<L, K_X\>-\frac{K_X^2+L^2}{2}-1\Big)\Lambda^4.
\end{split}
\end{equation}
\end{enumerate}
\end{Definition}

\begin{Remark}
\begin{enumerate}
\item
Note that with this definition we have $\Coeff_{\Lambda^d}\big[\chi^{X,H}_0(L)\big]=\chi(M^X_H(0,d),\mu(L))$ only for $d>4$.
\item  The coefficient of  $\Lambda^4$ of $\chi^{X,H}_{0}(L)$ has been chosen to make the generating functions
$\chi^{X,H}_{c_1}(L)$  more compatible among each other. In particular it will lead to a cleaner blowup formula for $\chi^{X,H}_{0}(L)$.
\end{enumerate}
\end{Remark}

\subsection{Vanishing of higher cohomology}
In this paper we will compute the holomorphic Euler characteristics $\chi(M_H^X(c_1,d),\mu(L))$ on rational surfaces. We then want to apply this to prove cases of Le Potier's strange duality, which is a statement about spaces of sections $H^0(M_H^X(c_1,d),\mu(L))$.
Thus  we need to see  that in the cases considered  $\chi(M_H^X(c_1,d),\mu(L))=\dim H^0(M_H^X(c_1,d),\mu(L))$.
We will show this using arguments closely related to \cite[\S1.4]{GNY}
Let  $L$ be a numerically effective  line bundle on $X$.

\begin{Proposition}\label{vanihh}
Fix $c_1,d$. Let $H$ be an ample line bundle on $X$ which is general with respect to $(c_1,d)$.
If $c_1$ is not divisible by $2$ in $H^2(X,\Z)$ or $d>8$, we have
$H^i(M_H^X(c_1,d),\mu(L))=0$ for all $i>0$, in particular
$$\dim H^0(M_H^X(c_1,d),\mu(L))=\chi(M_H^X(c_1,d),\mu(L)).$$
\end{Proposition}

\begin{proof}
As $-K_X$ is ample on $X$, and $L$ is numerically effective, we get that $L-2K_X$ is ample on $X$.
By \cite[Prop.8.3.2.]{HL} there exists a positive integer $n$, such that  $\mu(L-2K_X)^{\otimes n}$ is globally generated on $M_H^X(c_1,d)$.
Denote by  $\omega_M$ the dualizing sheaf of $M_H^X(c_1,d)$.
As $-K_X$ is ample we have by \cite{Bou} that  $M_H^X(c_1,d)$ is normal and has only rational singularities.
Therefore  \cite[Cor.7.70]{SS} gives that $H^i(M_H^X(c_1,d),\mu(L-2K_X)\otimes \omega_M)=0$ for $i>0$.
If $c_1$ is not divisible by $2$ in $H^2(X,\Z)$, then, by our assumption that $H$ is general, the moduli space $M_H^X(c_1,d)$ consists only of
stable sheaves. If $c_1=0$, again using that $H$ is general, we see that the strictly semistable points of $M_H^X(c_1,d)$ are of the form
$\I_{Z}(c_1/2)\oplus \I_{W}(c_1/2)$ for $0$-dimensional subschemes $Z$, $W$  of $X$ of length $d/8$.
In particular if $d$ is not divisible by $8$, $M_H^X(c_1,d)$ consists only of
stable sheaves, and if $d$ is divisible by $8$, the dimension of the locus $M^{sss}$ of strictly semistable sheaves is $d/2$. On the other hand $M_H^X(c_1,d)$ has pure dimension $d-3$.
Thus if $d>8$ we get $M^{sss}$ has codimension at least $2$ in $M_H^X(c_1,d)$.
In all these cases \remref{rem:canonical} says that $\omega_M=\mu(2K_X)$.
Thus by the above $H^i(M_H^X(c_1,d),\mu(L-2K_X)\otimes \omega_M)=H^i(M_H^X(c_1,d),\mu(L))=0$ for $i>0$.
\end{proof}

\subsection{Strange duality}\label{strd}
We briefly review the strange duality conjecture from \cite{LPst}.
Let $c,c^*\in K(X)_{num}$ with $c\in K_{c^*}$.
Let $H$ be ample line bundle on $X$ which is
both $c$-general and $c^*$-general.
%We denote $M(c):=M^H_X(c)$, $M(c^*):=M^H_X(c^*)$.
Write $\D_{c,c^*}:=\lambda(c^*)\in \Pic(M^X_H(c))$,
$\D_{c^*,c}:=\lambda(c)\in \Pic(M^X_H(c^*))$.
Assume that  all $H$-semistable sheaves $\F$ on $X$ of class $c$ and
all $H$-semistable sheaves $\G$ on $X$ of class $c^*$ satisfy
\begin{enumerate}
\item $\Tor_i(\F,\G)=0$ for all $i\ge 1$,
\item $H^2(X,\F\otimes \G)=0$.
\end{enumerate}
Both conditions are automatically satisfied if
$c$ is not of dimension $0$ and  $c^*$ is of dimension $1$
(see \cite[p.9]{LPst}).  If $c=(2,c_1,c_2)$ and $c^*=(0,L,\chi=-\<\frac{c_1(L)}2\cdot c_1\>)$, then $\D_{c,c^*}=\mu(L)$.

Put $\D:=\D_{c,c^*}\boxtimes \D_{c^*,c}\in \Pic(M^X_H(c)\times M^X_H(c^*))$.
In \cite[Prop.~9]{LPst} a canonical section $\sigma_{c,c^*}$ of $\D$ is constructed, whose zero set is supported on
$$\mathscr{D}:=\big\{([\F],[\G])\in   M^X_H(c)\times M^X_H(c^*)\bigm| H^0(X,\F\otimes \G)\ne 0\big\}.$$
The  element $\sigma_{c,c^*}$ of $H^0(M^X_H(c),\D_{c,c^*})\otimes H^0(M^X_H(c^*),\D_{c^*,c})$, gives  a linear map
\begin{equation}
\label{SDmap}
SD_{c,c^*}:H^0(M^X_H(c),\D_{c,c^*})^\vee \to H^0(M^X_H(c^*),\D_{c^*,c}),
\end{equation}
 called the {\it strange duality map}.
Le Potier's strange duality is then the following.

\begin{ConjQues}\label{sdcon} Is $SD_{c,c^*}$ an isomorphism?
\end{ConjQues}

\section{Wallcrossing formula}
\subsection{Theta functions and modular forms} For $\tau\in \H=\big\{\tau\in \C\bigm| \Im(\tau)>0\big\}$ put $q=e^{\pi i\tau/4}$ and for  $h\in \C$ put $y=e^{h/2}$. Note that the notation is not standard.
Recall the $4$ theta functions:
\begin{equation}
\begin{split}\label{theta}
\theta_1(h)&:=\sum_{n\in \Z} i^{2n-1} q^{(2n+1)^2} y^{2n+1},\qquad
\theta_2(h):=\sum_{n\in \Z} q^{(2n+1)^2} y^{2n+1},\\
\theta_3(h)&:=\sum_{n\in \Z} q^{(2n)^2} y^{2n},\qquad
\theta_4(h):=\sum_{n\in \Z} i^{2n}q^{(2n)^2} y^{2n}.
\end{split}
\end{equation}
We usually do not denote the argument $\tau$. If necessary we write
$\theta_i(\tau|h)$. The conventions are essentially the same as in
\cite{WW} and in \cite{Ak}, where the $\theta_i$ for $i\le 3$ are denoted $\vartheta_i$ and $\theta_4$ is however denoted $\vartheta_0$.
Denote
\begin{equation}\label{thetatilde}\theta_i:=\theta_i(0), \quad
\widetilde\theta_i(h):=\frac{\theta_i(h)}{\theta_i}, \quad i=2,3,4;\qquad \widetilde\theta_1(h):=\frac{\theta_1(h)}{\theta_4}
\end{equation}
the corresponding Nullwerte and the normalized theta functions.
The theta functions satisfy the Jacobi identity
%\begin{NB}Check whether it is really called like that\end{NB}
\begin{equation}
\label{Jacobi} \theta_3^4=\theta_2^4+\theta_4^4,
\end{equation} as well as the quadratic relations
\begin{equation}\label{quadtheta}\left(\begin{matrix}
0&\th_2^2&-\th_3^2&\th_4^2\\
-\th_2^2&0&-\th_4^2&\th_3^2\\
-\th_3^2&-\th_4^2&0&\th_2^2\\
\th_4^2&\th_3^2&-\th_2^2&0
\end{matrix}\right)\left(\begin{matrix}\th_1(h)^2\\ \th_2(h)^2\\ \th_3(h)^2\\ \th_4(h)^2\end{matrix}\right)=\left(\begin{matrix}0\\ 0\\ 0\\ 0\end{matrix}\right).\end{equation}
We define a modular function
$$u=-\frac{\th_2^2}{\th_3^2}-\frac{\th_3^2}{\th_2^2}
=\frac{1}{4q^2}(-1- 20q^4 + 62q^8 - 216q^{12} + 641q^{16}+\ldots),$$
  and two Jacobi functions
(Jacobi forms of weight and index $0$) by
\begin{align*}
\Lambda&:=\frac{\theta_1(h)}{\theta_4(h)}=-i (y-y^{-1})q-i(y^3-y+y^{-1}-y^{-3})q^5+\ldots,\\
M&:=2\frac{\widetilde \theta_2(h)\widetilde \theta_3(h)}{\widetilde \theta_4(h)^2}=(y+y^{-1})+3(y^3 - y - y^{-1} + y^{-3})q^4 + \ldots.
\end{align*}
The quadratic relations above imply
$$\frac{M^2}{4}=\left(\frac{\th_4^2\th_2(h)^2}{\th_2^2\th_4(h)^2}\right)\left( \frac{\th_4^2\th_3(h)^2}{\th_3^2\th_4(h)^2}\right)=\left(1-\frac{\th_3^2}{\th_2^2}\Lambda^2\right)
\left(1-\frac{\th_2^2}{\th_2^2}\Lambda^2\right)=(1+u\Lambda^2+\Lambda^4),$$
and comparing the coefficients of $q^0$ gives that
 $M=2\sqrt{1+u\Lambda^2+\Lambda^4}$.
 We also have the relation
 \begin{equation}\label{dLdh}\frac{\partial\Lambda}{\partial h}=\frac{\theta_2\theta_3}{4i}M,
\end{equation}
which follows from \cite[\S26]{Ak}, and which 
%\begin{NB}Give proof or reference. I think I should be able to find the suitable formula in Azhiezer.\end{NB}
 is equivalent to the formula
 \begin{equation}
 \label{hint}
 \begin{split}
 h&=\frac{2i}{\theta_2\theta_3}\int_{0}^\Lambda\frac{dx}{\sqrt{1+ux^2+x^4}}\\
 &= i(q^{-1} - 2q^3 + 3q^7+\ldots)\Lambda+
  i(\frac{1}{24}q^{-3} + \frac{3}{4} q - \frac{33}{8}q^5 + \ldots)\Lambda^3+\ldots.
  \end{split}
  \end{equation}
  We have the power series developments
\begin{equation}
\label{hhh}
\begin{split}
\frac{1}{\sqrt{1+u\Lambda^2+\Lambda^4}}&=\sum_{\substack{n\ge 0\\ n\ge k\ge0}}
\binom{-\frac{1}{2}}{n}\binom{n}{k}u^k\Lambda^{4n-2k}, \\
h&=\frac{2i}{\theta_{2}\theta_{3}}.
\sum_{\substack{n\ge 0\\ n\ge k\ge0}}
\binom{-\frac{1}{2}}{n}\binom{n}{k}\frac{u^k\Lambda^{4n-2k+1}}{4n-2k+1}.
\end{split}
\end{equation}

 \begin{Notation}\label{R}
\begin{enumerate}
\item We denote $
 \RR:=\Q[[q^2\Lambda^2,q^4]]$.
 \item Let $\Q[t_1,\ldots,t_k]_n$ be the set of polynomials in $t_1,\ldots,t_k$ of degree $n$ and $\Q[t_1,\ldots,t_k]_{\le n}$ the set of polynomials of degree at most $n$.
 \end{enumerate}
 \end{Notation}

  The formula \eqref{hint} together with the definition of $u$ show that $h\in q^{-1}\Lambda\RR.$
    A function $F(\tau,h)$ can via formula \eqref{hint} also be viewed as a function
 of $\tau$ and $\Lambda$.
  In this case, viewing $\tau$ and $\Lambda$ as the independent variables we define
 $$F' :=\frac{4}{\pi i} \frac{\partial F}{\partial \tau}=q\frac{\partial F}{\partial q},
 \quad F^*:=\Lambda\frac{\partial F}{\partial \Lambda}.$$
 Thus  \eqref{dLdh} and a simple calculation give that 
\begin{equation}
\label{hstar}
h^*=\frac{4i\Lambda}{\theta_2\theta_3 M}, \quad
u'=\frac{2\theta_4^8}{\theta_2^2\theta_3^2}.
\end{equation}
%\begin{NB}Do I need to give an argument for $u'$?\end{NB}
  \begin{Remark}
 The natural set of variables for working with elliptic functions  is $(\tau,h)$. We will see in a moment that the wallcrossing for the $K$-theoretic Donaldson invariants is given by a formula in modular forms and elliptic functions, expressed in terms of
 $\tau$ and $\Lambda$. In order to prove properties of the wallcrossing formula we usually have to work with the natural variables $(\tau,h)$ and then translate the result back into the variables $(\tau,\Lambda)$. Thus the interplay between the two sets of variables $(\tau,h)$ and $(\tau, \Lambda)$ is an important theme in this work.
\end{Remark}

\subsection{Wallcrossing formula}\label{wallcro}
Now we review the wallcrossing formula from \cite{GNY}.
Let $\sigma(X)$ be the signature of $X$.
Fix $c_1\in H^2(X,\Z)$. Let $L\in \Pic(X)$ with $\<c_1(L) ,c_1\>$ even.
Let $\xi\in H^2(X,\Z)$ with $\xi^2\le 0$ and $\xi-c_1\in 2H^2(X,\Z)$.

\begin{Definition}\label{wallcrossterm}
Let $$\Delta_\xi^X(L):=2 i^{\<\xi K_X\>} \Lambda^2 q^{-\xi^2}
y^{\<\xi(L-K_X)\>}\widetilde\theta_4(h)^{(L-K_X)^2}\theta_4^{\sigma(X)}u'h^*.$$
 By the results of the previous section it can be developed as a power series
$$\Delta_\xi^X(L)=\sum_{d\ge 0} f_d(\tau)\Lambda^d\in \Q((q))[[\Lambda]],$$
whose coefficients $f_d(\tau)$ are Laurent series in $q$.
The {\it  wallcrossing term} is
$$\delta_{\xi}^X(L):=\sum_{d\ge 0} \delta_{\xi,d}^X(L)\Lambda^d\in \C[[\Lambda]], $$
with
$$\delta_{\xi,d}^X(L)=\Coeff_{q^0}[f_d(\tau)].$$
\end{Definition}

{\bf Setup:} For the rest of section \ref{wallcro}
let $H_-$, $H_+$ be ample divisors on $X$, which do not lie on a wall
of type $(c_1,d)$. Let $B_+$ be the set of classes $\xi$ of type $(c_1,d)$
with $\<\xi, H_+\>>0 >\<\xi, H_-\>$.

The main result of \cite{GNY} is the following.

\begin{Theorem}\label{wallcr}
\begin{align*}
\chi(M^X_{H_+}(c_1,d),\mu(L))-\chi(M^X_{H_-}(c_1,d),\mu(L))&=\sum_{\xi\in B_+}\delta^X_{\xi,d}(L).
\end{align*}
\end{Theorem}
This is a combination of Prop.~2.11, Cor.~4.2 and Thm.~4.3 in \cite{GNY} together with the results of Section 4.4 in \cite{GNY}.
In \cite[Cor.~4.2]{GNY} one has to take $v=-v(L)$, thus
$\<\xi,(K_X+c_1(v)+\frac{\rk(v)}{2}(c_1-K_X)\>=\<\xi,(K_X-L)\>$ and
$(K_X+c_1(v)+\frac{\rk(v)}{2}(c_1-K_X))^2=(L-K_X)^2$. The results of \cite{GNY} apply to what are called there {\it good} walls. However our assumption that
$-K_X$ is ample implies that all walls on $X$ are good.

Note that, as $X$ is a simply connected surface with $p_g=0$,  the Euler number $e(X)$ and the signature $\sigma(X)$ are related by
$e(X)+\sigma(X)=4$, thus in \cite[Cor.~4.2]{GNY} one has
$$\exp(e(X) A+\sigma(X) B)=\frac{4\theta_4^{\sigma(X)}}{\theta_2^2\theta_3^3}.$$
Note that the $u$ of \cite{GNY} corresponds to $u\Lambda^2$ in the current paper, and the function $U_1$ of \cite{GNY} is denoted by $M$ here. Furthermore
   $d$ in  \cite{GNY} corresponds to  $d-3$ here.
Furthermore by definition $\theta_1(h)$ is equal to $-\theta_{11}(\frac{h}{2\pi i},\tau)$ in \cite{GNY}. We put $\bbeta=1$ in the results of \cite{GNY}.
Thus  Thm.~4.3 and the results of Section 4.4. of  \cite{GNY} give that
Theorem \ref{wallcr} is true if we replace $\Delta_\xi^X(L)$ by
$$i^{\<\xi K_X\>} \Lambda^3 q^{-\xi^2}
e^{\<\xi(L-K_X)\>h/2}\Big(\frac{\theta_1(h)}{\Lambda\theta_4}\Big)^{(L-K_X)^2}
\frac{16 i\theta_4^{\sigma(X)+8}}{\theta_2^3\theta_3^3M}.$$
Note that by definition
$e^{h/2}=y$, $\frac{\theta_1(h)}{\Lambda\theta_4}=\widetilde \theta_4(h)$, and finally we have
by \eqref{hstar} that
$u'h^*=\frac{8i\Lambda\theta_4^{8}}{\theta_2^3\theta_3^3M}.$
Thus the result follows.

\begin{Remark}  \label{chi04}\thmref{wallcr} also applies to $\Coeff_{\Lambda^4}\big[\chi^{X,H}_{0}(L)\big]$ and thus to all the generating functions
$\chi^{X,H}_{c_1}(L)$.
\begin{enumerate}
\item
Let $H_1,H_2$ be ample on $X$, assume they do not lie on a wall of type $(0,4)$.
Then
$$\Coeff_{\Lambda^4}\big[\chi^{X,H_1}_{0}(L)-\chi^{X,H_2}_{0}(L)\big]=\sum_{\xi}\delta^X_{\xi,4}(L).$$
where $\xi$ runs through all classes of type $(0,4)$ with $\<\xi, H_1\> >0>\<\xi, H_2\>$.

\item
Let $c_1\in H^2(X,\Z)$. Let $H_1,H_2$ be ample on $X$, assume they do not lie on a wall of type $(c_1)$. Then
$$\chi^{X,H_1}_{c_1}(L)-\chi^{X,H_2}_{c_1}(L)=\sum_{\xi}\delta^X_\xi(L),$$
where $\xi$ runs through all classes in $c_1+2H^2(X,\Z)$ with $\<\xi, H_1\> >0>\<\xi ,H_2\>$.
\end{enumerate}
\end{Remark}
\begin{proof} (2) follows immediately from \thmref{wallcr} and (1).

(1)
Let $\epsilon>0$ sufficiently small such that there is no class $\xi$ of type $(E,5)$ on $\widehat X$ with $\<\xi,H_1\><0<\<\xi,H_1-\epsilon E\>$ or with 
 $\<\xi,H_2\><0<\<\xi,H_2-\epsilon E\>$. Then by \lemref{blowsimple} we have  for $i=1,2$ that
 $\chi(M^{X}_{H_{i}}(0,4) ,\mu(L))=\chi(M^{\widehat X}_{H_{i}-\epsilon E}(E,5) ,\mu(L))$. Thus by \defref{KdonGen} we need to show
$$\chi(M^{\widehat X}_{H_{1}-\epsilon E}(E,5) ,\mu(L))-\chi(M^{\widehat X}_{H_{2}-\epsilon E}(E,5) ,\mu(L))=\sum_{\xi} \delta^{X}_{\xi,4}(L),$$
 where $\xi$ runs through the classes of type
$(0,4)$ on $X$ with $\<\xi, H_1 \>>0>\<\xi, H_2 \>$.
We have
\begin{equation}
\label{blblb}
\chi(M^{\widehat X}_{H_1-\epsilon E}(E,5) ,\mu(L))-\chi(M^{\widehat X}_{H_2-\epsilon E}(E,5) ,\mu(L))=\sum_{\xi'} \delta^{\widehat X}_{\xi',5}(L),\end{equation}
where $\xi'$ runs through the classes of type
$(E,5)$ on $\widehat X$ with $\<\xi' ,H_{1}-\epsilon E \>>0>\<\xi' ,H_{2} -\epsilon E\>$. Note that $H^2(\widehat X,\Z)=H^2( X,\Z)\oplus \Z E$.  Thus we get that these classes are of the form
$\xi'=\xi+(2n-1)E$ with $n\in \Z$, where $\xi$ is a class of type $(0,4)$ on $X$  with $\xi'^2=\xi^2-(2n-1)^2\ge -5$.
Note that by definition, if $\xi'^2<-5$, we get $\delta^{\widehat X}_{\xi',5}(L)=0$, thus we can replace the sum in
\eqref{blblb} by the sum over all $\xi'=\xi+(2n-1)E$ with $n\in \Z$.
Finally we note that
$$\sum_{n\in \Z}\delta^{\widehat X}_{\xi+(2n-1)E,5}(L)=\Coeff_{q^0\Lambda^5}\left[\sum_{n\in \Z}\Delta^{\widehat X}_{\xi+(2n-1)E}(L)\right],$$
and by \defref{wallcrossterm} we have
\begin{align*}
\sum_{n\in \Z}\Delta^{\widehat X}_{\xi+(2n-1)E}(L)&=\sum_{n\in \Z} i^{-(2n+1)}q^{(2n+1)^2}y^{2n+1}\WT_4(h)^{-1}\theta_4\Delta^X_\xi(L)\\
&=\frac{\theta_1(h)}{\theta_4(h)}\Delta^X_\xi(L)=\Lambda\Delta^X_\xi(L).
\end{align*}
Thus
$\sum_{n\in \Z}\delta^{\widehat X}_{\xi+(2n-1)E,5}(L)=\delta^X_{\xi,4}(L).$
\end{proof}

\begin{Remark}\label{delb}
\begin{enumerate}
\item $\delta_{\xi,d}^X(L)=0$ unless $d\equiv -\xi^2 \mod 4$ (equivalently
$d\equiv -c_1^2$ mod $4$).
\item
In the definition of $\delta_\xi^X(L)$ we can replace
$\Delta_{\xi}^X(L)$ by
\begin{equation}
\begin{split}\label{Delbar}
&\overline \Delta_{\xi}^X(L):= \frac{1}{2}(\Delta_{\xi}^X(L)-\Delta_{-\xi}^X(L))\\
&\ =i^{\<\xi, K_X\>} \Lambda^2 q^{-\xi^2}
\big(y^{\<\xi,L-K_X\>}-(-1)^{\xi^2}y^{-\<\xi,L-K_X\>}\big)\widetilde\theta_4(h)^{(L-K_X)^2}\theta_4^{\sigma(X)}u'h^*.
\end{split}
\end{equation}
\end{enumerate}
\end{Remark}
\begin{proof}
(1) As $h\in \Q[[q^{-1}\Lambda,q^4]]$, we also have $h^*, y,\widetilde \theta_4(h)\in \Q[[q^{-1}\Lambda,q^4]]$. Finally $u,u'\in q^{-2}\Q[[q^4]]$.
It follows that $\Delta_{\xi}^X(L)\in q^{-\xi^2}\Q[[q^{-1}\Lambda,q^4]]
$. Writing $\Delta_{\xi}^X(L)=\sum_{d} f_{d,r}(\tau)\Lambda^d$, we see that $\Coeff_{q^0}[f_{d,r}(\tau)]=0$ unless $d\equiv -\xi^2 \mod 4$.

(2) Note that $\widetilde \theta_4$ is even in $\Lambda$ and $h^*$ is odd in
$\Lambda$, thus $$\overline \Delta_{\xi}^X(L)= \sum_{d\equiv -\xi^2 \mod 2 } f_{d,r}(\tau)\Lambda^d,$$ and the claim follows by (1).
\end{proof}

\subsection{Polynomiality and vanishing of the wallcrossing}
By definition the wallcrossing terms
$\delta_\xi^X(L)$ are power series in $\Lambda$. We now show that
they are always polynomials. This has been shown  already in \cite[Rem.~2.9]{GNY} using a geometric
definition of $\delta_\xi^X(L)$. Here we will give a proof which only uses
elementary properties of theta functions. The arguments used here will play an important role in the rest of the paper.

The wallcrossing formula is expressed in terms of an expression in
Jacobi theta functions $\theta_i(h)$ and modular forms: we develop this expression as a Laurent series in $q$ and $\Lambda$ and take the coefficient of $q^0$.
Note however that the natural variables for this expression would be $\tau$ (or $q$) and the elliptic variable $h$. Thus for understanding the wallcrossing formula it is important to understand the interplay between the two sets of variables $(\tau,h)$ and $(q,\Lambda)$.

We have seen above that $h\in q^{-1}\Lambda\Q[[q^{-2}\Lambda^2,q^4]]$, and thus $y=e^{h/2}\in \Q[[q^{-1}\Lambda,q^4]]$.
We write
\begin{equation}\label{sinhh}
\zeta:=y-y^{-1}=2\sinh(h/2).
\end{equation}
We want to see that as a function of $q,\Lambda$ the function $\zeta$ has only a pole of order $1$  in $q$.
It will follow  that many of the functions we will encounter are almost regular in $q$ in the sense that as Laurent series in $\Lambda,q$, they have  only finitely many monomials
with non-strictly positive powers in $q$ whose coefficients do not vanish.

%\begin{Notation}
%For a variable $t$ and a nonnegative integer we denote $\Q(t)_n$ the set of all polynomials in $t$ of degree $n$ with $\Q$ coefficients.
%\end{Notation}
\begin{Lemma}\label{Rregular}
\begin{enumerate}
\item $\zeta=y-y^{-1}\in q^{-1}\Lambda\RR$.
\item $\zeta^{-1}\in q\Lambda^{-1}\RR$.
\item For all integers $n$ we have
\begin{align*}
\sinh((2n+1)h/2)&\in \Q[q^{-1}\Lambda]_{\le |2n+1|}\RR,\\
\cosh(nh)&\in  \Q[q^{-2}\Lambda^2]_{\le |n|} \RR,\\
\sinh(nh)h^*&\in \Q[q^{-2}\Lambda^2]_{\le |n|}\RR\\
\cosh((2n+1)h/2)h^*&\in \Q[q^{-1}\Lambda]_{\le |2n+1|} \RR.
\end{align*}
\item $\widetilde \theta_4(h)\in \RR
$, and we have $\widetilde \theta_4(h)=1+q^2\Lambda^2+O(q^4)$.
\end{enumerate}
\end{Lemma}
\begin{proof}
(1) By \eqref{theta} we see immediately that
$$\Lambda=\frac{\theta_1(h)}{\theta_4(h)}=\sum_{n\ge 0} q^{4n+1} (g_n(y)-g_n(y^{-1}))=\sum_{n\ge 0} q^{4n+1} f_n(\zeta)$$ for $g_n(y)$, and thus  $f_n(\zeta)$ suitable odd polynomials of degree $2n+1$, in other words
$\Lambda \in q\zeta \Q[[\zeta^2q^4,q^4]].$
Explicitely
$$\Lambda=-i(\zeta q+(\zeta^3+2\zeta)q^5+(\zeta^5+3\zeta^3+\zeta)q^9+(\zeta^7+5\zeta^5+7\zeta^3+2\zeta)q^{13}+\ldots) .$$
Thus we can form the inverse power series
$\zeta\in q^{-1}\Lambda\RR$, i.e.
$$\zeta=i((q^{-1} - 2q^3 + 3q^7+\ldots)\Lambda + (q - 5q^5 + \ldots)
\L^3 +(2 q^3 - 17q^7 +\ldots)\L^5 +\ldots).$$

(2) The coefficient $l_1$ of $\Lambda^1$ of $\zeta$ is a Laurent series in $q$ starting with $iq^{-1}$, and $\frac{\zeta}{f_1}\in 1+q^2\Lambda^2\RR$.
Thus we have $\frac{1}{\zeta}=\frac{1}{f_1}\frac{f_1}\zeta\in q\Lambda^{-1}\RR$.

(3)
Let $n\in \Z$.  From the definition $\zeta=\sinh(h/2)=y-y^{-1}$, $\sinh((2n+1)h/2)=y^{2n+1}-y^{-(2n+1)}$, $\cosh(nh)=y^n+y^{-n}$, we see immediately that $\sinh((2n+1)h/2)$ is an odd polynomial of degree $|2n+1|$ in $\zeta$ and $\cosh(nh)$ is even  of degree  $|2n|$ in $\zeta$. Thus
$\sinh((2n+1)h/2)\in q^{-1}\Lambda\Q[q^{-2}\Lambda^2]_{\le |n|}\RR,$
$\cosh(nh)\in \Q[q^{-2}\Lambda^2]_{\le |n|} \RR.$
Finally note that  $\sinh(n h)h^*=\frac{1}{n}\cosh(nh)$ for $n\ne 0$ and $\sinh(0 h/2)h^*=0$ and $\cosh((2n+1)h/2)h^*=\frac{2}{(2n+1)}\sinh((2n+1)h/2)$. 
.

(4) From the formula
$$\theta_4(h)=1+2\sum_{n\in \Z_{>0}} (-1)^n q^{4n^2} \cosh(nh),$$
it therefore follows that
$\widetilde \theta_4(h)\in  \RR
$, and we also easily see $\WT_4(h)=1+q^2\Lambda^2+O(q^4)$.
\end{proof}

\begin{Theorem}\label{vanwall}
\begin{enumerate}
\item
$\delta_{\xi,d}^X(L)=0$ unless $-\xi^2\le d\le \xi^2+2|\<\xi,L-K_X\>|+4$.
In particular $\delta_{\xi}^X(L)\in \Q[\Lambda]$.
\item $\delta_\xi^X(L)=0$ unless $-\xi^2\le |\<\xi,L-K_X\>|+2$. (Recall that
by definition $\xi^2<0$).
\end{enumerate}
\end{Theorem}

\begin{proof}
Let $N:=\<\xi,L-K_X\>.$ Note that by the condition that $\<L,\xi\>$ is even, we
have $(-1)^{\xi^2}=(-1)^N$. Assume first that $N$ is even.
Then by (\ref{Delbar}) we have
$$\overline \Delta_\xi^X(L)=q^{-\xi^2}i^{\<\xi,K_X\>}\Lambda^2\sinh(Nh/2)h^*
\widetilde \theta_4^{(L-K_X)^2}\theta_4^{\sigma(X)} u'.$$
If $N=0$, then $\Delta_\xi^X(L)=0$.
Now let $N\ne 0$. Then we note that by \lemref{Rregular} we have
$\sinh(Nh/2)h^*\in q^{-|N|}\Lambda^{|N|}\RR.$
Note that
$\Lambda^2u'\in q^{-2}\Lambda^2\RR.$
Putting this together we get
$$\Delta_\xi^X(L)\in q^{-\xi^2}q^{-|N|-2}\Lambda^{|N|+2}\RR.$$
In case $N$ is odd, we have $$\overline \Delta_\xi^X(L)=q^{-\xi^2}i^{\<\xi,K_X\>}\Lambda^2\cosh(Nh/2)h^*
\widetilde \theta_4^{(L-K_X)^2}\theta_4^{\sigma(X)} u'.$$
Again using
$\cosh(Nh/2)h^*\in q^{-|N|}\Lambda^{|N|}\RR,$
a similar argument shows that $\Delta_\xi^X(L)\in q^{-\xi^2}q^{-|N|-2}\Lambda^{|N|+2}\RR.$
Thus
$\delta_{\xi,d}^X(L)= 0$ unless $-\xi^2-\min(d,2|N|+4-d)\le 0$, i.e.
unless
$-\xi^2\le d\le \xi^2+2|N|+4$.
In particular $\delta_{\xi}^X(L)=0$ unless
$-\xi^2\le\xi^2+2|N|+4$, i.e. unless $-\xi^2\le |N|+2$.
\end{proof}

\begin{Remark}
\thmref{vanwall} implies that for any $\xi$ the generating function $\delta_{\xi}^X(L)$ for the wallcrossing terms $\delta_{\xi,d}^X(L)$ can be obtained by a finite simple computation.
We know that $\delta_{\xi,d}^X(L)=0$ for $d>\xi^2+2|\<\xi,L-K_X\>|+4$. Thus to determine $\delta_{\xi}^X(L)$ we only need to compute the
coefficient of $q^0$ of the coefficients of $\Lambda^d$ with $d\le \xi^2+2|\<\xi,L-K_X\>|+4$ of $\overline \Delta^X_\xi(L)$, this is a finite computation with
the formal Laurent series involved.
\end{Remark}

\thmref{vanwall} implies that
the generating functions
$\sum_{d>0}\chi(M^X_H(c_1,d),\mu(L))$ are essentially independent of $H$:
If $H_1$, $H_2$ are two ample line bundles on $X$, then
$\chi^{X,H_1}_{c_1}(L)-\chi^{X,H_2}_{c_1}(L)$ is a polynomial in $\Lambda$, i.e. $\chi(M_{H_1}^{X}(c_1,d),\mu(L))=\chi(M_{H_2}^{X}(c_1,d),\mu(L))$ for all $d$ which are sufficiently large.

\subsection{The case of rational ruled surfaces}
We investigate the dependence of the polarization of the $K$-theoretic Donaldson invariants in the case of $\P^1\times\P^1$ and $\widehat \P^2$.
We write $F=H-E$ for the class of a ruling on $\widehat \P^2$ and $G=(H+E)/2$.
\thmref{vanwall} implies that
the generating functions
$\sum_{d>0}\chi(M^X_H(c_1,d),\mu(L))$ are essentially independent of $H$:
If $H_1$, $H_2$ are two ample line bundles on $X$, then
$\chi^{X,H_1}_{c_1}(L)-\chi^{X,H_2}_{c_1}(L)$ is a polynomial in $\Lambda$, i.e. $\chi(M_{H_1}^{X}(c_1,d),\mu(L))=\chi(M_{H_2}^{X}(c_1,d),\mu(L))$ for all $d$ which are sufficiently large.

\begin{Proposition}\label{asympt}
Let $X=\P^1\times\P^1$ or $X=\widehat \P^2$,
Let $L$ be a line bundle on $X$, let $c_1\in H^2(X,\Z)$, and let $H_1$, $H_2$ be ample on $X$.
Then there exists a $d_0>0$, such that 
$$\chi^{X,H_1}_{c_1}(L)\equiv \chi^{X,H_2}_{c_1}(L) \mod \Lambda^{d_0}.$$
\end{Proposition}
\begin{proof}
We write $L=aF+bG$ with $a,b\in \frac{1}{2}\Z$, and $H_i=n_iF+G$ for $i=1,2$,  and $n_i\in \Q_{>0}$.
Then $L-K_X=(a+2)F+(b+2)G$. We can assume $n_1>n_2$. The classes $\xi$ of type $(c_1)$ with $\<\xi,H_1\><0<\<\xi, H_2\>$ are of the form
$\xi=\alpha F-\beta G$ with $\alpha,\beta\in \frac{1}{2}\Z_{>0}$ satisfying $\beta n_2<\alpha<\beta n_1$. Assume $\delta^X_\xi(L)\ne 0$.
Then we get by \thmref{vanwall}  
\begin{align*}
-\xi^2=2\alpha\beta\le |(a+2)\beta-(b+2)\alpha|+2\le (|a|+2)\beta+(|b|+2)\alpha+2\le  (|a|+4)\beta+(|b|+4)\alpha.
\end{align*}
Therefore $\alpha\le (|a|+4)$ or $\beta\le (|b|+4)$. In the first case $0<\beta\le \frac{|a|+4}{n_2}$, in the second $0<\alpha\le n_1\beta$.
In both cases $\alpha$ and $\beta$ are bounded, thus, as $\alpha,\beta\in \frac{1}{2}\Z$, we see that there are only finitely many $\xi$ of type $(c_1)$ with 
$\<\xi,H_1\><0<\<\xi, H_2\>$ and $\delta^X_\xi(L)\ne 0$. Let $d_0$ be the maximum over these classes of 
$\xi^2+2|\<\xi,L-K_X)|+4$. Then $\chi(M^X_{H_1}(c_1,d),\mu(L))=\chi(M^X_{H_2}(c_1,d),\mu(L))$ for all $d>d_0$.
\end{proof}

For the line bundles we consider above  we can be more specific.

\begin{Proposition}\label{ruledconst}
\begin{enumerate}
\item On $X=\P^1\times \P^1$ and $X=\widehat \P^2$, let $L=nF+lG$, with $0\le l\le 2$ and $n\ge 0$. We write an ample divisor  on $\P^1\times \P^1$ or $\widehat \P^2$ as $aF+bG$.
Then
$\chi^{X,aF+bG}_{0}(L)$ and $\chi^{X,aF+bG}_{F}(L)$ are  independent of the ample class $aF+bG$ as long as $\frac{a}{b}>\frac{n+2}{4}$.
\item If $0\le n\le 3$, $0\le e\le \min(n,2)$ we have  $\chi^{\widehat \P^2,P}_{0}(nH-eE)$ and $\chi^{\widehat \P^2,P}_{F}(nH-eE)$ are independent of the ample class $P$ on $\widehat \P^2$.
\end{enumerate}
\end{Proposition}
\begin{proof}
(1) Let $X=\P^1\times\P^1$ or $X=\widehat \P^2$. Let $c_1=0$
or $c_1=F$. Then the classes $\xi$ of type $(c_1)$ on $X$ with $\<F,\xi\><0$
can be written as $\xi=aF-bG$ with $a,b$ positive integers and $b$ even.  $\xi$ is orthogonal to $P=aF+bG$.  If  $\delta^{X}_\xi(L)\ne 0$, then by \thmref{vanwall} we get
$$-(aF-bG)^2\le |(aF-bG)((n+2)F+(l+2)G)|+2,$$
 i.e. $2ab\le a(l+2)-b(n+2)+2$
or $2ab\le -a(l+2)+b(n+2)+2$. By $b\ge 2$ and $n\ge 0$, we have  $-b(n+2)+2<0$; and by $b\ge 2$, $0\le l\le 2$ we have $2ab\ge a(l+2)$.
Thus $2ab>a(l+2)-b(n+2)+2$.
On the other hand again by $a\ge 1,b\ge 2$, the inequality $2ab\le -a(l+2)+b(n+2)+2$ implies $2a\le(n+2)$.
As $b\ge 2$, this implies that  $\frac{a}{b}\le \frac{n+2}{4}$.

(3) Now let $X=\widehat \P^2$, and $\xi=aF-bG$ a class of type $(0)$ or of type $(F)$, with $\<\xi, F\><0$. Then $a\ge 1$ and $b$ is even with $b\ge 2$.
Let $L=nF+lG$ with $l\le 2$, $n\le 3$.
We have seen above that $a\le \frac{n+2}{2}$. Thus, if $n\le 1$, we get $a=1$, but $b\ge 2$ implies that $\xi=(1-b/2)H-(1+b/2)E$ is not orthogonal to an ample class.
If $2\le n\le 3$, we get $a\le 2$, and we find that the only $\xi=aF-bG$ orthogonal to an ample class is $2F-2G=H-3E$.
This is  a class of type $(F)$, not of type $(0)$, so the claim follows if $c_1=0$.
If $c_1=F$, the possible values for $L$ are
$L=0$, $L=H-E$, $L=2H-2E$, $L=2H$, $L=3H-E$. In each of these cases one sees directly that
$$-(H-3E)^2>|(H-3E)(L+3H-E)|+2, $$
so that $\delta^{X}_{H-3E}(L)= 0$.
\end{proof}

\section{Indefinite Theta functions and vanishing   and blowup formulas}

\subsection{Theta functions for indefinite lattices}
We begin by reviewing some notations about modular forms.

\begin{Definition}
Let $T:=\left(\begin{matrix} 1 & 1\\0&1\end{matrix}\right)$,
$S:=\left(\begin{matrix} 0 & -1\\1&0\end{matrix}\right)\in SL(2,\Z)$.
 Let
$\Gamma^0(4):=\Big\{\left(\begin{matrix}a&b\\c&d\end{matrix}\right)\in SL(2,\Z)\Bigm| b \equiv 0 (\hbox{ mod }4)\Big\}.$
We recall that $\Gamma^0(4)$ is generated as a group by
$\pm T^4, \pm TST$.
Recall that the cusps of $\H/\Gamma^0(4)$ are $\infty$ (of width $4$)
and $0$, $2$ (of width $1$).

For $\tau\in \H$ and $A=\left(\begin{matrix} a&b\\c&d\end{matrix}\right)\in \Gamma^0(4)$, we write $A\tau:=\frac{a\tau+b}{c\tau+d}$. A holomorphic function $f:\H\to \C$ is called a weakly holomorphic modular form of weight $k$ on $\Gamma^0(4)$ if $f(A\tau)=(c\tau+d)^kf(\tau)$ for $\tau\in \H,$ $ A\in \Gamma^0(4)$, and $f$ is meromorphic at the cusps.
We denote $M_k^!(\Gamma^0(4))$ the set of all weakly holomorphic modular forms
of weight $k$ on $\Gamma^0(4)$.
\end{Definition}

For us a {\it lattice} is a free $\Z$-module $\Gamma$ together with a quadratic form
$Q:\Gamma\to \frac{1}{2}\Z$, such that the associated bilinear form
$x\cdot y:=Q(x+y)-Q(x)-Q(y)$ is nondegenerate and $\Z$-valued.
We denote the extension of the quadratic and bilinear form to
$\Gamma_\R:=\Gamma\otimes_{\Z} \R$ and $\Gamma_\C:=\Gamma\otimes_{\Z} \C$ by the same letters.
Later the lattice we consider will be $H^2(X,\Z)$ with the negative of the intersection form.
We will then denote $\<F,G\>$ the intersection form and $F^2$ the self intersection, and write
$F\cdot G$ for the negative of the intersection form.

Now let $\Gamma$ be a lattice of rank $r$. Denote by $M_\Gamma$ the set of meromorphic
maps $f:\Gamma_\C\times \H\to \C$.
For $A=\left(\begin{matrix} a&b\\c&d\end{matrix}\right)\in SL(2,\Z)$,
we define  a map
$|_kA:M_\Gamma\to M_\Gamma$ by
$$f|_{k}A(x,\tau):=(c\tau+d)^{-k}\exp\left(-2\pi \sqrt{-1}\frac{cQ(x)}{c\tau+d}\right)
f\left(\frac{x}{c\tau+d},\frac{a\tau+b}{c\tau+d}\right).$$
Then $|_kA$ defines an action of $SL(2,\Z)$ on $M_\Gamma$.

We briefly review the theta functions for indefinite lattices of type $(r-1,1)$
introduced in \cite{GZ}, when $\Gamma$ will be  $H^2(X,\Z)$ for a rational surface $X$ with the negative of the intersection form, and $h$ will be the class of an ample divisor on $X$.
%As we will only be interested in the case that
%the lattice is $H^2(X,\Z)$ with the negative of the intersection form, we
%will specialize to this case and write the pairing in the lattice as
%$V\cdot W:=-\<VW\>$, for $V,W\in H^2(X,\Q)$.

We denote
\begin{align*}S_\Gamma&:=\big\{ f\in \Gamma\bigm| f \hbox{ primitive},  Q(f)=0,\  f\cdot h <0\big\},\
C_\Gamma:=\big\{ m\in \Gamma_\R\bigm| Q(m)<0, \ m\cdot h<0\big\}.
\end{align*}
For $f\in S_\Gamma$ put
$D(f):=\big\{(\tau,x)\in \H\times \Gamma_\C\bigm| 0< \Im(f\cdot x)<\Im(\tau)\big\},$
and for $h\in C_\Gamma$ put $D(h)=\H\times \Gamma_\C$.
For $t\in \R$ denote $$\mu(t):=\begin{cases} 1& t\ge 0, \\0 & t<0.\end{cases}$$

Let $c,b\in \Gamma$. Let  $f,g\in S_\Gamma\cup C_\Gamma$.
Then for $(\tau,x)\in D(f)\cap D(g)$
define
$$\Theta^{f,g}_{\Gamma,c,b}(\tau,x):=\sum_{\xi\in \Gamma+c/2}(\mu(\xi\cdot f)-\mu(\xi\cdot  g)) e^{2\pi i \tau Q(\xi)}e^{2\pi i \xi\cdot(x+b/2)}.$$
In \cite{GZ} it is shown that this sum converges absolutely and locally uniformly on $D(f)\cap D(g)$.
Furthermore the following are shown:
\begin{Remark}
For $f\in S_\Gamma$, $g\in C_\Gamma$ the function $\Theta^{f,g}_{X,c,b}(\tau,x)$ has a meromorphic continuation to
 $|\Im(f\cdot x)/\Im(\tau)|<1$, given by the Fourier expansion.
$$ \sum_{\xi \cdot f \ne 0} (\mu(\xi\cdot f)-\mu(\xi\cdot g))e^{2\pi i \tau Q(\xi)}e^{2\pi i \xi\cdot(x+b/2)}+\frac{1}{1-e^{2\pi i f\cdot (x+b/2)}}\sum_{{\xi\cdot  f=0}\atop {f\cdot g\le \xi\cdot g<0}}
e^{2\pi i \tau Q(\xi)}e^{2\pi i \xi\cdot(x+b/2)},$$
with the sums running through $\xi\in \Gamma+c/2$.
\end{Remark}
\begin{Theorem}\label{thetatrans}
\begin{enumerate}
\item For $f,g\in S_\Gamma$
the function $\Theta^{f,g}_{X,c,b}(\tau,x)$ has a meromorphic continuation to
$\H\times \Gamma_\C$.
\item For
 $|\Im(f\cdot x)/\Im(\tau)|<1$ and $|\Im(g\cdot x)/\Im(\tau)|<1$ it has a Fourier development
\begin{align*}\label{theta}
&\Theta_{X,c,b}^{f,g}(x,\tau):=\frac{1}{1-e^{2\pi i f\cdot (x+b/2)}}
\sum_{\substack{\xi\cdot f=0\\ f\cdot g\le\xi\cdot g<0}}e^{2\pi i \tau Q(\xi)}e^{2\pi i \xi \cdot (x+b/2)}\\ &
-\frac{1}{1-e^{2\pi i g\cdot (x+b/2)}}\sum_{\substack{\xi\cdot g=0\\ f\cdot g \le \xi \cdot f<0 }}e^{2\pi i \tau Q(\xi)}e^{2\pi i \xi \cdot(x+b/2)}+
\sum_{\xi\cdot f>0>\xi\cdot g} e^{2\pi i \tau Q(\xi)}\big(e^{2\pi i \xi \cdot(x+b/2)}-
e^{-2\pi i \xi \cdot(x+b/2)}\big),
\end{align*}
where the sums are always over $v\in \Gamma+c/2$.
\item (parity)
$\Theta_{X,c,b}^{f,g}(-x,\tau)=-(-1)^{c\cdot b}\Theta_{X,c,b}^{f,g}(x,\tau),$
\item (modular properties)
\begin{equation*}
\label{thetajacobi}
\begin{split}
(\Theta_{X,c,b}^{f,g}\theta_{3}^{\sigma(\Gamma)})|_1S&=(-1)^{-b\cdot c/2} \Theta_{X,b,c}^{f,g}\theta_{3}^{\sigma(\Gamma)},\\
(\Theta_{X,c,b}^{f,g}\theta_{3}^{\sigma(\Gamma)})|_1T&=(-1)^{3Q(c)/2-c w/2} \Theta_{X,c,b-c+w}^{f,g}\theta_{4}^{\sigma(\Gamma)},\\
(\Theta_{X,c,b}^{f,g}\theta_{3}^{\sigma(\Gamma)})|_1T^2&=(-1)^{-Q(c)}
\Theta_{X,c,b}^{f,g}\theta_{3}^{\sigma(\Gamma)},\\
(\Theta_{X,c,b}^{f,g}\theta_{3}^{\sigma(\Gamma)})|_1T^{-1}S&=
(-1)^{-Q(c)/2-c\cdot b/2}\Theta_{X,w-c+b,c}^{f,g}\theta_{2}^{\sigma(\Gamma)},
\end{split}\end{equation*}
where $w$ is a characteristic element of $\Gamma$.
\end{enumerate}
\end{Theorem}
\begin{Remark}
For $f,\ g,\ h\in C_\Gamma\cap S_\Gamma$ we have the cocycle condition:
$\Theta^{f,g}_{\Gamma,c,b}(\tau,x)+\Theta^{g,h}_{\Gamma,c,b}(\tau,x)=\Theta^{f,h}_{\Gamma,c,b}(\tau,x)$, which holds wherever all three terms are defined.
\end{Remark}

In the following let $X$ be a rational algebraic surface. We assume for simplicity that $-K_X$ is ample on $X$.
We can express the difference of the $K$-theoretic Donaldson invariants  for two different polarizations in terms of these indefinite theta functions. Here we take $\Gamma$ to be $H^2(X,\Z)$ with the negative of the intersection form. In the formulas above we will take the characteristic element to be $K_X$.

\begin{Definition}
Let $F,G\in S_\Gamma\cup C_\Gamma$, let $c_1\in H^2(X,\Z)$.
We put
\begin{align*}
\Psi^{F,G}_{X,c_1}(L;\Lambda,\tau)&:=
\Theta^{F,G}_{X,c_1,K_X}\Big(\frac{(L-K_X)h}{2\pi i},\tau\Big)   \Lambda^2
\widetilde\theta_4(h)^{(L-K_X)^2}\theta_4^{\sigma(X)}u'h^*.
%\left(\widetilde\Theta_4(\tau,h)^{(L-K_X)^2}\frac{\theta_{01}^{\sigma(X)+8}}{\theta_{00}^3\theta_{10}^3}\frac{1}{\sqrt{1+u+\Lambda^4}}.
\end{align*}
%\begin{NB}
%Explain, why this always makes sense
%\end{NB}
\end{Definition}

\begin{Corollary}\label{thetawall}
Let $H_1,H_2$ be ample on $X$, and assume that they do not lie on a wall of type $(c_1)$.
Then
\begin{align*}
\chi^{X,H_2}_{c_1}(L)-\chi^{X,H_1}_{c_1}(L)=\Coeff_{q^0} \big[\Psi^{H_1,H_2}_{X,c_1}(L;\Lambda,\tau)\big].
\end{align*}
\end{Corollary}

\begin{proof}
We write
$J:=\Lambda^2\widetilde\theta_4(h)^{(L-K_X)^2}\theta_4^\sigma(X)u'h^*.$
%We have seen that $\theta_4(h)\in \Q[[q^2\Lambda^2,q^4]]$. By definition we have $u'\in q^{-2}\Q[q^4]$ and by \eqref{hhh} we have
%$h^*\in q^{-1}\Lambda\Q[q^-2\Lambda^2,q^4,\Lambda^4]$.
%Thus $J\in q^{-3}\Lambda^3\Q[q^-2\Lambda^2,q^4,\Lambda^4]$, and $q^{-\xi^2}e^{\<\frac{\xi}{2},L-K_X\>h}\in q^{-\xi^2}\Q[[q^{-1}\Lambda,q^{4},\Lambda^4]]$.
Now assume that $\xi$ is congruent to $c_1$ modulo $2H^2(X,\Z)$.
Then it follows that $\delta_{\xi,d}^X=0$ unless
$d+\xi^2\ge 0$ and $d$ is congruent to $-c_1^2$ modulo $4$, i.e.
unless $\xi$ is a class of type $(c_1,d)$.
Therefore \thmref{wallcr}, \remref{chi04} and \remref{delb} imply that
\begin{align*}
\chi^{X,H_2}_{c_1}(L)-\chi^{X,H_1}_{c_1}(L)&=\sum_{H_2\xi>0>H_1\xi} \Coeff_{q^0}\big[\overline \Delta^X_\xi(L)\big]\\
&=
\Coeff_{q^0}\Big[\sum_{H_2\xi>0>H_1\xi}
 i^{\<\xi,K_X\>}q^{-\xi^2}\big(e^{\<\frac{\xi}{2},L-K_X\>h}-(-1)^{\xi K_X}e^{\<-\frac{\xi}{2},L-K_X\>h}\big)J\Big]\\
&=\Coeff_{q^0}\Big[\Theta^{H_1,H_2}_{X,c_1,K_X}\Big(\frac{(L-K_X)h}{2\pi \sqrt{-1}},\tau\Big)J\Big]
\end{align*}
The sums are over $\xi\in 2H^2(X,\Z)+c_1$.
%In the last line we replace $\xi$ by $-\xi$, use
%that $J$ is even in $\Lambda$ and the coefficient of $\Lambda^d$ in the second line is zero
%unless $d$ is congruent modulo $c_1^2$ modulo $2$.
\end{proof}

We use \corref{thetawall} to extend the generating function $\chi^{X,H}_{c_1}(L)$ formally  to
$H\in S_L\cup C_L$.

\begin{Definition}\label{chiext}
Let $M$ be ample on $X$ and not on a wall of type $(c_1)$. Let $H\in S_X\cup C_X$. We put
$$\chi^{X,H}_{c_1}(L):=\chi^{X,M}_{c_1}(L)+\Coeff_{q^0} \big[\Psi^{M,H}_{X,c_1}(L;\Lambda,\tau)\big].$$
By the cocycle condition, the definition of $\chi^{X,M}_{c_1}(L)$  is independent of the choice of $H$.
Furthermore by  \corref{thetawall} this coincides with the previous definition in case $M$ is also ample and does not lie on a wall of type $(c_1)$.
\end{Definition}

\begin{Remark}\label{average}
Let $H$ be ample on $X$, possibly lying on a wall of type $(c_1,d)$, let $H_0$ be ample on $X$ in an adjacent chamber of type $(c_1,d)$, i.e. 
$H_0$ does not lie on a wall of type $(c_1,d)$ and there are no classes of type $(c_1,d)$ with $\<\xi, H_0\><0<\<\xi, H\>$.
Then the definition of $\Theta^{f,g}_{\Gamma,c,b}$ implies that 
$$\Coeff_{\Lambda^d}\big[\chi^{X,H}_{c_1}(L)\big]=\Coeff_{\Lambda^d}\big[\chi^{X,H_0}_{c_1}(L)\big]+\frac{1}{2}\sum_{\xi}\delta^X_{\xi,d}(L),$$
with $\xi$ running over all classes of type $(c_1,d)$ with $\<\xi, H_0\><0=\<\xi, H\>$.
\end{Remark}

\subsection{Extension of blowup formulas}
Now we will extend \lemref{blowsimple} to $\chi^{X,H}_{c_1}(L)$ for $H\in C_X\cup S_X$.
We will continue to denote  a class in $H^2(X,\Z)$ and its pullback to the blowup $\widehat X$ of $X$ in a point by the same letter.

\begin{Proposition} \label{blowgen}
Let $X$ be a rational surface. Let $H\in C_X\cup S_X$. Let $c_1\in H^2(X,\Z)$. Let $\widehat X$ be the blowup of $X$ in a general point, and $E$ the exceptional divisor. Assume 
$-K_{\widehat X} is ample.$
Let $L\in \Pic(X)$ with $\<L, c_1\>$ even.
\begin{enumerate}
\item $\chi^{\widehat X,H}_{c_1}(L)=\chi^{X,H}_{c_1}(L)$,
\item $\chi^{\widehat X,H}_{c_1+E}(L)=\Lambda \chi^{X,H}_{c_1}(L)$.
\end{enumerate}
%begin{NB} Determine $a_1,a_4,a_5$ by explicit computation. Done\end{NB}
\end{Proposition}
\begin{proof}
We can choose $H_0\in C_X$, which does not lie on any wall of type $(c_1)$ on $X$. In case $X=\P^2$ we take $H_0$ the hyperplane class. If $b_2(X)>1$ we can take $H_0$ general in $C_X$.
For each $d>0$ and all $c_1\in H^2(X,\Z)$ we choose  $\epsilon_d>0$, such that there is no class $\xi$ of type $(c_1,d)$ or of type $(c_1+E,d+1)$  on $\widehat X$ with
 $\<\xi, H_0\><0<\<\xi,H_0-\epsilon_d E\>$.
For simplicity we will write $M^{\widehat X}_{H_{0+}}(c_1,d):=M^{\widehat X}_{H_0-\epsilon_d}(c_1,d)$, $M^{\widehat X}_{H_{0+}}(c_1+E,d+1):=M^{\widehat X}_{H_0-\epsilon_d}(c_1+E,d+1)$.
Then by \lemref{blowsimple} we have for $c_1\not \in 2H^2(X,\Z)$ or $d>4$ that
\begin{equation}\label{blowsimform}
\begin{split}
\chi(M^{\widehat X}_{H_{0+}}(c_1,d),\mu(L))&=\chi(M^{X}_{H_0}(c_1,d),\mu(L)), \\ \chi(M^{\widehat X}_{H_{0+}}(c_1+E,d+1),\mu(L))&=\chi(M^{X}_{H_0}(c_1,d),\mu(L)).
\end{split}
\end{equation}
In case $c_1\not\in 2H^2(X,\Z)$, we see that $H_0$ does not lie on a wall of type $(c_1,d)$ or $(c_1+E,d)$, on $\widehat X$, and we get
$$\chi(M^{\widehat X}_{H_{0+}}(c_1,d),\mu(L))=\Coeff_{\Lambda^d}\big[\chi^{\widehat X,H_0}_{c_1}(L)\big],\quad
\chi(M^{\widehat X}_{H_{0+}}(c_1,d),\mu(L))=\Coeff_{\Lambda^{d+1}}\big[\chi^{\widehat X,H_0}_{c_1+E}(L)\big].$$
Thus (1) and (2)   follow for $H_0$ in case $c_1\not\in 2H^2(X,\Z)$.

Now we deal with the case $c_1=0$.
By \eqref{eq:Kdon0} and \eqref{blowsimform} and \remref{average} we get
\begin{equation}\label{Ewall}\begin{split}
\chi^{X,H_0}_{0}(L)&=\sum_{d>0}\chi(M^{\widehat X}_{H_{0+}}(E,d+1),\mu(L))\Lambda^d+(\<L,K_X\>-\frac{K_X^2+L^2}{2}-1)\Lambda^4\\
&=\frac{1}{\Lambda}\left(\chi^{\widehat X,H_0}_{E}(L)+\frac{1}{2}\sum_{\xi}\delta_{\xi}^{\widehat X}(L)\right)+(\<L,K_X\>-\frac{K_X^2+L^2}{2}-1)\Lambda^4,
\end{split}
\end{equation}
 where $\xi$ runs through all classes $\xi$ of type $(E,d+1)$ on $\widehat X$ with $\<\xi,H_0\>=0$ and $\<\xi,E\> <0$; in other words $\xi=(2n-1)E$, with $n>0$.
By \thmref{vanwall} we get that $\delta_{nE}^{\widehat X}(L)=0$ unless $n^2\le n+2$.
Thus   in \eqref{Ewall}  we only have to consider $\xi=E$, and \thmref{vanwall}
gives that $\delta_{E}^{\widehat X,d}(L)=0$ for $d> 5$. Explicit computations with the Fourier developments of the functions occurring in \eqref{Delbar} give
$$\delta_{E}^{\widehat X}(L)=(-\<2K_{\widehat X},L\> + (K_{\widehat X+L^2}^2 + 3))\Lambda^5=(-\<2K_{X},L\> + (K_{X}^2 +L^2+ 2))\Lambda^5.$$
This shows (2) for $H_0$ and $c_1=0$.

Let $\widetilde X$ be the blow up of $\widehat X$ in a general point. Then we get  by (1) in case $c_1\not\in 2H^2(X,\Z)$ and by (2) that
$$\chi^{X,H_0}_{0}(L)=\frac{1}{\Lambda}\chi^{\widehat X,H_0}_{E}(L)=\frac{1}{\Lambda}\chi^{\widetilde X,H_0}_{E}(L)=\chi^{\widetilde X,H_0}_{0}(L).$$
Now for general $H$ in $C_X\cup S_X$  \defref{chiext} gives
$$\chi^{X,H}_{c_1}(L)=\chi^{X,H_0}_{c_1}(L)+\Coeff_{q^0}\big[\Psi^{H_0,H}_{X,c_1}(L,\Lambda,\tau)\big], \quad \chi^{\widehat X,H}_{c_1}(L)=\chi^{\widehat X,H_0}_{c_1}(L)+\Coeff_{q^0}\big[\Psi^{H_0,H}_{\widehat X,c_1}(L,\Lambda,\tau)\big].$$
By $H^2(\widehat X,\Z)=H^2(X,\Z)+\Z E$, we find that
$$\Theta^{H_0,H}_{\widehat X,c_1,K_{\widehat X}}\left(\frac{1}{2\pi i} (L-K_{\widehat X})h,\tau\right)=\Theta^{H_0,H}_{X,c_1,K_X}\left(\frac{1}{2\pi i} (L-K_X)h,\tau\right)\theta_4(h).$$
As $(L-K_{\widehat X})^2=(L-K_{X})^2-1$, and $\sigma(\widehat X)=\sigma(X)-1$, we get by definition
$$\Psi^{H_0,H}_{\widehat X,c_1}(L,\Lambda,\tau)=\Psi^{H_0,H}_{X,c_1}(L,\Lambda,\tau)\frac{\theta_4(h)}{\widetilde\theta_4(h)\theta_4}=\Psi^{H_0,H}_{X,c_1}(L,\Lambda,\tau).$$
Similarly we get the following
$$ \chi^{\widehat X,H}_{c_1+E}(L)=\chi^{\widehat X,H_0}_{c_1+E}(L)+\Coeff_{q^0}\big[\Psi^{H_0,H}_{\widehat X,c_1+E}(L,\Lambda,\tau)\big].$$
By $H^2(\widehat X,\Z)=H^2(X,\Z)+\Z E$, we find that
$$\Theta^{H_0,H}_{\widehat X,c_1+E,K_{\widehat X}}\left(\frac{1}{2\pi i} (L-K_{\widehat X})h,\tau\right)=\Theta^{H_0,H}_{X,c_1,K_X}\left(\frac{1}{2\pi i} (L-K_X)h,\tau\right)\theta_1(h).$$
Thus we get  by definition
$$\Psi^{H_0,H}_{\widehat X,c_1+E}(L,\Lambda,\tau)=\Psi^{H_0,H}_{X,c_1+E}(L,\Lambda,\tau)\frac{\theta_1(h)}{\widetilde\theta_4(h)\theta_4}=\Psi^{H_0,H}_{X,c_1}(L,\Lambda,\tau)\frac{\theta_1(h)}{\theta_4(h)},$$
and use $\Lambda=\frac{\theta_1(h)}{\theta_4(h)}$. This shows the result.
\end{proof}

\subsection{Modularity properties}
We want to show that under suitable assumptions the difference of the $K$-theoretic Donaldson invariants between two points $F,G\in S_X$ vanishes.
For this we first show that $\Psi^{F,G}_{X,c_1}(L;\Lambda,\tau)$ has a power series development in $\Lambda$, whose coefficients are modular forms of
weight $2$ on $\Gamma^0(4)$. The result is then proven by replacing the
$q$-development at the cusp $\infty$ by that at the other two cusps of
$\H/\Gamma^0(4)$.

{\bf Convention:} In this whole section $A$ will always stand for a matrix $A=
\left(\begin{matrix} a&b\\ c&d\end{matrix}\right)\in SL(2,\Z)$.

%We first review the transformation properties of some standard theta functions
%that we want to use freely in the future.
%By the definitions we have
%\begin{align*}
%\theta_{00}(\tau+1)=\theta_{01}(\tau), \theta_{00}(\tau+1)=\theta_{01}(\tau)
%(\theta_{00}\theta_{10})_{\tau\to \tau+1}=(-1)^{1/4}(\theta_{01}\theta_{10})
%\begin{NB}
%\begin{align*}
%(\theta_{00}\theta_{10})|_{\tau\to TST\tau }&=(-1)^{1/4}(\theta_{01}\theta_{10})|_{\tau\to ST\tau}=
%(-1)^{-1/4}(\theta_{10}\theta_{01})|_{\tau\to T\tau}=\theta_{00}\theta_{10},\\
%\theta_{00}\theta_{10}|_{\tau\to T^4\tau}&=-\theta_{00}\theta_{10}.
%\end{align*}
%\end{NB}

\begin{Lemma}\label{modular}
Let $F,G\in S_X$.
Then
$$\Psi^{F,G}_{X,c_1}(L;\Lambda,\tau)\in M_2^!(\Gamma^0(4))[[\Lambda]].$$
\end{Lemma}
In order to prove \lemref{modular}, we first study the transformation
behaviour of $\Theta^{F,G}_{X,c_1,K_X}$ under $\Gamma^0(4)$.

\begin{Lemma}\label{jacobi}
\begin{enumerate}
\item $(\Theta^{F,G}_{X,c_1,K_X}\theta_{4}^{\sigma(X)})|_1{A}=
(-1)^{(c_1^2) b/4}\Theta^{F,G}_{X,c_1,K_X}\theta_{4}^{\sigma(X)}$
for $A\in \Gamma^0(4)$.
\item $(\Theta^{F,G}_{X,c_1,K_X}\theta_{4}^{\sigma(X)})|_1{S}=
i^{\<c_1,K_X\> }\Theta^{F,G}_{X,K_X,c_1}\theta_{2}^{\sigma(X)}.$
\end{enumerate}
\end{Lemma}
\begin{proof}
By using the identities \eqref{thetajacobi} systematically we get

\begin{align*}
(\Theta_{X,c_1,K_X}^{F,G}\theta_{4}^{\sigma(X)})|_1T^2&=(-1)^{\<c_1^2\>/4-\<c_1,K_X\>/2}(\Theta_{X,c_1,c_1}^{F,G}\theta_{3}^{\sigma(X)})|_1T\\&=
(-1)^{-\<c_1^2\>/2}\Theta_{X,c_1,K_X}^{F,G}\theta_{4}^{\sigma(X)},
\end{align*}
In particular $(\Theta_{X,c_1,K_X}^{F,G}\theta_{4}^{\sigma(X)})|_1T^4=
(-1)^{c_1^2}\Theta_{X,c_1,K_X}^{F,G}\theta_{4}^{\sigma(X)}$.
 Similarly we get
\begin{align*}
(\Theta_{X,c_1,K_X}^{F,G}\theta_{4}^{\sigma(X)})|_1TST&=
(-1)^{c_1^2/4-\<c_1,K_X\>/2}(\Theta_{X,c_1,c_1}^{F,G}\theta_{3}^{\sigma(X)})|_1ST\\&=
(-1)^{3c_1^2/4-\<c_1,K_X\>/2} (\Theta_{X,c_1,c_1}^{F,G}\theta_{3}^{\sigma(X)})|_1T\\&=\Theta_{X,c_1,K_X}^{F,G}\theta_{4}^{\sigma(X)}.
\end{align*}
The last  two formulas imply  (1). Finally
\begin{align*}
(\Theta_{X,c_1,K_X}^{F,G}\theta_{4}^{\sigma(X)})|_1S&=
(-1)^{c_1^2/4-\<c_1,K_X\>/2}(\Theta^{F,G}_{X,c_1,c_1}\theta_{3}^{\sigma(X)})|_1 T^{-1}S\\&=
(-1)^{c_1^2-\<c_1,K_X\>/2}\Theta^{F,G}_{X,K_X,c_1}\theta_{2}^{\sigma(X)}=(-1)^{\<c_1,K_X\>/2}\Theta^{F,G}_{X,K_X,c_1}\theta_{2}^{\sigma(X)},
\end{align*}
where the last equality is because $K_X$ is characteristic, i.e. $(-1)^{c_1^2}=(-1)^{\<c_1,K_X\>}$.
\end{proof}

\begin{pf*}{Proof of \lemref{modular}}
By \cite[\S 22]{Ak} we have
\begin{equation}\label{th44}
\begin{split}
\widetilde \theta_4(z)|_{0}S&=\frac{\theta_2(z)}{\theta_2},\quad \frac{\theta_3(z)}{\theta_3}|_{0}S=\frac{\theta_3(z)}{\theta_3},\quad
\widetilde \theta_4(z)|_{0}T=\frac{\theta_3(z)}{\theta_3},\\
\widetilde \theta_4(z)|_{0}T^2&=\widetilde \theta_4(z),\quad
\widetilde \theta_4(z)|_{0}TST=\frac{\theta_3(z)}{\theta_3}|_0ST=\frac{\theta_3(z)}{\theta_3}|_0T=\widetilde \theta_4(z).
\end{split}
\end{equation}
In particular  $\widetilde\theta_4(h)|_0 A=\widetilde \theta_4(h)$ for $A\in \Gamma^0(4)$.
%\begin{align*}
%&\Theta^{F,G}_{X,c_1,K_X}\Big((L-K_X)\frac{z}{c\tau+d},\frac{a\tau+b}{c\tau+d}\Big)\theta_{01}\Big(0,\frac{a\tau+b}{c\tau+d}\Big)^{\sigma(X)}\Big(\frac{-\theta_{11}(\frac{z}{c\tau+d},\frac{a\tau+b}{c\tau+d})}{\theta_{01}(0,\frac{a\tau+b}{c\tau+d}}\Big)^{\<(L-K_X)^2\>}\\
%&=(-1)^{b/4(\<c_1^2\>+\<(L-K_X)^2\>)}(c\tau+d)
%\Theta^{F,G}_{X,c_1,K_X}\big((L-K_X)z,\tau\big)\theta_{01}^{\sigma(X)}\Big(\frac{-\theta_{11}(z,\tau)}{\theta_{01}}\Big)^{\<(L-K_X)^2\>}.
%\end{align*}
Writing
$$F(z,\tau):=\Theta^{G,F}_{X,c_1,K_X}\left(\frac{(L-K_X)z}{2\pi i},\tau\right)\theta_{4}^{\sigma(X)}\widetilde\theta_4(z)^{(L-K_X)^2},$$
we get by \lemref{jacobi} that
\begin{equation}\label{Ftrans}
F\Big(\frac{z}{c\tau+d},A\tau\Big)=
(-1)^{\<c_1^2\>b/4}(c\tau+d)F(z,\tau), \quad A\in \Gamma^0(4).
\end{equation}
Again by \cite[\S 22]{Ak} we see that $u=-\frac{\theta_2^2}{\theta_3^2}-\frac{\theta_3^2}{\theta_2^2}$ is a modular function on $\Gamma^0(4)$, and by definition it is clear that it is holomorphic on $\H$. We also see from \cite[\S 22]{Ak} that
$\theta_{2}\theta_{3}$ transforms under $A\in \Gamma^0(4)$ according to
$\theta_{2}(A\tau)\theta_{3}(A\tau)
=(-1)^{b/4}(c\tau+d)\theta_{2}(\tau)\theta_{3}(\tau)$.
Thus by \eqref{hhh}
we see that
$h$ transforms under $A\in \Gamma^0(4)$ by
\begin{equation}\label{hG}
h(\Lambda,A\tau)=(-1)^{b/4}\frac{h(\Lambda,\tau)}{c\tau+d}.
\end{equation}
Thus  by \eqref{Ftrans} we get for $A\in \Gamma^0(4)$ that
\begin{equation}\label{Gtrans}
\begin{split}
F\big(h(\Lambda,A\tau),A\tau\big)&=F\Big((-1)^{b/4} \frac{h(\Lambda,\tau)}{c\tau+d},A\tau\Big)=
(-1)^{c_1^2b/4}(c\tau+d)F\big((-1)^{b/4}h(\Lambda,\tau),\tau\big)\\
&=(-1)^{b/4}(c\tau+d)F\big(h(\Lambda,\tau),\tau\big).
\end{split}
\end{equation}
In  the last line we use $F(-z,\tau)=-(-1)^{\<c_1^2\>}F(z,\tau)$, which follows from \thmref{thetatrans}(3)
%that $\Theta^{F,G}_{X,C,K_X}(-x,\tau)=-(-1)^{\<C^2\>}\Theta^{F,G}_{X,C,K_X}(x,\tau)$
and the fact that $\widetilde \theta_{4}(z)$ is even in $z$.
We use \eqref{hstar} to write $$G(\Lambda,\tau):=\Lambda^2u'h^*=\frac{4i\Lambda^3\theta_{4}^8}{\theta_{2}^3\theta_{3}^3}\frac{1}{\sqrt{1+u\Lambda^2+\Lambda^4}}.$$ Using
\eqref{hhh} and the fact that $u$ is a modular function on $\Gamma^0(4)$, we have
$G(\Lambda,A\tau)=
(-1)^{b/4}(c\tau+d)G(\Lambda,\tau)$ for $A\in  \Gamma^0(4)$.
Putting all this together, we obtain for $A\in \Gamma^0(4)$ that
\begin{equation}\label{modu2}
\begin{split}
\Psi^{F,G}_{X,c_1}(L;\Lambda,A\tau)&=F\big(h(\Lambda,A\tau),A\tau\big)
G(\Lambda,A\tau)=(c\tau+d)^2\Psi^{F,G}_{X,c_1}(L;\Lambda,\tau).
\end{split}
\end{equation}
The Fourier development \eqref{theta} of $\Theta^{G,F}_{X,c_1,K_X}$
and the standard Fourier development of $\theta_{4}(z)$
imply that we can write $F$ as a formal power series  $F=\sum_{n\ge 0} f_n(\tau)z^n$, where each $f_n$ is meromorphic at the cusps and holomorphic on $\H$.
It is easy to see that $u$ is holomorphic on $\H$ and meromorphic  at the cusps.
We use that $\theta_{2}$,  $\theta_{3}$, $\theta_{4}$ are holomorphic on $\H$ and at the  cusps and without zero on $\H$.
By \eqref{hhh} we can write
$h(\Lambda,\tau)=\sum_{n\ge 1} h_{n}(\tau)\Lambda^{n}$, $G(\Lambda,\tau)=\sum_{n\ge 0} w_n(\tau)\Lambda^n,$
where each $h_n,\ w_n$ is holomorphic on $\H$ and meromorphic at the cusps.
Thus we obtain that
$$\Psi^{F,G}_{X,c_1}(L;\Lambda,\tau)=F\big(h(\Lambda,\tau),\tau\big)G(\Lambda,\tau)=\sum_{n\ge 0} p_n(\tau) \Lambda^n,$$
where each $p_n$ is holomorphic on $\H$ and meromorphic at the cusps. Thus    by \eqref{modu2} each $p_n$ is a weakly holomorphic modular form of weight $2$ on $\Gamma^0(4)$.
%\begin{NB} Check whether $\Psi$ can actually have a pole in $\Lambda$. I think it has in general a pole of order $1$.
%Actually it has a zero of at least order $2$, because of the factor $\Lambda^3$. \end{NB}
\end{pf*}

\subsection{Vanishing of the difference between boundary points}
Let again $X$ be a projective surface with $-K_X$ ample. Let  $F,G\in S_X$.
In this section we want to show that for any $c_1,d$ and any line bundle $L$ on $X$ with $\<c_1,L\>$ even, we have
$\chi(M^X_F(c_1,d),\mu(L))=\chi(M^X_G(c_1,d),\mu(L))$.

\begin{Theorem}\label{strucdiff}
Let $-K_X$ be ample. Let  $F,G\in S_X$.  Then
$$\chi^{X,G}_{c_1}(L;\Lambda)-\chi^{X,F}_{c_1}(L;\Lambda)=0.$$
\end{Theorem}
\begin{proof}
By \corref{thetawall} we have to show that
$\Coeff_{q^0}\big[\Psi^{F,G}_{X,c_1}(L;\Lambda,\tau)\big]=0.$

By \lemref{modular} we have   $\Psi^{F,G}_{X,c_1}(L;\Lambda,\tau)\in M_2^!(\Gamma^0(4))[[\Lambda]]$, i.e. $[\Psi^{F,G}_{X,c_1}(L;\Lambda,\tau)2\pi  i d\tau]_{\Lambda^d}$ is  for all $d$ a holomorphic differential form on $\H/\Gamma^0(4)$,
with a meromorphic extension to the cusps $\infty$, $0$ and $2$.
Note that $2\pi i q d\tau=8dq$, thus, taking into account that the width of the cusp $\infty$ is $4$, we get by the  residue theorem that
\begin{align*}\Coeff_{q^0}\big[\Psi^{F,G}_{X,c_1}&(L;\Lambda,\tau)\big]=
\res_{\tau=\infty}\Big[\Psi^{F,G}_{X,c_1}(L;\Lambda,\tau)\frac{\pi  i }{2}d\tau\Big]\\
&=
-\res_{\tau=0}\Big[\Psi^{F,G}_{X,c_1}(L;\Lambda,\tau) \frac{\pi  i }{2} d\tau
\Big]
-\res_{\tau=2} \Big[\Psi^{F,G}_{X,c_1}(L;\Lambda,\tau) \frac{\pi  i }{2} d\tau
\Big]\\&=
-\frac{1}{4}\Coeff_{q^0}\Big[\tau^{-2}\Psi^{F,G}_{X,c_1}(L;\Lambda,S\tau)\Big]
-\frac{1}{4}\Coeff_{q^0}\Big[(\tau-2)^{-2}\Psi^{F,G}_{X,c_1}(L;\Lambda,ST^{-2}\tau)\Big].
\end{align*}
By  \cite[\S 22]{Ak} we have
\begin{equation}\label{th234S}
\theta_2(-1/\tau)=\sqrt{-i\tau}\theta_4(\tau), \quad \theta_3(-1/\tau)=\sqrt{-i\tau}\theta_3(\tau), \quad \theta_3(-1/\tau)=\sqrt{-i\tau}\theta_2(\tau),
\end{equation}
Again we write
\begin{align*}F(z,\tau)&:=\Theta^{G,F}_{X,c_1,K_X}\left(\frac{(L-K_X)z}{2\pi i},\tau\right)\theta_{4}^{\sigma(X)}\widetilde\theta_4(z)^{(L-K_X)^2},\\
G(\Lambda,\tau)&:=\Lambda^2u'h^*=\frac{4i\Lambda^3\theta_4^8}{\theta_2^3\theta_3^3}\frac{1}{1+u\Lambda^2+\Lambda^4}.
\end{align*}
We put
$\widetilde u:=u|_0S$.
Then by \eqref{th44} we see that $\widetilde u=-\frac{\theta_{3}^4+\theta_{4}^4}{\theta_{3}^2\theta_{4}^2}=-2+O(q^4).$
Let
$\widetilde h(\Lambda,\tau):=\tau h(\Lambda,S\tau)$. By \eqref{th234S} we get $(\theta_2\theta_3)|_1S=-i \theta_3\theta_4.$ Thus by \eqref{hhh} we see that
$$
\widetilde h=-\frac{2}{\theta_{2}\theta_{3}}.
\sum_{\substack{n\ge 0\\ n\ge k\ge0}}
\binom{-\frac{1}{2}}{n}\binom{n}{k}\frac{\widetilde u^k\Lambda^{4n-2k+1}}{4n-2k+1}.
$$
In particular both $\widetilde u$ and $\widetilde h$ are regular at $q=0$.
We get by
 \lemref{jacobi}, \eqref{th234S} and \eqref{th44}  that
\begin{align*}F\Big(\frac{z}{\tau},S\tau\Big)&=
(-1)^{\<c_1,K_X\>/2}
 \Theta_{X,K_X,c_1}^{G,F}\Big(\frac{(L-K_X)z}{2\pi  i }, \tau\Big)
(-1)^{-\sigma(X)/4}\theta_{2}^{\sigma(X)} \left(\frac{\theta_{2}(z)}{\theta_{2}}\right)^{(L-K_X)^2}.\\
 G(\Lambda,\tau)|_1S &=\frac{4\Lambda^3\theta_2^8}{\theta_3^3\theta_4^3}\frac{1}{1+\widetilde u\Lambda^2+\Lambda^4}.
 \end{align*}
and
$$ F\big(h(\Lambda,S\tau),S\tau)=
(-1)^{\<c_1,K_X\>/2}
 \Theta_{X,K_X,c_1}^{G,F}\left(\frac{(L-K_X)\widetilde h}{2\pi  i }, \tau\right)(-1)^{-\sigma(X)/4}\theta_{2}^{\sigma(X)} \left(\frac{\theta_{2}(z)}{\theta_{2}}\right)^{(L-K_X)^2}.$$
 Putting this together we see that
 \begin{align*}
-\frac{1}{4}\tau^{-2}\Psi^{F,G}_{X,c_1}(L;\Lambda,S\tau)&= F\big(h(\Lambda,S\tau),S\tau)\cdot G(\Lambda,\tau)|_1S
\end{align*}
can be written as $\theta_2^{8+\sigma(X)}H(\Lambda,\tau)$, where $H(\Lambda,\tau)$ is regular at $q=0$.
Recall that $-K_X$ is ample, thus $K_X^2>0$. As $K_X^2-\sigma(X)=8$, this implies $\sigma(X)>-8$.  As $\theta_2$ has a zero of order $1$ in $q$, we find that
$$\Coeff_{q=0}\big[-\frac{1}{4}\tau^{-2}\Psi^{F,G}_{X,c_1}(L;\Lambda,S\tau)\big]=0.$$

Now we compute $\Coeff_{q^0}\big[(\tau-2)^{-2}\Psi^{F,G}_{X,c_1}(L;\Lambda,ST^{-2}\tau)\big]$.
It is easy to see that
$u(T^{-2}\tau)=-u(\tau)$, thus $u(T^{-2}\tau)\Lambda^2=(i\Lambda)^2 u   $ and
$\frac{1}{\theta_{2}\theta_{3}}| T^{-2}=
 \frac{i}{\theta_{2}\theta_{3}}$
Thus we get by \eqref{hhh} that
$
h(\Lambda,T^{-2}\tau)=h( i \Lambda,\tau)$ and
$G(\Lambda,T^{-2}\tau):=i^3G(i\Lambda,\tau)$.
We also see $\widetilde \theta_4(z)|_{T^{-2}}=\widetilde \theta_4(z)$ and
$$
\big(\Theta^{G,F}_{X,c_1,K_X}((L-K)z,\tau)\theta_{4}^{\sigma(X)}\big)|_{\tau\to T^{-2}\tau}=(-1)^{c_1^2/2}\Theta^{G,F}_{X,c_1,K_X}((L-K)z,\tau)\theta_{4}^{\sigma(X)}
$$
Combining these facts, we get
$\Psi^{F,G}_{X,c_1}(L;\Lambda,T^{-2}\tau)= i^{c_1^2+3}
\Psi^{F,G}_{X,c_1}(L, i \Lambda,\tau).$
Therefore  we get
$$(\tau-2)^{-2}\Psi^{F,G}_{X,c_1}(L;\Lambda,ST^{-2}\tau)
= i^{c_1^2+3}\tau^{-2}
\Psi^{F,G}_{X,c_1}(L, i \Lambda,S\tau),$$
and thus
$\Coeff_{q^0}\big[\Psi^{F,G}_{X,c_1}(L;\Lambda,ST^{-2}\tau]= 0$.
\end{proof}

\subsection{Vanishing at boundary of the positive cone}
The following standard fact allows us to compute the $K$-theoretic Donaldson
for rational surfaces.
\begin{Remark}\label{vanfib}
Let $X$ be a simply connected algebraic surface, and let
$\pi:X\to \P^1$ be a morphism whose general fibre is isomorphic to
$\P^1$. Let $M$ be ample on $X$. Let $F\in H^2(X,\Z)$ be the class of a fibre, and assume that
$\<c_1,F\>$ is odd. 
Then for all $d>0$ there exists an $\epsilon_d>0$, such that  $M_{F+\epsilon}^X(c_1,d)=\emptyset$ for all $d$ and all $0<\epsilon\le<\epsilon_d$.
As $\<F,\xi\>\ne 0$ for all $\xi$ of type $c_1$, we get that for all $d>0$ 
$$\Coeff_{\Lambda^d}[\chi^{X,F}_{c_1}(L)]=\chi(M_{F+\epsilon_d}^X(c_1,d),\mu(L))=0,$$ 
and therefore
$\chi^{X,F}_{c_1}(L)=0$.
\end{Remark}

This result together with \thmref{strucdiff} implies that for many rational surfaces $\chi^{X,F}_{c_1}(L)=0$ for all $F\in S_X$.

\begin{Theorem}\label{vanbound}
Let $X$ be $\P^1\times \P^1$ or the blowup of $\P^2$ in at most $7$ general points. Let $c_1\in H^2(X,\Z)$ and let $L$ be a line bundle on $X$ with $\<c_1,L\>$ even.
Let $F\in S_X$. Then
  $\chi^{X,F}_{c_1}(L)=0$.
\end{Theorem}

\begin{proof}
We note that $-K_X$ is ample on $X$.
If there is a morphism $\pi:X\to \P^1$ with general fibre isomorphic to $\P^1$, if $G$ is the class of a fibre of $\pi$,  then $G\in S_X$. Furthermore, if  $\<c_1, G\>$ odd,
 we get by \remref{vanfib} that
$\chi^{X,G}_{c_1}(L)=0$. Then \thmref{strucdiff} implies that $\chi^{X,F}_{c_1}(L)=0$ for all $F\in S_X$.

Let $\widehat X$ be the blowup of $X$ in a general point. We denote by $E$ the exceptional divisor.
Then  $\widehat X$ is the blowup of $\P^2$ in at most 8 general points and $-K_{\widehat X}$ is ample on $\widehat X$ .
Denote by $H$ the pullback of the hyperplane class of $\P^2$. Then $G:=H-E$ is the class of the fibre of a morphism $\pi:\widehat X\to \P^1$, whose general fibre is $\P^1$.

If $\<c_1,G\>$ is odd, then $\chi^{\widehat X,G}_{c_1}(L)=0$, and thus, applying  \thmref{strucdiff}  again, we get $\chi^{\widehat X,F'}_{c_1}(L)=0$ for all $F'\in S_{\widehat X}$.
If $F\in S_X$, then also $F\in S_{\widehat X}$ and by \propref{blowgen} we get  $\chi^{X,F}_{c_1}(L)=\chi^{\widehat X,F}_{c_1}(L)=0$.

If $\<c_1,G\>$ is even, then $\<c_1+E,G\>$ is odd, and thus $\chi^{\widehat X,G}_{c_1+E}(L)=0$. Again \thmref{strucdiff}  gives that $\chi^{\widehat X,F}_{c_1+E}(L)=0$ for all $F\in S_X$.
Thus we get by \propref{blowgen} that $\chi^{X,F}_{c_1}(L)=\frac{1}{\Lambda}\chi^{\widehat X,F}_{c_1+E}(L)=0$.
\end{proof}

\subsection{Blowup polynomials and  a higher blowup formula}

In this section we introduce and study the "blowup polynomials"
$R_n(\la,x)$, $S_n(\la,x)$, which are
related to addition formulas for the standard theta functions $\theta_1(z)$ and $\theta_4(z)$.
These are related to "higher blowup formulas": if $\widehat X$ is the blowup of $X$ in a point and $E$ the exceptional divisor, they relate
$\chi^{X,M}_{c_1}(L)$ to $\chi^{\widehat X,M}_{c_1}(L-(n-1)E)$ and $\chi^{X,M}_{c_1+E}(L-(n-1)E)$.

\begin{Definition}\label{blowpol}
Define for all $n\in \Z$  rational functions  $R_n$, $S_n\in \Q(\la,x)$ by
$R_0=R_1=1,$
$S_1=\la,\ S_2=\la x,$
the  recursion relations
\begin{align}
\label{recurR} R_{n+1}&=\frac{R_{n}^2-\la^2 S_n^2}{R_{n-1}},\qquad  n\ge 1,\\
\label{recurS} S_{n+1}&=\frac{S_{n}^2-\la^2 R_{n}^2}{S_{n-1}},\qquad  n\ge 2.
\end{align}
and $R_{-n}=R_{n}$, $S_{-n}=S_{n}$.
%We will prove  below that the $R_n$,
%$S_n$
%are indeed polynomials in $\la,x$.
The definition gives
\begin{align*}
R_1&=1,\ R_2=(1-\la^4), R_3=-\la^4 x^2 + (1-\la^4)^2,\ R_4=-\la^4x^4+(1-\la^4)^4,\\
S_1&=\la,\ S_2=\la x,\ S_3=\la(x^2-(1-\la^4)^2),\ S_4=\la x\big((1-\la^8)x^2-2(1-\la^4)^3\big).
\end{align*}
One can show that the $R_i$, $S_i$ are polynomials, but we will not need it here.
We will want to show that these polynomials are related to the following  expressions in theta functions.
We put
\begin{equation}\label{RSmod}
\widetilde R_n:=\frac{\WT_4(nh)}{\WT_4(h)^{n^2}},\quad \widetilde S_n:= \frac{\WT_1(nh)}{\WT_4(h)^{n^2}}.
\end{equation}
\end{Definition}

\begin{Proposition}\label{rnprop}
$R_n$, $S_n$  satisfy
\begin{equation}\label{ThRn}
\widetilde R_n=R_n(\Lambda,M),\qquad
\widetilde S_n=S_n(\Lambda,M).
\end{equation}
\end{Proposition}
\begin{proof}

%First we  show that there are unique rational functions $R_n,S_n\in
%\Q(x,t)$, such that
%$\widetilde R_n=R_n(\Lambda^4,M)$,
%$\widetilde S_n=S_n(\Lambda^4,M)$, and that these functions satisfy the initial conditions and the recursion formulas of (2):

%It is easy to see that $\Lambda$ and $M$ are algebraically independent, i.e. there exists no polynomial $f\in \Q[\lambda,x]\setminus \{0\}$, such that $f(\Lambda,M)=0$ as
%a function on $\H\times \C$. Thus rational functions $R_n,\ S_n\in
%\Q(x,t)$ satisfying \eqref{ThRn} are unique if they exist.
%\begin{NB} This point needs to be proven, or otherwise I have to show I do not need it. \end{NB}

As $\theta_1(h)$ is odd in $h$ and $\theta_4(h)$ is even in $h$, it follows that
$\widetilde R_{-n}=\widetilde R_n$, $\widetilde S_{-n}=-\widetilde S_n$, and by definition we get
$\widetilde R_0=1$, $\widetilde R_1=1$, $\widetilde S_1=\Lambda$.
The duplication formula for $\theta_1(g)$ (see \cite[\S 2.1 Ex.~5]{WW})
$$\theta_1(2g)\theta_2\theta_3\theta_4=2\theta_1(h)\theta_2(h)\theta_3(h)\theta_4(h)$$ gives
\begin{equation}
\widetilde S_2=\frac{\theta_1(2h)}{\WT_4(h)^4}=
\frac{2\WT_1(h)\theta_2(h)\theta_3(h)}{\th_2\th_3\WT_4(h)^3}=\Lambda M.
\end{equation}
The addition formula for  $\theta_4(z)$ (see \cite[\S 2.1 Ex.~1]{WW})
$$\theta_4(y+z)\theta_4(y-z)\theta_4^2=\theta_4(y)^2\theta_4(z)^2-\theta_1(y)^2\theta_1(z)^2$$
applied to
$y=nh$,  gives
\begin{equation}
\begin{split}
\WR_{n+1} \WR_{n-1}&=\frac{\WT_4((n+1)h)}{\WT_4(h)^{(n+1)^2}}\cdot \frac{\WT_4((n-1)h)}{\WT_4(h)^{(n-1)^2}}\\
&=
\frac{\WT_4(nh)^2}{\WT_4(h)^{2n^2}}-
\frac{\theta_{1}(nh)^2\theta_1(h)^2}{\WT_4(h)^{2n^2}\theta_4(h)^2}=
\WR_n^2-\Lambda^2\WS_n^2,
\end{split}
\end{equation}
where in the last step we have used the definition $\Lambda=\frac{\theta_1(h)}{\theta_4(h)}.$
Similarly the addition formula for  $\theta_1(z)$ (see \cite[\S 2.1 Ex.~1]{WW})
$$\theta_1(y+z)\theta_1(y-z)\theta_4^2=\theta_1(y)^2\theta_4(z)^2-\theta_4(y)^2\theta_1(z)^2$$
applied to
$y=nh$  gives
\begin{equation}
\begin{split}
\WS_{n+1} \WS_{n-1}&=\frac{\WT_1((n+1)h)}{\WT_4(h)^{(n+1)^2}}\cdot \frac{\WT_1((n-1)h)}{\WT_4(h)^{(n-1)^2}}\\
&=
\frac{\WT_1(nh)^2}{\WT_4(h)^{2n^2}}-
\frac{\WT_{4}(nh)^2\theta_1(h)^2}{\theta_4(h)^2\WT_4(h)^{2n^2}}=
\WS_n^2-\Lambda^2\WR_n^2.
\end{split}
\end{equation}
This shows the result.
\end{proof}

We use this result to prove a higher blowup formula. We will use it here for $n=2$. In a forthcoming paper the first named author
 will systematically use the higher blowup formulas
to study the $K$-theoretic Donaldson invariants of $\P^2$.

\begin{Proposition}\label{highblow}
Let $X$ be $\P^2$, $\P^1\times\P^1$ or the blowup of $\P^2$ in at most $7$ points.
Let $c_1\in H^2(X,\Z)$ and let $L$ be a line bundle on $X$ with $\<c_1, L\>$ even.
Let $F,M\in C_X\cup S_X$.
\begin{enumerate}
\item $\Psi^{F,M}_{\widehat X,c_1}(L-(n-1)E,\Lambda,\tau)=\Psi^{F,M}_{X,c_1}(L,\Lambda,\tau)R_{n}(\Lambda, M)$.
\item $\Psi^{F,M}_{\widehat X,c_1+E}(L-(n-1)E,\Lambda,\tau)=\Psi^{F,M}_{X,c_1}(L,\Lambda,\tau)S_{n}(\Lambda, M)$.
\end{enumerate}
\end{Proposition}

\begin{proof}
(1) Note that $H^2(\widehat X,\Z)=H^2(X,\Z)+\Z E$ and $L-(n-1)E-K_{\widehat X}=L-K_X-nE$. Also we have $(-1)^{\<K_{\widehat X},(nE)\>}=(-1)^n$,
$\<(nE), (kE)\>=-nk$.
By definition we get therefore
\begin{align*}
\Theta^{F,M}_{\widehat X,c_1,K_{\widehat X}}((L-(n-1)E-K_{\widehat X})z)&=\Theta^{F,M}_{\widehat X,c_1,K_{\widehat X}}(\frac{1}{2\pi i}(L-K_X)z)\sum_{k\in \Z} (-1)^n e^{\pi i \tau n^2}e^{nkz}
\\
&= \Theta^{F,M}_{\widehat X,c_1,K_{\widehat X}}(\frac{1}{2\pi i}(L-K_X)z)\theta_4(nz).
\end{align*}
We also see that $\sigma(\widehat X)=\sigma(X)-1$ and $(L-K_X-nE)^2=(L-K_X)^2-n^2$.
Putting this into the definitions of $\Psi^{F,M}_{\widehat X,c_1}$ and $\Psi^{F,M}_{X,c_1}$, we see that
$$\Psi^{F,M}_{\widehat X,c_1}(L-(n-1)E,\Lambda,\tau)=\Psi^{F,M}_{X,c_1}(L,\Lambda,\tau)\frac{\theta_4(nh)}{\theta_4\widetilde\theta_4(h)^{n^2}}=\Psi^{F,M}_{X,c_1}(L,\Lambda,\tau)R_{n}(\Lambda, M).$$

(2) The proof is similar. The same argument as above shows that
\begin{align*}
\Theta^{F,M}_{\widehat X,c_1+E,K_{\widehat X}}((L-(n-1)E-K_{\widehat X})z)&=\Theta^{F,M}_{\widehat X,c_1,K_{\widehat X}}(\frac{1}{2\pi i}(L-K_X)z)\sum_{k\in \Z} (-1)^n e^{\pi i \tau (n+1/2)^2}e^{(n+1/2)kz}
\\
&= \Theta^{F,M}_{\widehat X,c_1,K_{\widehat X}}(\frac{1}{2\pi i}(L-K_X)z)\theta_1(nz).
\end{align*}
Thus the definitions of $\Psi^{F,M}_{\widehat X,c_1+E}$ and $\Psi^{F,M}_{X,c_1}$ give that
$$\Psi^{F,M}_{\widehat X,c_1+E}(L-(n-1)E,\Lambda,\tau)=\Psi^{F,M}_{X,c_1}(L,\Lambda,\tau)\frac{\theta_1(nh)}{\theta_4\widetilde\theta_4(h)^{n^2}}=\Psi^{F,M}_{X,c_1}(L,\Lambda,\tau)S_{n}(\Lambda, M).$$
\end{proof}

\begin{Theorem}\label{1blow}
Let $X$ be $\P^2$, $\P^1\times\P^1$ or the blowup of $\P^2$ in at most $7$ points.
Let $c_1\in H^2(X,\Z)$ and let $L$ be a line bundle on $X$ with $\<c_1, L\>$ even.
Then we have for all $M\in C_X$
$$\chi^{X,M}_{c_1}(L)=\frac{\chi^{\widehat X,M}_{c_1}(L-E)}{1-\Lambda^4}.$$
\end{Theorem}
\begin{proof}
First consider the case that $X$ is not $\P^2$.
Then by \thmref{vanbound} there is an $F\in S_X$ with $\chi^{X,F}_{c_1}(L)=0=\chi^{\widehat X,F}_{c_1}(L-E)$.
Thus by \corref{thetawall} and \propref{highblow} we get
\begin{align*}
\chi^{\widehat X,M}_{c_1}(L-E)&=\Coeff_{q^0}\big[\Psi^{F,M}_{\widehat X,c_1}(L-E;\Lambda,\tau)\big]=\Coeff_{q^0}\big[\Psi^{F,M}_{X,c_1}(L;\Lambda,\tau)R_2(M,\Lambda)\big]\\&=
\Coeff_{q^0}\big[\Psi^{F,M}_{X,c_1}(L;\Lambda,\tau)\big](1-\Lambda^4)=\chi^{X,F}_{c_1}(L-E)(1-\Lambda^4).
\end{align*}
Now assume that $X=\P^2$ and $c_1=H$ is the hyperplane class. Let $p_1,p_2$ be two different points of $\P^2$.
For $i=1,2$ let $X_i$ be the blowup of $\P^2$ in $p_i$ with exceptional divisor $E_i$, and let $X$ be the blowup of $\P^2$ in $p_1$ and $p_2$.
Let $F_1:=H-E_1$, $F_2:=H-E_2$. Then
$$\chi^{\P^2,H}_H(L)=\chi^{X_2,H}_H(L)=\Coeff_{q^0}\big[\Psi^{F_2,H}_{X_2,H}(L;\Lambda,\tau)\big]=\frac{1}{1-\Lambda^4}\Coeff_{q^0}\big[\Psi^{F_2,H}_{X,H}(L-E_1;\Lambda,\tau)\big].$$
On the other hand by \thmref{strucdiff}  $\Coeff_{q^0}\big[\Psi^{F_2,F_1}_{X,H}(L-E_1;\Lambda,\tau)\big]=0$, and thus we get by \propref{blowgen}
\begin{align*} \Coeff_{q^0}\big[\Psi^{F_2,H}_{X,H}(L-E_1;\Lambda,\tau)\big]&=\Coeff_{q^0}\big[\Psi^{F_1,H}_{X,H}(L-E_1;\Lambda,\tau)\big]=
\Coeff_{q^0}\big[\Psi^{F_1,H}_{X_1,H}(L-E_1;\Lambda,\tau)\big]\\&=\chi^{X_1,H}_H(L-E_1).
\end{align*}
Finally let $X=\P^2$ and $c_1=0$. We use again \propref{blowgen} and the same argument to get
\begin{align*}
\chi^{\P^2,H}_0(L)&=\chi^{X_2,H}_0(L)=\Coeff_{q^0}\big[\Psi^{F_2,H}_{X_2,0}(nH;\Lambda,\tau)\big]=\frac{1}{1-\Lambda^4}\big(\Coeff_{q^0}\big[\Psi^{F_2,H}_{X,0}(nH-E_1;\Lambda,\tau)\big]\big).
\end{align*}
Again by \thmref{strucdiff} we get  $\Coeff_{q^0}\big[\Psi^{F_2,F_1}_{X,0}(nH-E_1;\Lambda,\tau)\big]=0$, and thus by \thmref{vanbound}
\begin{align*} \Coeff_{q^0}\big[\Psi^{F_2,H}_{X,0}(nH-E_1;\Lambda,\tau)\big]&=\Coeff_{q^0}\big[\Psi^{F_1,H}_{X,0}(nH-E_1;\Lambda,\tau)\big]=
\Coeff_{q^0}\big[\Psi^{F_1,H}_{X_1,0}(nH-E_1;\Lambda,\tau)\big]\\&=\chi^{X_1,H}_0(nH-E_1).
\end{align*}
The claim follows.
%\begin{NB} The constants need all the be determined. It is important to get them right, because in the end in a suitable sense the errors will cancel.
%\end{NB}
\end{proof}

\section{$K$-theoretic Donaldson invariants of rational ruled surfaces}
In this section we will compute  generating functions for $K$-theoretic Donaldson invariants of rational ruled surfaces, proving \thmref{p11t}.
We will do this by proving some recursion formulas for them, which determine them, once suitable initial conditions are satisfied.

\subsection{The limit of the invariant at the boundary point}
Let $X=\P^1\times \P^1$ or $\widehat \P^2$ the blowup of $\P^2$ in a point.
We denote the line bundles on $\P^1\times\P^1$ and $\widehat \P^2$ in a uniform way.
\begin{Notation}
Let $X=\P^1\times \P^1$ or $X=\widehat \P^2$. In the case $X=\P^1\times \P^1$ we denote $F$ the class of the fibre of the projection to the first factor, and by $G$ the class of the fibre of the projection to the second factor. In the case $X=\widehat \P^2$, let $H$ be the pullback of the hyperplane class on $\P^2$ and $E$ the class of the exceptional divisor. Then  $F:=H-E$ is the fibre of the ruling of $X$. We put $G:=\frac{1}{2}(H+E)$. Note that $G$ is not an integral cohomology class. In fact, while $H^2(\P^1\times\P^1,\Z)=\Z F\oplus \Z G$, we have
$$H^2(\P^1\times\P^1,\Z)=\Z H\oplus \Z E=\big\{aF+bG\bigm| a\in \Z,b\in 2\Z \hbox{ or } a\in \Z+\frac{1}{2}, b\in 2\Z+1\big\}.$$
On the other hand we note that both on $X=\P^1\times\P^1$ and $\widehat \P^2$ we have
$F^2=G^2=0$, $\<F,G\>=1$, and $-K_X=2F+2G$.
\end{Notation}

We want to define and study the limit of the $K$-theoretic Donaldson invariant $\chi\big(M_{P}^X(c_1,d),\mu(L)),$ as the ample class $P$ tends to $F$.
For $c_1=F$ or $c_1=0$ this will be different from our previous definition of $\chi\big(M_{F}^X(c_1,d),\mu(L)).$

\begin{Definition}
Fix $d\in \Z$ with $d\equiv -c_1^2 (4)$. For $n_d>0$ sufficiently large, $n_dF+ G$ is ample on $X$, and the condition
$\<c_1,F\>$ odd implies that there is no wall $\xi$ of type $(c_1,d)$ with
$\<\xi ,(n_dF+ G)\>>0> \<\xi, F\>$.  Let $L\in \Pic(X)$ with $\<c_1,L\>$ even.
We define
\begin{align*}M_{F_+}^X(c_1,d)&:=M_{n_dF+G}^X(c_1,d), \\
\chi(M_{F_+}^X(c_1,d),\mu(L))&:=\chi(M_{n_d F+G}^X(c_1,d),\mu(L)),\\
\chi^{X,F_+}_{c_1}(L)&:=\sum_{d\ge 0} \chi(M_{F_+}^X(c_1,d),\mu(L))\Lambda^d.\\
\end{align*}
\end{Definition}

Now we give a formula for $\chi^{X,F_+}_{0}(nF+mG)$ and $\chi^{X,F_+}_{F}(nF+mG)$. The rest of this section will be mostly devoted to giving an explicit
evaluation of this formula for $m\le 2$. In work in progress  the  method will be generalised for higher values of $m$ and also to $c_1$ different from $0$ and $F$.

\begin{Proposition}\label{Fplus}
Let $X=\P^1\times\P^1$ or $X=\widehat \P^2$.
\begin{enumerate}
\item
Let $nF+mG$ be a line bundle on $X$ with $m$ even. Then
$$\chi^{X,F_+}_{F}(nF+mG)=\Coeff_{q^0}\left[\frac{1}{2\sinh((m/2+1)h)}\Lambda^2\WT_4(h)^{2(n+2)(m+2)}u'h^*\right].$$

\item
Let $nF+mG$ be a line bundle on $X$. Then
$$\chi^{X,F_+}_{0}(nF+mG)=-\Coeff_{q^0}\left[\frac{1}{2}\left(\coth((m/2+1)h)\right)\Lambda^2\WT_4(h)^{2(n+2)(m+2)}u'h^*\right].$$
\end{enumerate}
%\begin{NB} Here the correction term is still missing.\end{NB}

\end{Proposition}

\begin{proof}
We denote $\Gamma_X=H^2(X,\Z)$ with inner product the negative of the intersection form.
Let $c_1=0$ or $c_1=F$,  fix $d$, and let $s\in \Z_{\ge 0}$ be sufficiently large so that there is no class $\xi$ of $(c_1,d)$ with $\<\xi, F\><0<\<\xi,  G+sF\>$.
Write $L:=nF+mG$.
By definition
\begin{align*}
\chi(M^{sF+G}_X&(F,d),\mu(L))=\Coeff_{\Lambda^d}\Coeff_{q^0}\left[\Psi^{F,G+sF}_{X,F}(L;\Lambda,\tau)\right]\\
&=\Coeff_{\Lambda^d}\Coeff_{q^0}\left[\Theta^{F,G+sF}_{\Gamma_X,F,K_X}(\frac{1}{2\pi i}(L-K_X)h,\tau)\Lambda^2\WT_4(h)^{(L-K_X)^2}u'h^*\right]\\
&=
\Coeff_{\Lambda^d}\Coeff_{q^0}\left[\frac{e^{-\<\frac{F}{2},L-K_X\>}}{1-e^{-\<F,L-K_X\>h}}\Lambda^2\WT_4(h)^{(L-K_X)^2}u'h^*\right]+\sum_{\<F,\xi\>>0>\<G+sF,\xi\>}\delta^X_\xi(L).
\end{align*}
Here the second sum is over the classes of type $(F,d)$. By our assumption on $n$ the second sum is empty, so we get
$$\chi^{X,F_+}_F(L)=\Coeff_{q^0}\left[\frac{e^{-\<\frac{F}{2},L-K_X\>}}{1-e^{-\<F,L-K_X\>h}}\right]=
\Coeff_{q^0}\left[\frac{\Lambda^2\WT_4(h)^{(L-K_X)^2}u'h^*}{2\sinh(\<\frac{F}{2},L-K_X\>h)}\right].$$
In the case $c_1=0$ the argument is very similar. By definition and \thmref{vanbound} we have
\begin{align*}
\chi&(M^{nF+G}_X(0,d),\mu(L))=\Coeff_{\Lambda^d}\Coeff_{q^0}\left[\Psi^{F,sF+G}_{X,0}(L;\Lambda,\tau)\right]\\
&=\Coeff_{\Lambda^d}\Coeff_{q^0}\left[\Theta^{F,sF+G}_{\Gamma_X,0,K_X}(\frac{1}{2\pi i}(L-K_X)h,\tau)\Lambda^2\WT_4(h)^{(L-K_X)^2}u'h^*\right]\\
&=
\Coeff_{\Lambda^d}\Coeff_{q^0}\left[\frac{\Lambda^2\WT_4(h)^{(L-K_X)^2}u'h^*}{1-e^{-\<F,L-K_X\>h}}\right]+\sum_{\<F,\xi\>>0>\<M,\xi\>}\delta^X_\xi(L).
\end{align*}
The second sum is again over the walls of type $(0,d)$, and thus it is $0$.
Thus we get
\begin{align*}
\chi^{X,F_+}_{0}(L)&=\Coeff_{q^0}\left[\frac{\Lambda^2\WT_4(h)^{(L-K_X)^2}u'h^*}{1-e^{-\<F,L-K_X\>h}}\right]\\
&=
-\Coeff_{q^0}\left[\frac{1}{2}\left(\coth(\<F,(L-K_X)/2\>h)+1\right)\Lambda^2\WT_4(h)^{(L-K_X)^2}u'h^*\right].
\end{align*}
Note that by \remref{delb}, we get $$\Coeff_{q^0}[\Lambda^2\WT_4(h)^{(L-K_X)^2}u'h^*]=\Coeff_{q^0}[(1-1)\Lambda^2\WT_4(h)^{(L-K_X)^2}u'h^*]=0.$$
\end{proof}

\begin{Remark}\label{Gplus}
In the case $\P^1\times\P^1$, we can in the same way define $M^{\P^1\times\P^1}_{G_+}(c_1,d):=M^{\P^1\times\P^1}_{n_dG+F}(c_1,d)$ for $n_d$ sufficiently large with respect to $d$, and
$$\chi^{\P^1\times\P^1,G_+}_{c_1}(nF+mG):=\sum_{d>0} \chi(M^{\P^1\times\P^1}_{G_+}(c_1,d),\mu(nF+mG)).$$
Then we see immediately that
$\chi^{\P^1\times\P^1,G_+}_{F}(nF+mG)=0$, and we get by symmetry from \propref{Fplus} that
$$\chi^{\P^1\times\P^1,G_+}_{0}(nF+mG)=-\Coeff_{q^0}\left[\frac{1}{2}\left(\coth((n/2+1)h)\right)\Lambda^2\WT_4(h)^{2(n+2)(m+2)}u'h^*\right].$$
\end{Remark}

\subsection{Theta constant identities}

We use the blowup polynomials $R_n(x,\lambda)$, $S_n(x,\lambda)$ and the blowup functions $\widetilde R_n=R_n(M,\Lambda)$, $\widetilde S_n=S_n(M,\Lambda)$
of \defref{blowpol}
to find
identities between expressions in theta functions, evaluated at some division point. These will then below be used to give recursion formulas for $K$-theoretic
Donaldson invariants of rational ruled surfaces.

\begin{Definition}
Fix $\tau\in \H$.
Let $r\in \Z_{\ge 0}$ and $l\in \Z$ with  $l\equiv r (2)$.
\begin{align*}
RR_{\frac{r-l}{2},\frac{r+l}{2}}(h)&:=
\widetilde R_{\frac{r-l}{2}}(h)-\WT_4(z)^{rl}
 \widetilde R_{\frac{r+l}{2}}(h),\\
 SS_{\frac{r-l}{2},\frac{r+l}{2}}(h)&:=
\widetilde S_{\frac{r-l}{2}}(h)+ e^{\pi i r\Re(h)}\WT_4(h)^{rl}
 \widetilde S_{\frac{r+l}{2}}(h).\\
  \end{align*}
 \end{Definition}

\begin{Proposition} \label{RS0}Fix $r\in \Z_{\ge 0}$, and fix $a\in \Z$.
Let $l\in \Z$ with $l\equiv r (2)$. Then
\begin{enumerate}
\item $RR_{\frac{r-l}{2},\frac{r+l}{2}}(2\pi i \frac{a}{r})=0$,
\item $SS_{\frac{r-l}{2},\frac{r+l}{2}}(2\pi i \frac{a}{r})=0$.
\end{enumerate}
\end{Proposition}
 \begin{proof}
(1)  Put $z_0:=2\pi i \frac{a}{r}$.
As $\theta_4(h)$ is even, we get  $$\theta_4(\frac{r-l}{2}z_0)=\theta_4(-\frac{r-l}{2}z_0)=\theta_4(-2\pi i \frac{r-l}{2r}a)=\theta_4(2\pi i \frac{r+l}{2r}a)=\theta_4(\frac{r+l}{2}z_0),$$ where  we used
 $\theta_4(z+a)=\theta_4(z)$.
 By the definition of $\widetilde R_n$ we get
 \begin{align*}
 \widetilde R_{\frac{r-l}{2}}(z_0)
 &=\frac{\WT_4(\frac{r-l}{2}z_0)}{\WT_4(z_0)^{\frac{(r-l)^2}{4}}}
 = \frac{\WT_4(\frac{r+l}{2}z_0)}{\WT_4(z_0)^{\frac{(r-l)^2}{4}}}=\WT_4(z_0)^{rl} \frac{\WT_4(\frac{r+l}{2}z_0)}{\WT_4(z_0)^{\frac{(r+l)^2}{4}}}
 = \WT_4(z_0)^{rl} \widetilde R_{\frac{r+l}{2}}(z_0).
 \end{align*}
(2) As $\theta_1(z)$ is odd, and  $\theta_1(z+2\pi i a)=(-1)^a\theta_1(z)$, we get
$$\theta_1(\frac{r-l}{2}z_0)=-\theta_1(-\frac{r-l}{2}z_0)=-\theta_1(-2\pi i \frac{r-l}{2r}a)=(-1)^{a+1}\theta_1(2\pi i \frac{r+l}{2r}a)=(-1)^{a+1}\theta_1(\frac{r+l}{2}z_0).$$
By the definition of $\widetilde S_n$ we get
 \begin{align*}
 \widetilde S_{\frac{r-l}{2}}(z_0)
 &=\frac{\WT_1(\frac{r-l}{2}z_0)}{\WT_4(z_0)^{\frac{(r-l)^2}{4}}}
 = (-1)^{a+1}\WT_4(z_0)^{rl} \frac{\WT_1(\frac{r+l}{2}z_0)}{\WT_4(z_0)^{\frac{(r+l)^2}{4}}}
 = (-1)^{a+1}\WT_4(z_0)^{rl} \widetilde S_{\frac{r+l}{2}}(z_0).
 \end{align*}
\end{proof}

 \begin{Proposition}\label{thetanull}
 Let $a\in \Z$.
 \begin{enumerate}
\item Put $Q_2(z):=\widetilde \theta_4(z)^4(1-\Lambda^4)$.
Then
$\big(Q_2(z)-1\big)|_{z=\pi i a}=0.$
\item
Put $Q_3(z):=\widetilde \theta_4(z)^3(1-\Lambda^4)$.
Then
$\big(Q_3(z)-1\big)|_{z=2\pi i \frac{a}{3}}=0.$
\item Put $Q_4(z):=\widetilde \theta_4(z)^8(1-\Lambda^4)^3$. Then
$\big(Q_4(z)-(1-(-1)^a\Lambda^4)\big)|_{z=\pi i \frac{a}{2}}=0.$
\end{enumerate}
\end{Proposition}

\begin{proof}
(1) Put $z_0=\pi i a$. Applying (1) of \propref{RS0}, with $r=2$, $l=2$, we get
$\widetilde R_{0}(z_0)=\WT_4(z_0)^4
 \widetilde R_{2}(z_0)$, and the claim follows because
 $\widetilde R_{0}=1$, $\widetilde R_{2}=(1-\Lambda^4)$.

 (2) Put $z_0=2\pi i \frac{a}{3}$. Applying (1) of \propref{RS0}, with $r=3$, $l=1$, we get
$\widetilde R_{1}(z_0)= \WT_4(z_0)^3
 \widetilde R_{2}(z_0)$, and the claim follows because
 $\widetilde R_{1}=1$, $\widetilde R_{2}=(1-\Lambda^4)$.

 (3) Put $z_0=\pi i \frac{a}{2}$. Apply (1) and (2) of \propref{RS0}, with $r=4$, $l=2$. This gives
 \begin{align}\label{RR13}
 1&=\widetilde R_{1}(z_0)= \WT_4(z)^8\widetilde R_3(z_0),\\
\label{SS13} 1&=\widetilde S_{1}(z_0)=-(-1)^a\WT_4(z_0)^8\widetilde S_3(z_0).
 \end{align}
 We have $\widetilde R_3=-\Lambda^4M^2+(1-\Lambda^4)^2$,
 $\widetilde S_3=M^2-(1-\Lambda^4)^2$. Thus subtracting $(-1)^a\Lambda^4(z_0)$ times \eqref{SS13} from   \eqref{RR13} we get
$$ (1-(-1)^a\Lambda(z_0)^4)=\WT_4(z)^8(1-\Lambda(z_0)^4)^3.$$
  \end{proof}

\subsection{Recursion formulas from theta constant identities}
We now   use the theta constant identities of \propref{RS0} to show recursion formulas in $n$ for the $K$-theoretical Donaldson invariants $\chi^{X,F_+}_{0}(nF+mG)$, $\chi^{X,F_+}_{F}(nF+mG)$ for $0\le m\le 2$ for polarizations near the fibre class.
We consider expressions relating the left hand sides of the formulas of \propref{Fplus} for $\chi^{X,F_+}_{0}(nF+mG)$, $\chi^{X,F_+}_{F}(nF+mG)$ for successive values of $n$.
We show that   the theta constant identities of \propref{RS0} imply that these  expressions
are almost holomorphic in $q$, i.e. that they have only finitely many monomials $\Lambda^d q^s$ with nonzero coefficients and $s\le 0$.
This will then give recursion formulas for  $\chi^{X,F_+}_{0}(nF+mG)$, $\chi^{X,F_+}_{F}(nF+mG)$.

%The following notation will be useful in the rest of this section.
%\begin{Notation}
%For a Laurent series  $F(q,\Lambda),G(q,\Lambda)\in \Q((q,\Lambda))$, and $n\in \Z$ we write $F(q,\Lambda)\equiv G(q,\Lambda)\mod q^{n}$ if the coefficients of all monomials in
%$q$ and $\Lambda$ containing a power of $q$ lower than $n$ of $F$ and $G$ are equal.
%\end{Notation}

\begin{Proposition}\label{theth}
% We have the following
%$$\frac{1}{2\sinh(h)}\widetilde\theta_4(2h)u'h^*\Lambda^2\in \Lambda^4\Q[[q^2\Lambda^2,q^4]])$$
%\begin{NB} this is not true!, Although the coefficient for the coefficient of $q^0$ is. I will need to give a different proof. I think it can be given by wallcrossing. \end{NB}
\begin{align*}
\tag{1}&\widetilde\theta_4(h)= 1+q^2\Lambda^2+O(q^3),\\
\tag{2}&-\frac{1}{2}\coth(h/2)u'h^*\Lambda^2=-\frac{1}{2}q^{-2}\Lambda^2-\Lambda^4+O(q),\\
\tag{3}&\frac{1}{2\sinh(h)}\left(\widetilde\theta_4(h)^4(1-\Lambda^4)-1\right)u'h^*\Lambda^2= \Lambda^4+O(q),\\
\tag{4}&-\frac{1}{2}\coth(h)\left(\widetilde\theta_4(h)^4(1-\Lambda^4)-1\right)u'h^*\Lambda^2 =-\Lambda^4+\frac{1}{2}q^{-2}\Lambda^6+3\Lambda^8+O(q),\\
\tag{5}&-\frac{1}{2}\coth(3h/2)\left(\widetilde\theta_4(h)^3(1-\Lambda^4)-1\right)u'h^*\Lambda^2= -\frac{1}{2}\Lambda^4+\frac{1}{2}q^{-2}\Lambda^6+\frac{5}{2}\Lambda^8+O(q),\\
\tag{6}&-\frac{1}{2}\tanh(h)\left(\widetilde\theta_4(h)^8(1-\Lambda^4)^3-(1+\Lambda^4)\right)u'h^*\Lambda^2\\&\qquad\qquad= 2q^{-2}\Lambda^6+13\Lambda^8-\frac{3}{2}q^{-2}\Lambda^{10}-14\Lambda^{12}+\frac{1}{2}q^{-2}\Lambda^{14}+5\Lambda^{16}+ O(q).
\end{align*}
%$$-\frac{1}{4\tanh(h)}\left(\widetilde\theta_4(h)^8(1-\Lambda^4)^3-(1+\Lambda^4)\right)u'h^*\Lambda^2\in q^{-16}\Q[[q^2\Lambda^2,q^4]])$$
%and its coefficient of $q^0$ is  $\frac{13}{2}\Lambda^8-7\Lambda^{12}+\frac{5}{2}\Lambda^{16}$.
%\begin{NB} check the constants. Done .\end{NB}
\end{Proposition}
\begin{proof} (1) was already shown in \lemref{Rregular}(4).

(2) In \lemref{Rregular} it was shown that $\frac{1}{\sinh(h/2)}\in q\Lambda^{-1}\RR$, and that  $\cosh(h/2)h^*\in q^{-1}\Lambda\RR$.
It is also easy to see that the coefficient of $q\Lambda^{-1}$ of $\frac{1}{\sinh(h/2)}$ is $-2i$.
As $u\in q^{-2}\Q[[q]]$, we find therefore that
$$\frac{\cosh(h/2)}{2\sinh(h/2)}u'h^*\Lambda^2\in q^{-2}\Lambda^2\RR,$$ and explicit computation with lower order coefficients of  the power series involved determines the coefficients of degree at most $0$ in $q$. This shows (2).

By definition it is easy to see that $\WT_1(h)\in (y-y^{-1})q\Q[[y^{\pm2}q^4,q^4]],$ and we get
 \begin{equation}\label{WTser}\quad \WT_4(h)\in  \Q[[y^{\pm2}q^4,q^4]], \quad \Lambda^4\in \Q[y^2,y^{-2}]_{\le 1}\Q[[y^{\pm2}q^4,q^4]]
 .\end{equation}

 (3)  Let $$f(y,q)=\sum_{n\ge 0} f_n(y)q^{4n}=\WT_4(h)^4(1-\Lambda^4)-1.$$
% \begin{NB} No need to argue with $\widetilde R_2$.\end{NB}
% Note that by \eqref{RSmod} and \propref{rnprop} we have $\WT_4(h)^4(1-\Lambda^4)=\WT_4(2h)$. Thus by \eqref{WTser} we get
 %$f(y,q)\in -(y^2-y^{-2})^2q^4+q^4\Q[[y^{\pm4}q^4,q^4]].$
 By the above $f(y,q)\in  \Q[y^2,y^{-2}]_{\le 1}\Q[[y^{\pm2}q^4,q^4]]$.
 Furthermore $f(y,q)$ is symmetric under
 $y \leftrightarrow y^{-1}$.
By \propref{thetanull}(1) the function $f(y,q)$ vanishes (identically in $q$) for $2h\in 2\pi i \Z$, i.e. if $y^4=1$, i.e. for $y^2=y^{-2}$.
Thus every coefficient $f_n(y)$  is as a symmetric Laurent polynomial in $y^2$ divisible by $y^2-y^{-2}$.
Let $$g(y,q)=\sum_{n\ge 1} g_n(y)q^{4n}:=\frac{1}{2\sinh(h)}(\WT_4(h)^4(1-\Lambda^4)-1).$$
Then $g_1=-(y^2-y^{-2})$ and $g_n$ an antisymmetric Laurent polynomial in $y^2$ of degree at most $n$.
Thus for all $n$  we get that $g_n(y)$ is a $\Q$-linear combination of $\sinh(kh)$ with $k=1,\ldots,n$. Therefore we get by \lemref{Rregular} that
\begin{equation}\label{gser}
g_n(y)h^*u'\Lambda^2 \in  \Q[q^{-2}\Lambda^2]_{\le n+1}\RR.
\end{equation}
Thus the only possible monomial $\Lambda^mq^n$ with non vanishing coefficient of
\begin{align*}g(y,q)u'h^*\Lambda^2=
\sum_{n} g_n(y)h^*u'\Lambda^2 q^{4n},\end{align*}
  with nonpositive power of $q$ is $\Lambda^4 q^0$, and direct simple computation gives that its coefficient is $1$.
%\begin{align*}g(y,q)u'h^*\Lambda^2=
%\sum_{n} g_n(y)h^*u'\Lambda^2 q^{4n}\in \Lambda^4\Q[[q^2\Lambda^2,q^4]],\end{align*}
%and $\Coeff_{q^0}\big[g(y,q)u'h^*\Lambda^2\big]=\Coeff_{q^0}\big[g_1(y)q^4h^*u'\Lambda^2\big]=\Lambda^4$, where the last step is a simple direct computation with the series.
This shows (3)

(4) Note that $$-\frac{1}{2}\coth(h)(\WT_4(h)^4(1-\Lambda^4)-1)=-\cosh(h)g(y,q).$$
By \lemref{Rregular} we have $\cosh(h)\in 1+(q^{-2}\Lambda^2)_{\le 1}\RR$.
By \eqref{gser} this implies that the only monomials in $\Lambda$, $q$
with non vanishing coefficients and nonpositive degree in $q$ in $-\cosh(h)g(y,q)$ are $\Lambda^4q^0$, $\Lambda^6q^2$, $\Lambda^8q^0$.
A simple direct computation with the leading terms of the Laurent series involved determines the coefficients of these monomials.

(5) Let
$$f^1(y,q)=\sum_{n>0} f^1_n(y)q^{4n}=\WT_4(h)^3(1-\Lambda^4)-1.$$
By \eqref{WTser} we get that
$f^1(y,q)\in \Q[y^2,y^{-2}]_{\le 1}\Q[[y^{\pm2}q^4,q^4]].$ Again we see that every $f^1_n(y)$ is a symmetric Laurent polynomial in $y^2$.
By \propref{thetanull} we have that $f^1(y,q)$ vanishes (identically in $q$) for $3h\in 2\pi i \Z$,  i.e. for $y^3=y^{-3}$.
Thus every coefficient $f^1_n(y)$  is a symmetric Laurent polynomial in $y$ divisible by $y^3-y^{-3}$.
Let
\begin{align}\label{g1ser}
g^1(y,q)&=\sum_n g^1_n(y)q^{4n}:=\frac{1}{2}\coth(3h/2)f^1(y,q)%\in  -(y-y^{-1})(y^3+y^{-3})q^4-(y-y^{-1})(y^2+y^{-2})(y^3+y^{-3})q^8+\Q[[y^{\pm2}q^4,q^4]]_{\ge 3}.
\end{align}
Then all $g^1_n(y)=\frac{y^3+y^{-3}}{2(y^3-y^{-3})}f^1_n(y)$ are antisymmetric Laurent polynomials in $y^2$ of degree at most $n+1$.
Thus all $g^1_n(y)$ are linear combinations of $\sinh(kh)$ for $k=1,\ldots,n+1$.
Therefore we get by \lemref{Rregular} that $$g^1_n(y)h^*u'\Lambda^2\in \Q[q^{-2}\Lambda^2]_{\le n+2}\Q[[\Lambda^2q^2,q^4]].$$
Thus we get by \eqref{g1ser} that the only monomials
$\Lambda^mq^n$ with non vanishing coefficients in
$g^1(y,q)u'h^*\Lambda^2$ are $\Lambda^4$, $q^{-2}\Lambda^6$, $\Lambda^8$, and their coefficients are again determined by a simple direct computation.

(6) Let
$$f^2(y,q)=\sum_{n>0} f^2_n(y)q^{4n}=\WT_4(h)^8(1-\Lambda^4)^3-(1+\Lambda^4).$$
By \eqref{WTser} we get that $f^2(y,q)=\Q[y^{2},y^{-2}]_{\le 3}\Q[[y^{\pm2}q^4,q^4]]$.
%Each coefficient $f^2_n(y)$ is % By \propref{thetanull} $f^2(y,q)$ vanishes (identically in $q$) for $4h\in 2\pi i \Z$, but $2h\not \in  2\pi i \Z$.
Thus again every coefficient $f^2_n(y)$ is a symmetric Laurent polynomial in $y^2$.
By \propref{thetanull} $f^2(y,q)$ vanishes (identically in $q$) for $4h\in 2\pi i \Z$, but $2h\not \in  2\pi i \Z$.
Thus every $f^2_n(y)$  is a symmetric Laurent polynomial in $y^2$ divisible   by $\frac{y^4-y^{-4}}{y^2-y^{-2}}=y^2+y^{-2}$.
Let
$$g^2(y,q):=\sum_n g^2_n(y)q^{4n}:=\frac{1}{2}\tanh(h)f^2(y,q).$$
Then all $g^2_n(y)$ are antisymmetric Laurent polynomials in $y^2$ of degree  at most   $n+3$.
Thus
$g^2_n(y)h^*u'\Lambda^2\in \Q[q^{-2}\Lambda^2]_{\le n+4}\RR$.
By the definition of $g^2(y,q)$ this gives that the only monomials
$\Lambda^mq^n$ with non vanishing coefficients in
$g^2(y,q)u'h^*\Lambda^2$ are $\Lambda^4$, $q^{-2}\Lambda^6$, $q^{-4}\Lambda^8$, $\Lambda^8$, $q^{-6}\Lambda^{10}$, $q^{-2}\Lambda^{10}$, $q^{-4}\Lambda^{12}$, $\Lambda^{12}$, $q^{-2}\Lambda^{14}$, $\Lambda^{16}$, and their coefficients are again determined by a simple direct computation.
%and $g^1_1(y)=-(y-y^{-1})(y^3+y^{-3})$, $g^1_2(y)=-(y-y^{-1})(y^2+y^{-2})(y^3+y^{-3})$, and
%$g^1_n(y)$ has degree at most $n$ in $y^2$.
%\begin{NB}
%{\bf Outline of the steps of the proof.}
%\begin{enumerate}
%\item Formula (1) is easy and already  proven. Now also (2) is essentially proven.
% The rest of the argument is for formulas (3)-(6).
%\item Notice that the expressions $F$ on the left hand side e.g.
%$$\left(\widetilde\theta_4(h)^4(1-\Lambda^4)-1\right), \quad \left(\widetilde\theta_4(h)^3(1-\Lambda^4)-1\right)$$
%are elements in
%$\Q[y^{\pm 1}][[q]]$, i.e. they are Laurent polynomials in $y$, (anti)-symmetric in $y$.
%\item Show that the fact that they vanish at suitable division points implies that they are divisible by certain Laurent polynomials $f(y)$ in $y$
%(e.g. by $2\sinh(z)=y-y^{-1}$).
%in the sense that all the coefficients of powers of $q$ which are Laurent polynomials in $y$ are divisible $f(y)$.
%This fact has been also numerically checked, but if I am not mistaken the proof is also quite simple.
%\item Show using arguments as in the proof of \thmref{vanwall} that this implies that after the variable transformation to $\Lambda$ we have that
%$F/f(y)$ is almost holomorphic in $q$, which implies the result. (This is quite straightforward).
%\end{enumerate}
%\end{NB}
\end{proof}

\begin{Proposition}\label{p110}  For $X=\P^1\times\P^1$ or $X=\widehat \P^2$ and all $n\in \Z$ we have
\begin{enumerate}
\item
$\displaystyle{1+\chi^{X,F_+}_{F}(nF)=\frac{1}{(1-\Lambda^4)^{n+1}}}$,
\item
$\displaystyle{1+(2n+5)\Lambda^4+\chi^{X,F_+}_{0}(nF)=\frac{1}{(1-\Lambda^4)^{n+1}}}$.
\end{enumerate}
\end{Proposition}

\begin{proof}%\begin{NB} Clarify the overall sign?\end{NB}
 (1)  By \propref{Fplus} we need to show
\begin{equation}\label{sifo}
1+\coeff_{q^0}\Big[\frac{1}{2\sinh(h)} \Big(\frac{\theta_4(h)}{\theta_4}\Big)^{4(n+2)}u'h^*\Lambda^2\Big]=
\frac{1}{(1-\Lambda^4)^{n+1}}.
\end{equation}

The proof is by both descending and ascending induction on $n\in \Z$.
We first study the case $n=-1$.

By \propref{Fplus} we know that
$$\chi^{\P^1\times\P^1,F_+}_{F}(-F)=\Coeff_{q^0}\left[\frac{1}{2\sinh(h)}\Lambda^2\WT_4^4u'h^*\right].$$
On the other hand
$\chi^{\P^1\times\P^1,G}_{F}(-F)=0$ and thus
$$\chi^{\P^1\times\P^1,F_+}_{F}(-F)=\sum_{\<\xi, F\>>0>\<\xi, G\>} \delta_\xi^{\P^1\times\P^1}(-F),$$ where the sum is over classes of type $(F)$, i.e. over all
$\xi=-(2n-1)F+2mG$ with $n,m\in \Z_{\ge 0}$.
By \thmref{vanwall} we get that  $\delta_{-(2n-1)F+2mG}^{\P^1\times\P^1}(-F)=0$ unless
$8nm-4m\le |4n-2-2m|+2$. This means in case $4n-2-2m\ge 0$ that  $8nm-4m\le 4n-2m$, which is impossible, and in case
$4n-2-2m\le 0$ that $8nm-4m\le 2m-4n+4$, which is also impossible. Thus $\chi^{\P^1\times\P^1,F_+}_{F}(-F)=0$ and therefore also
$\Coeff_{q^0}\left[\frac{1}{2\sinh(h)}\Lambda^2\WT_4^4u'h^*\right]=0$. This shows the case $n=-1$.

By \lemref{Rregular} we have $\WT_4(h)=1 +O(q)$. Thus we get by \propref{theth}(3), for all $n\in \Z$ that
\begin{align*}\frac{1}{2\sinh(h)}&\big(\WT_4(h)^{4(n+2)}(1-\Lambda^4)-\WT_4(h)^{4(n+1)}\big)u'h^*\Lambda^2\\&=
\WT_4(h)^{4(n+1)}\frac{1}{2\sinh(h)}\big(\WT_4(h)^{4}(1-\Lambda^4)-1\big)u'h^*\Lambda^2 = \Lambda^4+O(q).
\end{align*}
Thus, writing $$f_n:=1+\coeff_{q^0}\Big[\frac{1}{2\sinh(h)}\WT_4(h)^{4(n+2)}u'h^*\Lambda^2\Big],$$ we have
$(1-\Lambda^4)f_n-f_{n-1}=\Lambda^4+(1-\Lambda^4)-1=0$, i.e. $f_{n}(1-\Lambda^4)=f_{n-1}$ for all $n\in \Z$.  We have shown above that $f_{-1}=1$.
Thus the claim follows.

(2) By \propref{Fplus} we know that
$$\chi^{\P^1\times\P^1,F_+}_{0}(-F)=\Coeff_{q^0}\left[-\frac{1}{2}\coth(h)\Lambda^2\WT_4^4u'h^*\right].$$
By  \propref{theth}(1) we have $\WT_4(h)= 1+q^2\Lambda^2+O(q^3)$. Thus \propref{theth}(2) gives
$\Coeff_{q^0}\left[-\frac{1}{2}\coth(h/2)\Lambda^2\WT_4^4u'h^*\right]=-3\Lambda^4.$
Thus we get by \remref{Gplus} that
$$\chi^{\P^1\times\P^1,G_+}_{0}(-F)=\Coeff_{q^0}\left[-\frac{1}{2}\coth(h/2)\Lambda^2\WT_4^4u'h^*\right]=-3\Lambda^4.$$
This gives
$$\chi^{\P^1\times\P^1,F_+}_{0}(-F)=\chi^{\P^1\times\P^1,F_+}_{0}(-F)-\chi^{\P^1\times\P^1,G_+}_{0}(-F)-3\Lambda^4=\sum_{\<\xi, F\>>0>\<\xi, G\>} \delta_\xi^{\P^1\times\P^1}(-F)-3\Lambda^4,$$
where the sum is over classes of type $(0)$, i.e. over all
$\xi=-2nF+2mG$ with $n,m\in \Z_{\ge 0}$.
By \thmref{vanwall} we get that  $\delta_{-2nF+2mG}^{\P^1\times\P^1}(-F)=0$ unless
$8nm\le |4n-2m|+2$, which is impossible. Thus $\chi^{F_+}_{\P^1\times\P^1,F}(-F)=-3\Lambda^4$ and therefore also
$\Coeff_{q^0}\left[-\frac{1}{2}\coth(h)\Lambda^2\WT_4^4u'h^*\right]=-3\Lambda^4$.

%By \lemref{Rregular} we have 
$\WT_4(h)= 1+q^2\Lambda^2+O(q^3).$  Thus we get by \propref{theth}(2), for all $n\in \Z$ that
\begin{align*}-\frac{1}{2}\coth(h)&\big(\WT_4(h)^{4(n+2)}(1-\Lambda^4)-\WT_4(h)^{4(n+1)}\big)u'h^*\Lambda^2\\&=
-\WT_4(h)^{4(n+1)}\frac{1}{2}\coth(h)\big(\WT_4(h)^{4}(1-\Lambda^4)-1\big)u'h^*\Lambda^2 = \Lambda^4+(2n+5)\Lambda^8+O(q).
\end{align*}
We put $g_n:=1+(2n+5)\Lambda^4+\chi^{X,F_+}_{0}(nF)$. Then we have by \propref{Fplus} that
$$g_n=1+(2n+5)\Lambda^4+\Coeff_{q^0}\left[-\frac{1}{2}\coth(h)\WT_4(h)^{4(n+2)}u'h^*\Lambda^2\right].$$
We have by the above $g_{-1}= 1$.
We get for all $n\in \Z$ that
\begin{align*}
(1-\Lambda^4)g_n-g_{n-1}&=(1+(2n+5)\Lambda^4)(1-\Lambda^4)-(1+(2n+3)\Lambda^4)-\Lambda^4+(2n+5)\Lambda^8
=0,\end{align*} i.e. $(1-\Lambda^4)g_n=g_{n-1}$. By  induction over $n\in \Z$ this gives $g_n=\frac{1}{(1-\Lambda^4)^{n+1}}.$
\end{proof}

\begin{Proposition}\label{p11G}  For $X=\P^1\times\P^1$ and $n\in \Z$, and for $X=\widehat \P^2$ and $n\in \Z+\frac{1}{2}$ we have
\begin{align*}
1+(3n+7)\Lambda^4+\chi^{X,F+}_{0}(nF+G)&=\frac{1}{(1-\Lambda^4)^{2n+2}}.\\
\end{align*}
\end{Proposition}
\begin{proof}
We will treat the cases of $\P^1\times\P^1$ and $\widehat \P^2 $ together and prove the result by induction over $n\in \frac{1}{2}\Z$.
By \propref{Fplus} we have for $n\in \frac{1}{2}\Z$ that
$$\chi^{X,F_+}_{0}(nF+G)=\coeff_{q^0}\left[
-\frac{1}{2}\coth(3h/2)\WT_4(h)^{6(n+2)}u'h^*\Lambda^2\right].$$
Here $X=\P^1\times\P^1$ if $n\in \Z$ and $\widehat \P^2$ otherwise.  %\begin{NB} there will be a small correction term\end{NB}
For $n\in \frac{1}{2}\Z$ let $$h_n:=1+(3n+7)\Lambda^4+\chi^{X,F_+}_{0}(nF+G)=\coeff_{q^0}\left[
-\frac{1}{2}\coth(3h/2)\WT_4(h)^{6(n+2)}u'h^*\Lambda^2\right]+1+(3n+7)\Lambda^4.$$
We want to show by induction on $n\in \frac{1}{2}\Z$ that
$h_n=\frac{1}{(1-\Lambda^4)^{2n+2}}$.

{\bf Case $n=0$.} By \propref{p110} and symmetry, we have
$1+7\Lambda^4+\chi^{\P^1\times\P^1,G_+}_{0}(G)=
\frac{1}{(1-\Lambda^4)^2}.$
Thus $$1+7\Lambda^4+\chi^{\P^1\times\P^1,F_+}_{0}(G)=
\frac{1}{(1-\Lambda^4)^2}+\sum_{\xi} \delta_{\xi}^{\P^1\times\P^1}(G),$$
where $\xi$ runs over all $-2nF+2mG$ with $n,m\in \Z_{>0}$.
By \thmref{vanwall} we have that $\delta_{-2nF+2mG}^{\P^1\times\P^1}(G)=0$, unless
$8nm\le |4m-6n|+2$, which is impossible. Thus $$g_0=1+7\Lambda^4+\chi^{\P^1\times\P^1,G_+}_{0}(G)=
\frac{1}{(1-\Lambda^4)^2}.$$

{\bf Induction step.} By \propref{theth} we have
$$-\frac{1}{2}\coth(3h/2)\left(\widetilde\theta_4(h)^3(1-\Lambda^4)-1\right)u'h^*\Lambda^2= -\frac{1}{2}\Lambda^4+\frac{1}{2}q^{-2}\Lambda^6+\frac{5}{2}\Lambda^8+ O(q).$$
Using also $\WT_4(h)=1+q^2\Lambda^2+O(q^3)$ we get
\begin{align*}
-\coeff_{q^0}\big[\frac{1}{2}\coth(3h/2)&\Big(\WT_4(h)^{6(n+2)}(1-\Lambda^4)-\WT_4(h)^{6(n+3/2)}\big)u'h^*\Lambda^2\big]&=1/2\Lambda^4+(3n+7)\Lambda^8.
\end{align*}
Thus we get
$$h_n(1-\Lambda^4)-h_{n-1/2}=\frac{1}{2}\Lambda^4+(3n+7)\Lambda^8+(1+(3n+7)\Lambda^4)(1-\Lambda^4) -(1+(3n+\frac{11}{2})\Lambda^4)=0.$$
Thus by induction $h_n=\frac{1}{(1-\Lambda^4)^{2n+2}}$.
\end{proof}

%\begin{Conjecture}\label{conjG2}
%$$\frac{1}{\cosh(h)}\Big(\Big(\frac{\theta_4(h)}{\theta_4}\Big)^8
%(1-\Lambda^4)^3-(1+\Lambda^4)\Big)\in q^{-6}\Lambda^2\Q[[q^2\Lambda^2,q^4]]\cap q\Q[[q,\Lambda]].$$
%\end{Conjecture}

\begin{Proposition}\label{p112G}
Let $X=\P^1\times\P^1$ or $X=\widehat \P^2$.
\begin{enumerate}
\item
 For all $n\in \Z$
$$\sum_{d}\chi(M_{F_+}^{X}(F,d),nF+2G)\Lambda^{d}=
\frac{1}{2}\frac{(1+\Lambda^4)^n-(1-\Lambda^4)^n}{(1-\Lambda^4)^{3n+3}}.$$
\item  For all $n\in \Z$:
$$1+(4n+9)\Lambda^4+\sum_{d}\chi(M_{F+}^{X}(0,d),nF+2G)\Lambda^{d}=
\frac{1}{2}\frac{(1+\Lambda^4)^n+(1-\Lambda^4)^n}{(1-\Lambda^4)^{3n+3}}.$$
\end{enumerate}
\end{Proposition}
\begin{proof}
(1) By \propref{Fplus}, we have
\begin{align*}\chi^{X,F_+}_{F}(nF+2G)=
\coeff_{q^0}\Big[\frac{1}{2\sinh(2h)}\WT_4(h)^{8(n+2)}u'h^*\Lambda^2\Big].
\end{align*}
Note that
$\frac{1}{2\sinh(2h)}=\frac{1}{4\sinh(h)\cosh(h)}=
\frac{1}{4}\big(\coth(h)-\tanh(h)\big).$
By  \propref{Fplus} and \propref{p110}, we have
\begin{equation}\label{coth}\begin{split}
\coeff_{q^0}\Big[\frac{1}{2}\coth(h)\WT_4(h)^{8(n+2)}u'h^*\Lambda^2\Big] &=-\chi^{X,F_+}_{0}((2n+2)F)\\&=-\frac{1}{(1-\Lambda^4)^{2n+3}}+(1+(4n+9)\Lambda^4).
\end{split}
\end{equation}

We will show by induction on $n\in \Z$ that
\begin{equation}\label{tanh}
l_n:=1+(4n+9)\Lambda^4-\coeff_{q^0}\big[\frac{1}{2}\tanh(h)\WT_4(h)^{8(n+2)}u'h^*\Lambda^2\big]=\frac{(1+\Lambda^4)^n}{(1-\Lambda^4)^{3n+3}}.
\end{equation}
(1) follows directly from  \eqref{coth} and\eqref{tanh}:
\begin{align*}
&\chi^{X,F_+}_F(nF+2G)=\coeff_{q^0}\frac{1}{4}\big(\coth(h)-\tanh(h)\big)\WT_4(h)^{8(n+2)}u'h^*\Lambda^2\Big]\\&=
\frac{1}{2}\left(-\frac{1}{(1-\Lambda^4)^{2n+3}}+1+(4n+9)\Lambda^4+\frac{(1+\Lambda^4)^n}{(1-\Lambda^4)^{3n+3}}-1-(4n+9)\Lambda^4\right)=\frac{1}{2}\frac{(1+\Lambda^4)^n-(1-\Lambda^4)^n}{(1-\Lambda^4)^{3n+3}}.
\end{align*}

{\bf Case $n=0$.} We have
$\chi^{\P^1\times\P^1,G_+}_{F}(2G)=0.$
Thus $$\chi^{\P^1\times\P^1,F_+}_{F}(2G)=
\sum_{\xi} \delta_{\xi}^{\P^1\times\P^1}(2G),$$
where $\xi$ runs over all $-(2n-1)F+2mG$ with $n,m\in \Z_{>0}$.
By \thmref{vanwall} we have that $\delta_{-(2n-1)F+2mG}^{\P^1\times\P^1}(G)=0$, unless
$8nm-4m\le |4m-8n+4|+2$, which is impossible. Thus by the above
$$0=\chi^{\P^1\times\P^1,F_+}_{F}(2G)=\coeff_{q^0}\Big[\frac{1}{4}(\coth(h)-\tanh(h))\WT_4(h)^{16}u'h^*\Lambda^2\Big].$$
This gives
$$-\coeff_{q^0}\Big[\frac{1}{2}\tanh(h)\WT_4(h)^{16}u'h^*\Lambda^2\Big]=-\coeff_{q^0}\Big[\frac{1}{2}\coth(h)\WT_4(h)^{16}u'h^*\Lambda^2\Big]
=\frac{1}{(1-\Lambda)^{3}}-1-9\Lambda^4.$$

{\bf Induction step.}
By  \propref{theth} we have
\begin{align*}-\frac{1}{2}\tanh(h)\big(\WT_4(h)^8&(1-\Lambda^4)^3-(1+\Lambda^4)\big)h^*u'\Lambda^2\\&= 2q^{-2}\Lambda^6+13\Lambda^8-\frac{3}{2}q^{-2}\Lambda^{10}-14\Lambda^{12}+\frac{1}{2}q^{-2}\Lambda^{14}+5\Lambda^{16}+O(q).\end{align*}
Using again that $\WT_4(h) = 1+q^2\Lambda^2+O(q^3)$, this gives
\begin{align*}
-\coeff_{q^0}\Big[\frac{1}{2}\tanh(h)&\big(\WT_4(h)^{8(n+2)}-\WT_4(h)^{8(n+1)}\big)u'h^*\Lambda^2\Big]\\&=(16n + 29)\Lambda^8-(12n+26)\Lambda^{12}+(4n + 9)\Lambda^{16}.
\end{align*}
Again one checks that this gives
$(1-\Lambda^4)^3l_n-(1+\Lambda^4)l_{n-1}=0$.
This shows (1).

(2) By \propref{Fplus}, we have
\begin{align*}\chi^{X,F_+}_0(nF+2G)=
-\coeff_{q^0}\Big[\frac{1}{2}\coth(2h)\WT_4(h)^{8(n+2)}u'h^*\Lambda^2\Big],
\end{align*}
Using
$-\frac{1}{2}\coth(2h)=
\frac{1}{4}\big(-\coth(h)-\tanh(h)\big)$ we get from \eqref{coth} and \eqref{tanh} that
$$\chi^{X,F_+}_0(nF+2G)=\frac{1}{2}\left(\frac{1}{(1-\Lambda^4)^{2n+3}}+\frac{(1+\Lambda^4)^{n}}{(1-\Lambda^4)^{3n+3}}\right)-1-(4n+9)\Lambda^4,$$ and the claim follows.
\end{proof}

\begin{proof}[Proof of \thmref{p11t}]
This follows directly from \propref{p110}, \propref{p11G}, \propref{p112G} and \propref{ruledconst} and \defref{KdonGen}.
\end{proof}

\subsection{Blowdown and the $K$-theoretic Donaldson invariants of $\P^2$}

Now we will prove \thmref{P22}, by combining the formulas \thmref{p11t} for $\widehat \P^2$ with the
blowup formulas \lemref{blowsimple} and \lemref{1blow}.

\begin{proof}[Proof of \thmref{P22}]
(1) By \thmref{p11t} we have
\begin{align*}\chi^{\widehat\P^2,H_+}_{0}(H)=\chi^{\widehat\P^2,H_+}_{0}(\frac{1}{2}F+G)=\frac{1}{(1-\Lambda^4)^3}-1-\frac{17}{2}\Lambda^4.
\end{align*}
We have $\chi^{\widehat \P^2,H}_{0}(H)=\chi^{\widehat \P^2,H_+}_{0}(H)-\frac{1}{2}\sum_{\xi} \delta^{\widehat \P^2}_{\xi}(H)$, where $\xi$ runs over all classes of class $0$ with
$0=\<\xi,H\>>\<\xi, E\>$. These are all the $\xi=2nE$ with $n\in \Z_{\ge 0}$.
By \thmref{vanwall} we have that $\delta^{\widehat \P^2}_{2nE}(H)=0$ unless $4n^2\le |2n|+2$. This is only possible for $n=1$ and direct computation gives
$\delta^{\widehat \P^2}_{2E}(H)=\Lambda^4$.
Thus $$\chi^{\widehat \P^2,H}_{0}(H)=\frac{1}{(1-\Lambda^4)^3}-1-9\Lambda^4.$$

On the other hand we have by  \lemref{blowsimple}
$\chi^{\P^2,H}_{0}(H)=\chi^{\widehat \P^2,H_+}_{0}(H)$.

Alternatively observe that by  \thmref{p11t}
\begin{align*}\chi^{\widehat\P^2,H_+}_{0}(H-E)=\chi^{\widehat\P^2,H_+}_{0}(F)=\frac{1}{(1-\Lambda^4)^2}-1-7\Lambda^4,
\end{align*}
As above
$$\chi^{\widehat \P^2,H}_{0}(H-E)=\chi^{\widehat \P^2,H_+}_{0}(H-E)-\frac{1}{2}\sum_{n>0} \delta^{\widehat \P^2}_{2nE}(H-E).$$
By \thmref{vanwall} we have that $\delta^{\widehat \P^2}_{2nE}(H-E)=0$ unless $4n^2\le |4n|+2$. This is only possible for $n=1$ and direct computation gives
$\delta^{\widehat \P^2}_{2E}(H-E)=2\Lambda^4-18\Lambda^8$.
Thus we get
by \lemref{1blow}
\begin{equation}\label{P0H}
\chi^{\P^2,H}_{0}(H)=\frac{\chi^{\widehat P^2,H}_{0}(H-E)-\Lambda^4+9\Lambda^8}{(1-\Lambda^4)}=\frac{1}{(1-\Lambda^4)^3}-1-9\Lambda^4.
\end{equation}
The result follows because $\Coeff_{\Lambda^d}\big[\chi^{\P^2,H}_{0}(H)\big]=\chi(M_{H}^{\P^2}(0,d),\mu(H)))$ for $d>4$.

(2)
By \thmref{p11t} we have
\begin{align*}\chi^{\widehat\P^2,H_+}_{0}(2H)=\chi^{\widehat\P^2,H_+}_{0}(F+2G)=\frac{1}{(1-\Lambda^4)^6}-1-13\Lambda^4,
\end{align*}
As above, using again \thmref{vanwall}
$$\chi^{\widehat \P^2,H}_{0}(2H)=\chi^{\widehat \P^2,H_+}_{0}(2H)-\frac{1}{2}\sum_{n>0} \delta^{\widehat \P^2}_{2nE}(2H)=\chi^{\widehat \P^2,H_+}_{0}(2H)-\frac{1}{2}\delta^{\widehat \P^2}_{2E}(2H),$$ and
$\delta^{\widehat \P^2}_{2E}(2H)=\Lambda^4$.
Thus   \lemref{blowsimple} gives
\begin{equation}\label{P02H}
\chi^{\P^2,H}_{0}(2H)=\frac{1}{(1-\Lambda^4)^6}-1-\frac{27}{2}\Lambda^4.
\end{equation}
An alternative proof is by observing that
\begin{align*}\chi^{\widehat\P^2,H_+}_{0}(2H-E)=\chi^{\widehat\P^2,H_+}_{0}(\frac{3}{2}F+G)=\frac{1}{(1-\Lambda^4)^5}-1-\frac{23}{2}\Lambda^4,
\end{align*}
Using \thmref{vanwall} this gives
$$\chi^{\widehat \P^2,H}_{0}(2H-E)=\chi^{\widehat \P^2,H_+}_{0}(2H-E)-\frac{1}{2}\sum_{n>0} \delta^{\widehat \P^2}_{2nE}(2H-E)=\chi^{\widehat \P^2,H_+}_{0}(2H-E)-\frac{1}{2}\delta^{\widehat \P^2}_{2E}(2H-E),$$ and
$\delta^{\widehat \P^2}_{2E}(2H-E)=2\Lambda^4-27\Lambda^8$.
Thus  \lemref{1blow} gives again
$\chi^{\P^2,H}_{0}(2H)=\frac{1}{(1-\Lambda^4)^6}-1-\frac{27}{2}\Lambda^4.$

(3) By \thmref{p11t} and  \lemref{blowsimple} we have
\begin{equation}\label{PH2H}
\Lambda \chi^{\P^2,H}_{H}(2H)=\chi^{\widehat\P^2,H_+}_{F}(2H)=\chi^{\widehat\P^2,H_+}_{F}(F+2G)=\frac{\Lambda^4}{(1-\Lambda^4)^6},
\end{equation}
The result follows because $\Coeff_{\Lambda^d}\big[\chi^{\P^2,H}_{H}(2H)\big]=\chi(M_{H}^{\P^2}(0,d),\mu(H)))$ for $d>0$.

(4) By \thmref{p11t} we have
\begin{align*}\chi^{\widehat\P^2,H_+}_{0}(3H-E)=\chi^{\widehat\P^2,H_+}_{0}(2F+2G)=\frac{1+\Lambda^8}{(1-\Lambda^4)^9}-1-17\Lambda^4.
\end{align*}
Using \thmref{vanwall} this gives
$$\chi^{\widehat \P^2,H}_{0}(3H-E)=\chi^{\widehat \P^2,H_+}_{0}(3H-E)-\frac{1}{2}\sum_{n>0} \delta^{\widehat \P^2}_{2nE}(3H-E)=\chi^{\widehat \P^2,H_+}_{0}(3H-E)-\frac{1}{2}\delta^{\widehat \P^2}_{2E}(3H-E),$$ and
$\delta^{\widehat \P^2}_{2E}(3H-E)=2\Lambda^4-38\Lambda^8$, this gives
by  \lemref{1blow}
\begin{equation}\label{P03H}\chi^{\P^2,H}_{0}(3H)=\frac{1+\Lambda^8}{(1-\Lambda^4)^{10}}-1-19\Lambda^4.
\end{equation}
\end{proof}

\subsection{Some further invariants of the blowup of the plane}

In this subsection we apply the blowup formula \lemref{blowsimple}
to the results of the previous section to obtain $K$-theoretic invariants with respect to first Chern class $H$ or $E$.

\begin{Corollary}
\begin{enumerate}
\item
For $P=aH+bF$ with $\frac{b}{a}<2$ we have
\begin{align*}
 \Lambda+\sum_{d>4}\chi(M^{\widehat \P^2}_P(E,d),\mu(H))\Lambda^d&=\frac{\Lambda}{(1-\Lambda^4)^3},\\ \Lambda+\sum_{d>4}\chi(M^{\widehat \P^2}_P(E,d),\mu(2H))\Lambda^d&=\frac{\Lambda}{(1-\Lambda^4)^6},\\
\Lambda+\sum_{d>4}\chi(M^{\widehat \P^2}_P(E,d),\mu(3H))\Lambda^d&=\frac{\Lambda+\Lambda^9}{(1-\Lambda^4)^{10}}.\\
\end{align*}
\item For $P=aH+bF$ with $\frac{b}{a}<1$ we have
\begin{align*}
\sum_{d>0}\chi(M^{\widehat \P^2}_P(H,d),\mu(2H))\Lambda^d&=\frac{\Lambda^3}{(1-\Lambda^4)^6}.\\
\end{align*}
\end{enumerate}
\end{Corollary}
\begin{proof}
(1)
Let $\xi\in E+2H^2(X,\Z)$ with $\< H,\xi\>>0> \<F,\xi\>$ and with $\delta^{\widehat \P^2}_\xi(kH)\ne 0$ for some $k$ with $1\le k\le 3$.
Then \thmref{vanwall} gives that $\xi=2nF-(2m-1)E$ with $n,m\in \Z_{>0}$ and
$8nm-4n+(2m-1)^2\le |(4+2k)n-2m+1|+2$. Thus either $2m+1-2kn\ge 8mn+(2m-1)^2$,which is impossible, or
$2kn-2m+3\ge 8nm-8n+(2m-1)^2$, which implies $m=1$. Thus if $P=aH+bF$ with $\frac{b}{a}<2$ then there is no class $\xi$ of type $(E)$ with
 $\< H,\xi\>>0> \<P, \xi\>$ with $\delta^{\widehat \P^2}_\xi(kH)\ne 0$.
Therefore  $$\chi^{\widehat\P^2,H_+}_{E}(kH)=\sum_{d>4}\chi(M^{\widehat \P^2}_P(E,d),\mu(kH))\Lambda^d.$$

(a) \eqref{P0H} and \lemref{blowsimple} give
$$\Lambda\chi^{\P^2,H}_{0}(H)=\frac{\Lambda}{(1-\Lambda^4)^3}-\Lambda-9\Lambda^5$$
By \eqref{Ewall} we have thus $$\chi^{\widehat\P^2,H_+}_{E}(H)=\Lambda\chi^{\P^2,H}_{0}(H)+(\<-H, K_{\P^2}\>+\frac{K_{\P^2}^2+H^2}{2}+1)\Lambda^5=\frac{\Lambda}{(1-\Lambda^4)^3}-\Lambda.$$

(b) \eqref{P02H} and \lemref{blowsimple} give
$$\Lambda\chi^{\P^2,H}_{0}(2H)=\frac{\Lambda}{(1-\Lambda^4)^6}-\Lambda-\frac{27}{2}\Lambda^5.$$
By \eqref{Ewall} we have $$\chi^{\widehat\P^2,H_+}_{E}(2H)=\Lambda\chi^{\P^2,H}_{0}(H)+\frac{27}{2}\Lambda^5=\frac{\Lambda}{(1-\Lambda^4)^6}-\Lambda.$$

(c) \eqref{P02H}, \lemref{blowsimple} and \eqref{Ewall} give that
$$\chi^{\widehat\P^2,H_+}_{E}(3H)=\Lambda\left(\frac{1+\Lambda^8}{1-\Lambda^{10}}-1-19\Lambda^4\right)+19\Lambda^5.$$

(2) Let $\xi\in H+2H^2(X,\Z)$ with $\< H,\xi\>>0> \<F,\xi\>$ and with $\delta^{\widehat \P^2}_\xi(2H)\ne 0$ for some $d$.
Then \thmref{vanwall} gives that $\xi=(2n-1)F-(2m-1)E$ with $n,m\in \Z_{>0}$ and
$(4n-2)(2m-1)+(2m-1)^2\le |(8n-4-2m+1|+2$. Thus either $(4n-3)(2m-1)+(2m-1)^2\le -8n+6$, which is impossible, or
$(4n-1)(2m-1)+(2m-1)^2\le 8n-2$, which implies $m=1$.

Thus if $P=aH+bF$ with $\frac{b}{a}<1$ then there is no class $\xi$ of type $(E)$ with
 $\< H,\xi\>>0> \<P, \xi\>$ with $\delta^{\widehat \P^2}_\xi(2H)\ne 0$.
Therefore  $$\chi^{\widehat\P^2,H_+}_{H}(2H)=\sum_{d>0}\chi(M^{\widehat \P^2}_P(H,d),\mu(H))\Lambda^d.$$

By \eqref{PH2H} and \lemref{blowsimple} we have
$\chi^{\widehat\P^2,H}_{H}(2H)=\chi^{\P^2,H}_{H}(2H)=\frac{\Lambda^3}{(1-\Lambda^4)^6}$. As $H$ does not lie on a wall of type $(H)$, we get
$\chi^{\widehat\P^2,H_+}_{H}(2H)=\chi^{\widehat\P^2,H}_{H}(2H)$.
\end{proof}

\section{Le Potier's strange duality.}\label{lpsd}
We have briefly recalled the setting of strange duality conjecture in \secref{strd}.  We use the same notations as in \secref{strd}, Theorem \ref{p11t} and Theorem \ref{P22}.  We will prove Theorem \ref{sdmain} in this section.
%\begin{Theorem}\label{sdmain}Let $X=\P^2$, $\P^1\times\P^1$ or $\widehat \P^2$.  Let the polarization $H$ be both $c$-general and $c^*$-general.  Then the strange duality map $SD_{c,c^*}$ in Section \ref{strd} is an isomorphism for the following three cases

%(1) $c=(2,0,c_2)$ with $c_2>2$ and $c^*=(0,-K_X,\chi=0)$, moreover if $X=\P^1\times\P^1$ or $\widehat \P^2$, then $H=aF+bG$ with $a\geq b$;

%(2) $X=\P^1\times\P^1$ or $\widehat \P^2$ with $H=aF+bG$ and $\frac ab\geq \frac 54$, $c=(2,0,c_2)$ with $c_2>2$ and $c^*=(0,2G+3F,\chi=0)$;

%(3) $X=\P^2$ with $H$ the hyperplan class, $c=(2,H,c_2)$ with $c_2>0$ and $c^*=(0,2H,\chi=-1)$.
%\end{Theorem}
\begin{Remark}\label{ppd}(1) If $c^*=(0,L,\chi=0)$, then $M_H^X(c^*)$ does not depend on the polarization $H$, hence any ample class $H$ is $c^*$-general.  Actually if $c^*$ is of rank 0, then the morphism $\lambda:~K_{c^*,H}\rightarrow \Pic(M_H^X(c^*))$ introduced in \secref{sec:detbun} can be extended to $K_{c^*}$ for any ample $H$.

(2) $H$ may not be $c$-general with $c=(2,0,c_2)$, then %$M^X_H(c)^s\subsetneq M^X_H(c)$ and 
the determinant line bundle $\mathcal{D}_{c,c^*}=\lambda(c^*)$ is not well-defined over the whole space $M^X_H(c)$, unless $\<L,\xi\>=0$ for every $\xi$ of type $c$ with $\<H,\xi\>=0$ and $
\xi^2+4c_2\in 8\mathbb{Z}_{\geq0}$.  One can see this from the construction of $\lambda(c^*)$: we first have a good $GL(V)$-quotient $\Omega(c)\ra M^X_H(c)$ with $\Omega(c)$ an open subset of some Quot-scheme.  There is a universal family $\mathscr{F}$ over $\Omega(c)$, and we get $\lambda(c^*)$ by descending the determinant line bundle $\lambda_{\mathscr{F}}(c^*)$ over $\Omega(c)$ to $M^X_H(c)$.  $\lambda_{\mathscr{F}}(c^*)$ is $GL(V)$-linearized and  $\lambda(c^*)$ is well-defined if and only if $\lambda_{\mathscr{F}}(c^*)$ satisfies the descent condition (see Theorem 4.2.15 in \cite{HL}).  Hence we know that 
$\lambda(c^*)$ is certainly well-defined over the stable locus $M^X_H(c)^s$ since $c^*\in K_c$.  Denote this line bundle by $\mathcal{D}_{c,c^*}^s$.  There are also  strictly semistable points in $M^X_H(c)$, which correspond to S-equivalence classes of sheaves $\I_Z(\xi)\oplus \I_W(-\xi)$, with $\<H,\xi\>=0$ and $
\xi^2+4c_2\in 8\mathbb{Z}_{\geq0}$ and $len(Z)=len(W)=\xi^2/8+c_2/2$.  Let $c_{+\xi}$ (resp. $c_{-\xi}$) be the class of $\I_Z(\xi)$ (resp. $\I_W(-\xi)$) in $K(X)$.  Then by the descent condition, $\lambda(c^*)$ is well-defined over the strictly semistable point $[\I_Z(\xi)\oplus \I_W(-\xi)]$ if and only if $\chi(c^*\otimes c_{\pm\xi})=0$ (see the proof of Theorem 8.1.5 in \cite{HL}) which is equivalent to say that $\<L,\xi\>=0$ since $c^*=(0,L,\chi=0)$.% If $d>8$, $M(c)-M(c)^s$ is of codimension $\geq 2$ in $M(c)$.  Moreover because $-K_X$ is ample, $M(c)$ has only rational singularities, hence $\text{Pic}(M(c))=\text{Pic}(M(c)^s)$
%\begin{LG} Can you give references for this (also inside the paper). More seriously, is the condition $
%\xi^2+4c_2\in 8\mathbb{Z}_{\geq0}$ correct (as I said above this is NOT the correct condition for a wall). 
%\end{LG}
%\begin{YY}I think now it is clear.\end{YY}
\end{Remark}

Denote by $M^X_H(c)^g$ the biggest open subset of $M^X_H(c)$ where $\lambda(c^*)$ is well-defined.  We denote this line bundle over $M^X_H(c)^g$ by $D_{c,c^*}^g$.
Notice that $M^X_H(c)^s\subset M^X_H(c)^g$ and by Remark \ref{ppd} sheaves of the  form $\I_Z(c_1/2)\oplus \I_W(c_1/2)$ are in $M^X_H(c)^g$, where $Z,W$ are $0$-dimensional subschemes of length $d/8$.  This is because in this case $\xi=0$. 
%\begin{LG} why are they in $M^X_H(c)^g?$\end{LG} \begin{YY} I think now it is ok.\end{YY} 
We have $M^X_H(c)^g=M^X_H(c)$ if the polarization $H$ is $c$-general or proportional to $L$.    Replacing $M^X_H(c)$ by $M^X_H(c)^g$ and $\mathcal{D}_{c,c^*}$ by $\mathcal{D}_{c,c^*}^g$, we can define a morphism \begin{equation}\label{vsdmap}SD_{c,c^*}^g:H^0(M^X_H(c)^g,\D_{c,c^*}^g)^\vee \to H^0(M^X_H(c^*),\D_{c^*,c}).\end{equation} analogously to \eqref{SDmap}. We will prove the  following theorem which implies Theorem \ref{sdmain}.
\begin{Theorem}\label{vsdmain}$SD_{c,c^*}^g$ is an isomorphism for the three cases of \thmref{sdmain}.
%\begin{LG} it is not nice to refer to section \ref{insd}, it would be better to have a different thing to refer to.\end{LG}
\end{Theorem}

\subsection{Numerical condition.}
We first show that the two vector spaces $H^0(M^X_H(c),\mathcal{D}_{c,c^*})$ (or $H^0(M^X_H(c)^g,\mathcal{D}_{c,c^*}^g)$ in general) and $H^0(M^X_H(c^*),\mathcal{D}_{c^*,c})$ have the same dimension.
\begin{Proposition}\label{nucon} For the three cases of \thmref{sdmain}, we have
\begin{equation}\label{fir:nucon}dim~H^0(M^X_H(c)^g,\mathcal{D}_{c,c^*}^g)=dim~H^0(M^X_H(c)^s,\mathcal{D}_{c,c^*}^s)=dim~H^0(M^X_H(c^*),\mathcal{D}_{c^*,c}).\end{equation}
In particular if $H$ is $c$-general or proportional to $L$, then $M_H^X(c)^g=M_H^X(c)$ and 
\begin{equation}\label{sec:nucon}dim~H^0(M^X_H(c),\mathcal{D}_{c,c^*})=dim~H^0(M^X_H(c)^s,\mathcal{D}_{c,c^*}^s)=dim~H^0(M^X_H(c^*),\mathcal{D}_{c^*,c}).\end{equation}
\end{Proposition}
\begin{proof} Since $c^*=(0,L,\chi=-\<\frac{c_1(L)}2, c_1\>)$ for the three cases of \thmref{sdmain}, $\D_{c,c^*}=\mu(L)$.  

In the cases (1) and (2) of \thmref{sdmain}, every strictly semistable sheaf is S-equivalent to a sheaf of the form $\mathcal{I}_Z(\xi/2)\oplus\mathcal{I}_W(-\xi/2)$ with $\xi$ a class of type $c$ such that $\<H,\xi\>=0$ and  $0$-dimensional subschemes $Z$, $W$ of length $(4c_2+\xi^2)/8\leq c_2/2$.  Hence the locus $M^X_H(c)^{sss}$ of strictly semistable sheaves is of codimension $\geq 2$ in $M^X_H(c)$ because $dim~M^X_H(c)^s=4c_2-3$ and $c_2>2$.

Furthermore $-K_X$ is ample and hence $M^X_H(c)$ has only rational singularities.  %If $\D_{c,c^*}$ is well-defined over all $M(c)$, in particular if the polarization is $c$-general, 
Therefore we have
\begin{equation}\label{hsh}dim~H^0(M^X_H(c)^g,\mathcal{D}^g_{c,c^*})=dim~H^0(M^X_H(c)^s,\mathcal{D}_{c,c^*}^s).
\end{equation}
The dimension of $H^0(M^X_H(c^*),\D_{c^*,c})$ has been computed in \cite{Yuan} (see Theorem 4.4.1 and Theorem 4.5.2 in \cite{Yuan}) for cases (1) and (2).  Theorem \ref{p11t}, Theorem \ref{P22} and Proposition \ref{vanihh} provide the dimension of $H^0(M^X_H(c),\D_{c,c^*})$ for $H$ $c$-general.  By comparing those two results, we get (\ref{sec:nucon}) for cases (1) and (2) with $c$-general polarization.

Now we assume that the polarization $H$ lies on a wall $W^{\xi}$ with $\xi$ a class of type $c$.  We want to show (\ref{fir:nucon}) for cases (1) and (2).  With no loss of generality, we assume $\<\xi, K_X\>\leq0$.  Let $H_{+}$ be a $c$-general polarization lying in the chamber next to $W^{\xi}$ such that $\<\xi, H_{+}\>>0$.  We have a surjective birational map
\[\rho:M^X_{H_{+}}(c)^s\dashrightarrow M^X_H(c)^s,\]
by sending every sheaf to itself, which is an isomorphism outside the locus $E_{H_{+}}$ consisting of $H_+$-stable sheaves lying in the following exact sequence
\begin{equation}\label{niso}0\rightarrow \mathcal{I}_Z(-\widetilde{\xi}/2)\rightarrow \F \rightarrow \mathcal{I}_{W}(\widetilde{\xi}/2)\rightarrow 0,\end{equation}
with $\<\widetilde{\xi},H_{+}\>>0$ and $\<\widetilde{\xi},H\>=0$, and $Z,W$ $0$-dimensional subschemes satisfying $len(Z)+len(W)=c_2+\widetilde{\xi}^2/4$. %where $l_{\widetilde{\xi}}:=c_2/2+\widetilde{\xi}^2/8$.  
Because now we have $\Pic(X)\cong H^2(X,\mathbb{Z})$ is free of rank 2, we must have $\widetilde{\xi}=a\xi$ for some $a>0$ and hence $\<\widetilde{\xi},K_X\>\leq 0$. 
%\begin{LG} I think this statement is wrong. You should take the sheaves lying is such extensions with $len(Z)+len(W)=c_2-\xi^2/4$. Because the map will not be defined for all these sheaves.
%Strictly speaking you should also be a bit more careful. The wall $W^\xi$ can be defined by more then one class. Assume that $\xi$ has minimal $-\xi^2$, among the classes of type $c$ defining $W^\xi$. Then all positive multiples $\eta$ of $\xi$ which still satisfy $\eta^2+4c_2\le 0$ will also define the wall $W^\xi$, and one has to consider the above extension for all of them.
%\end{LG}\begin{YY}Yes, you are right.  Now I have fixed the problem.  Don't worry. \end{YY}
  It is obvious that $\rho$ identifies the determinant line bundles $\D_{c,c^*}$ on both sides.  Hence to show (\ref{fir:nucon}) for $H$, it is enough to show $E_{H_{+}}$ is of codimension $\geq 2$ in $M^X_{H_{+}}(c)^s$, which has pure dimension $d-3=4c_2-3$.  

$\Hom(\I_W(\widetilde{\xi}/2),\I_Z(-\widetilde{\xi}/2))=0$ because $\<H_{+},\widetilde{\xi}\>>0$.   Since $-K_X$ is ample and $\<H,\widetilde{\xi}\>=0$, we have $\Ext^2(\I_W(\widetilde{\xi}/2),\I_Z(-\widetilde{\xi}/2))\cong\Hom(\I_Z(-\widetilde{\xi}),\I_W(\widetilde{\xi}+K_X))^{\vee}=0$.  Hence 
\[dim~\Ext^1(I_W(\widetilde{\xi}/2),\I_Z(-\widetilde{\xi}/2))=-\chi(I_W(\widetilde{\xi}/2),\I_Z(-\widetilde{\xi}/2))=\<\widetilde{\xi}\cdot K_X\>/2-\widetilde{\xi}^2/4+c_2.\]
Hence \[dim ~E_{H_{+}}\leq \displaystyle{\max_{\tiny{\begin{array}{c}\widetilde{\xi}/2\in H^2(X,\mathbb{Z}),2l_{\widetilde{\xi}}\in\mathbb{Z}_{\geq0}\\
\<\widetilde{\xi},H_{+}\>>0,\<\widetilde{\xi},H\>=0\end{array}}}}\{3c_2+\widetilde{\xi}^2/4+\<\widetilde{\xi}, K_X\>/2-1\}.\]  
Since $\<\xi, K_X\>\leq 0,$  $dim~E_{H_{+}}\leq \displaystyle{\max_{\tiny{\begin{array}{c}\widetilde{\xi}/2\in H^2(X,\mathbb{Z}),2l_{\widetilde{\xi}}\in\mathbb{Z}_{\geq0}\\
\<\widetilde{\xi},H_{+}\>>0,\<\widetilde{\xi},H\>=0\end{array}}}}\{3c_2+\widetilde{\xi}^2/4-1\}$.  Then
\begin{equation}N_c:=dim~M_{H_+}(c)^s-dim~E_{H_+}\geq \displaystyle{\min_{\tiny{\begin{array}{c}\widetilde{\xi}/2\in H^2(X,\mathbb{Z}),2l_{\widetilde{\xi}}\in\mathbb{Z}_{\geq0}\\
\<\widetilde{\xi},H_{+}\>>0,\<\widetilde{\xi},H\>=0\end{array}}}}\{c_2-2-\widetilde{\xi}^2/4\}.
%=\displaystyle{\min_{\tiny{\begin{array}{c}\widetilde{\xi}/2\in H^2(X,\mathbb{Z}),2l_{\widetilde{\xi}}\in\mathbb{Z}_{\geq0}\\
%\<\widetilde{\xi},H_{+}\>>0,\<\widetilde{\xi},H\>=0\end{array}}}}\{2(l_{\widetilde{\xi}}-1)-\widetilde{\xi}^2/2\}.
\end{equation}
Since $\widetilde{\xi}/2\in H^2(X,\mathbb{Z})$ and $\widetilde{\xi}^2<0$, $-\widetilde{\xi}^2/4\geq 1$.  Moreover since $c_2>2$, we have $c_2-2-\widetilde{\xi}^2/4\geq 2$ and hence $N_c\geq 2$.  This proves the claim  for cases (1) and (2).  

For case (3), $M^X_H(c)=M^X_H(c)^s=M^X_H(c)^g$ and there is no wall.  By Theorem 3.5 in \cite{LP2} $M^X_H(c^*)\cong |2H|\cong \P^5$ and moreover a sheaf $\G \in M^X_H(c^*)$ with support $C_{\G}$ is isomorphic to $\mathcal{O}_{C_{\G}}\otimes\mathcal{O}_{\P^2}(-1)$.  We have $c=(2,H,c_2)$.  If $c_2=1$, then $M^X_H(c)$ consists of only one point $[\cal T_{\P^2}(-1)]$ with $\cal T_{\P^2}$ the tangent bundle over $\P^2$.  $H^0(\cal T_{\P^2}(-2)\otimes\mathcal{O}_{C})=0$ for any $C\in |2H|$.  Hence $\D_{c^*,c}\cong \oo_{|2H|}$ for $c_2=1$ and for $c_2>1$ we have $\D_{c^*,c}\cong \oo_{|2H|}(c_2-1)$ by the following lemma due to Le Potier.  
\begin{Lemma}[Proposition 2.8 in \cite{LP3}]\label{point}If $x$ is not a base point of $|L|$, then the determinant line bundle $\lambda_x:=\lambda([\oo_x])$ associated to the skyscraper sheaf $[\oo_x]$ over $M(0,L,\chi)$ satisfies  $\lambda_x\cong \pi^{*}\oo_{|L|}(-1)$, where $\pi:M(0,L,\chi)\rightarrow |L|$ is the projection sending every sheaf to its support.
\end{Lemma}
By Theorem \ref{P22} and Proposition \ref{vanihh} we have for case (3)
\[dim~H^0(M(c),D_{c,c^*})=dim~H^0(\P^5,\oo_{\P^5}(c_2-1))=dim~H^0(M(c^*),\D_{c^*,c})={c_2+4 \choose c_2-1}.\]
This finishes the proof of the  proposition.
\end{proof}

\subsection{Strange duality.}
Recall that we have introduced in Section \ref{strd} the locus 
$$\mathscr{D}:=\big\{([\F],[\G])\in   M^X_H(c)\times M^X_H(c^*)\bigm| H^0(X,\F\otimes \G)\ne 0\big\},$$
 which gives a canonical section $\sigma_{c,c^*}$ of the line bundle $\D:=\D_{c,c^*}\boxtimes\D_{c^*,c}\in\Pic(M^X_H(c)\times M^X_H(c^*))$ and induces the strange duality map
\begin{equation}\label{sdmap}SD_{c,c^*}:H^0(M^X_H(c),\D_{c,c^*})^\vee \to H^0(M^X_H(c^*),\D_{c^*,c}).\end{equation}
If the polarization is not $c$-general, we have the map
\begin{equation}\label{vsdmap}SD_{c,c^*}^g:H^0(M^X_H(c)^g,\D_{c,c^*}^g)^\vee \to H^0(M^X_H(c^*),\D_{c^*,c}).\end{equation}

To show Theorem \ref{sdmain} (resp. Theorem \ref{vsdmain}), it is by Proposition \ref{nucon} enough to show that $SD_{c,c^*}$ (resp. $SD_{c,c^*}^g$) is surjective. 

Let $\F$ be a semistable sheaf of class $c$.  Denote by $s_{\F}$ the section of $\D_{c^*,c}$ over $M^X_H(c^*)$ up to scalars defined by the canonical divisor $D_{\F}:=\big\{\G\in M^X_H(c^*)|H^0(\F\otimes \G)\ne 0\big\}.$ 
\begin{Lemma}$s_{\F}$ only depends on the $S$-equivalence class of $\F$.  More precisely, take a Jordan-H\"older filtration of $\F$
\[0=JH_0\subset JH_1\subset\ldots\subset JH_{\ell}=\F\]
and let $J_i:=JH_i/JH_{i-1}$, then $s_{\F}=\Pi_{i\geq 1}s_{J_i}$.
\end{Lemma}
\begin{proof}It is enough to show that $D_{\F}=\cup_{i\geq 1}D_{J_i}$.  Since the strictly semistable locus in $M^X_H(c^*)$ forms a subset of codimension $\geq 2$ (See Proposition 3.4 in \cite{LP2} and Lemma 4.2.6 in \cite{Yuan}), It is enough to show that $D_{\F}$ and $\cup_{i\geq1}D_{J_i}$ coincide on the stable locus $M^X_H(c^*)^s$.  We write down an exact sequence
\begin{equation}\label{see}0\ra J_1\ra \F\ra \F/J_1\ra 0.
\end{equation}
Take any $\G$ stable sheaf of class $c^*$.  We tensor (\ref{see}) by $\G$ and get
\begin{equation}\label{saw}0\ra \G\otimes J_1\ra \G\otimes \F\ra \G\otimes \F/J_1\ra 0,
\end{equation}
and
\begin{equation}\label{go}0\ra H^0(\G\otimes J_1)\ra H^0(\G\otimes \F)\ra H^0(\G\otimes \F/J_1)\ra H^1(\G\otimes J_1).\end{equation}
Notice that $\chi(\G\otimes J_1)=h^0(\G\otimes J_1)-h^1(\G\otimes J_1)=0.$  By (\ref{go}) we see that $H^0(\G\otimes \F)\ne 0\Leftrightarrow H^0(\G\otimes J_1)\ne 0$ or $H^0(\G\otimes \F/J_1)\ne 0$.  Hence we get $D_{\F}=D_{J_1}\cup D_{\F/J_1}$, and the lemma follows by applying the induction assumption to $\F/J_1$.
\end{proof}

We then have the following criterion for the surjectivity of maps $SD_{c,c^*}$ and $SD^s_{c,c^*}$.
\begin{Proposition}\label{ned}If there are finitely many semistable sheaves $\F_i$ $i\in I$ of class $c$ such that $\{s_{\F_i}\}_{i\in I}$ spans $H^0(M^X_H(c^*),\D_{c^*,c})$, then the map $SD_{c,c^*}$ in (\ref{sdmap}) is surjective.  Moreover, if $\F_i$ can be chosen to be in $M^X_H(c)^g$, then $SD_{c,c^*}^g$ in (\ref{vsdmap}) is surjective.
\end{Proposition}
\begin{proof}This is very standard in linear algebra.  We have the section $\sigma_{c,c^*}$ associated to the divisor $\mathscr{D}$ over $M^X_H(c)\times M^X_H(c^*)$.  Let $\{e_k\}_{1\leq k\leq m}$ and $\{f_j\}_{1\leq j\leq l}$ be the basis of $H^0(M^X_H(c^*),\D_{c^*,c})$ and $H^0(M^X_H(c),\D_{c,c^*})$ respectively.  Then we have $\sigma_{c,c^*}=\sum a_{kj}e_k\otimes f_j$ with $a_{kj}$ constant.  Thus we have a $m\times l$ matrix $A=(a_{kj})$, and the map $SD_{c,c^*}$ is surjective if and only if $A$ is of rank $m$. 

On the other hand $s_{\F_i}=\sum a_{kj}\cdot f_j([\F_i])\cdot e_k$, where $f_j([\F_i])$ is the value of $f_j$ at point $[\F_i]$.  Because $\{s_{\F_i}\}$ spans $H^0(M^X_H(c^*),\D_{c^*,c})$, we can recover $e_k$ as a linear combination of the $s_{\F_i}$, hence we can find an $l\times m$ matrix $B$ such that $A\cdot B=Id_m$, hence $A$ is of rank $m$.  The statement for $SD^g_{c,c^*}$ also follows if $[\F_i]\in M(c)^g$ for all $i$.
 \end{proof}

We will use Proposition \ref{ned} to show the surjectivity of $SD_{c,c^*}^g$ for the three cases of \thmref{sdmain}.

~~~
\begin{flushleft}{\textbf{$\dagger$ cases (1) and (2).}}\end{flushleft} %%%%  Case (1) and (2)  %%%%%

In these two cases we have $c^*=(0,L,\chi=0)$ with $L=-K_X$ or $L=-K_X+F$ and $c=(2,0,c_2)$.  

If $c^*=(0,L,\chi=0)$ with $L$ some effective line bundle, we denote by $\pi$ the projection $M^X_H(c^*)\ra |L|$ sending every sheaf to its support.  Denote by $\z$ the determinant line bundle associated to $[\oo_X]$ over $M^X_H(c^*)$.  Therefore by Lemma \ref{point} we have $\D_{c^*,c}\cong \z^2(c_2):=\z^2\otimes\pi^{*}\oo_{|L|}(c_2)$ with $c=(2,0,c_2)$.    Notice that these notations are compatible with those in \cite{Yuan}.
\begin{Lemma}\label{geno}For $c^*=(0,L,\chi=0)$ with $L=-K_X$ or $L=-K_X+F$, the multiplication map 
\[m_1:~H^0(M^X_H(c^*),\z)\otimes H^0(M^X_H(c^*),\pi^{*}\oo_{\ls}(c_2))\ra H^0(M^X_H(c^*),\z(c_2)),\]
is surjective. 
\end{Lemma}
\begin{proof}By Lemma 4.3.3 in \cite{Yuan}, $\pi_{*}\oo_{M^X_H(c^*)}\cong \oo_{\ls}$.  Thus $\pi_{*}(\pi^{*}\oo_{\ls}(n))\cong \oo_{\ls}(n)$.  By Theorem 4.4.1 and Theorem 4.5.2 in \cite{Yuan}, $\pi_{*}\z\cong\oo_{\ls}$ and the multiplication map
\[\widetilde{m}_1:~H^0(\ls,\pi_{*}\z)\otimes H^0(\ls,\pi_{*}\pi^{*}(\oo_{\ls}(n)))\ra H^0(\ls,\pi_{*}\z(n))\]
is surjective.  Hence so is $m_1$.
\end{proof}

We choose a finite collection of distinct points $\{x_j\}_{j\in J}$ on $X$, and associate to each point $x_j$ a divisor in $\ls$ consisting of curves passing through $x_j$, which gives a section $t_{x_j}$ of $\oo_{\ls}(1)$.  Let $d_L=dim~\ls$, then it is possible to choose $d_L+1$ distinct points $x_j$ such that $\{t_{x_j}\}_{j=1}^{d_L+1}$ spans $H^0(\oo_{\ls}(1))$.  Hence we can choose $n(d_L+1)$ distinct points $x^k_j$ with $1\leq j\leq d_L+1$, $1\leq k\leq n$ such that $\{t_{j_1,\cdots,j_n}\}$ spans $H^0(\oo_{\ls}(n))$, where $t_{j_1\cdots,j_n}$ is defined as follows.

\[t_{j_1,\cdots,j_n}:=\prod_{k=1}^n t_{x_{j_k}^k},~with~t_{x_{j_k}^k} the~section~associated~to~x^k_{j_k}.\]

Let $R$ be a subset of $\{x_{j}^k\}$ and denote by $\I_R$ the ideal sheaf of points appearing in $R$.  Then $s_{\I_R}=s_{\oo_X}\times \pi^{*}t_R$ where $t_R:=\prod_{x\in R} t_x$.  For any two disjoint subsets $R$ and $T$ of $\{x_j^k\}$ such that $\#R=[\frac n2]$ is the round-down of $\frac{n}{2}$ and $\#T=n-[\frac n2]$, we define a rank 2 torsion free sheaf $\F_{R,T}$ to be an extension of $\I_R$ by $\I_T$, i.e.
\begin{equation}\label{first}0\ra \I_T\ra \F_{R,T}\ra \I_R\ra 0.
\end{equation}
Moreover if $n\geq 2$ and $n$ is odd, we ask (\ref{first}) not to split.
It is easy to see that for all $n\geq 2$ or $n=0$, $\F_{R,T}$ is semistable of class $c=(2,0,c_2=n)$ and $s_{\F_{R,T}}=s_{\I_T}\cdot s_{\I_R}$.  Hence we have the following lemma as an easy corollary of Lemma \ref{geno}
\begin{Lemma}\label{fog}For $n\geq 2$ or $n=0$, the set $\{s_{\F_{R,T}}\}_{R,T}$ spans the image of $H^0(M^X_H(c^*),\z(c_2))$ in $H^0(M^X_H(c^*),\z^2(c_2))$ via the natural embedding induced by the following sequence
\begin{equation}\label{the}0\ra \z(c_2)\ra \z^2(c_2)\ra\z^2(c_2)|_{D_{\z}}\ra 0,\end{equation}
where $D_{\z}:=\big\{\G\in M^X_H(c^*)\bigm |H^0(\G)\ne 0\big\}$ is the canonical divisor associated to $\z$.
\end{Lemma}
\begin{Remark}\label{fin}Notice that $[\F_{R,T}]\in M^X_H(c)^g$ under any polarization.
\end{Remark}
By Theorem 4.4.1 and Theorem 4.5.2 in \cite{Yuan}, $\z(n)$ has no higher cohomology for $n\geq0$.  Then by (\ref{the}) we have the exact sequence
\begin{equation}\label{gthe}0\ra H^0(\z(n))\ra H^0(\z^2(n))\xrightarrow{\varpi} H^0(D_{\z},\z^2(n)|_{D_{\z}})\ra 0.
\end{equation}
\begin{Lemma}\label{gent}The  multiplication map 
\[m_2:~H^0(D_{\z},\z^2(2)|_{D_{\z}})\otimes H^0(M^X_H(c^*),\pi^{*}(\oo_{\ls}(n-2)))\ra H^0(D_{\z},\z^2(n)|_{D_{\z}})\]
is surjective.
\end{Lemma} 
\begin{proof}By Lemma 4.4.4 and Lemma 4.5.4 in \cite{Yuan}, the multiplication map
\[\widetilde{m}_2:~H^0(\ls,\pi_{*}(\z^2(2)|_{D_{\z}}))\otimes H^0(\ls,\pi_{*}\pi^{*}(\oo_{\ls}(n-2)))\ra H^0(\ls,\pi_{*}(\z^2(n)|_{D_{\z}}))\]
is surjective.  Hence so is $m_2$ since $\pi_{*}(\pi^*\oo_{\ls}(n-2))\cong\oo_{\ls}(n-2)$.
\end{proof}

%Now we use the assumption that there is a semistable vector bundle $E_2$ of class $c_2$.  We first show some properties of $E_2$.
\begin{Proposition}\label{etwo}We can find a set of finitely many slope-stable vector bundles $\{\cal E^i_2\}$ of class $(2,0,c_2=2)$ such that the images of $s_{\cal E_2^i}$ in $H^0(D_{\z},\z^2(2)|_{D_{\z}})$ via map $\varpi$ in (\ref{gthe}) span $H^0(D_{\z},\z^2(2)|_{D_{\z}})$. 
\end{Proposition}
\begin{proof} For case (1), we have $D_{\z}\cong \ls$ and $\z^2(2)\cong\oo_{\ls}$ by Lemma 4.4.3 and Lemma 4.4.4 in \cite{Yuan}.  It is enough to show that there exists a slope-stable vector bundle $\cal E_2$ of class $(2,0,c_2=2)$, such that  for all $\G\in D_{\z}$, $H^0(\G\otimes \cal E_2)=0$.  If we are on $\P^2$, then any semistable bundle of class $(2,0,c_2=2)$ is slope-stable.  If we are on the  two Hirzebruch surfaces, then by Lemma \ref{slop} and Remark \ref{apple} below there are slope-stable bundles of class $(2,0,c_2=2)$ for any polarization.   

Notice that in this case $\G\in D_{\z}\Leftrightarrow \G\cong \oo_{C_{\G}}$ with $C_{\G}$ the support of $\G$.  Let $\cal E_2$ be any slope-stable vector bundle of class $(2,0,c_2=2)$.  We want to show that $H^0(\oo_C\otimes \cal E_2)=0$ for all curves $C\in\ls.$  We have the following exact sequence
\[0\ra K_X\ra \oo_X\ra\oo_C\ra0.\]
Tensoring it by $\cal E_2$ and taking global sections,  we get
\[0\ra H^0(K_X\otimes \cal E_2)\ra H^0(\cal E_2)\ra H^0(\cal E_2\otimes\oo_C)\ra H^1(K_X\otimes\cal E_2).\]
By stability we have $H^0(\cal E_2)=H^2(\cal E_2)=0$ and moreover since $\chi(\cal E_2)=0$, we have $H^1(\cal E_2)=H^1(K_X\otimes \cal E_2^{\vee})^{\vee}=0$.  Because $\cal E_2$ is a rank 2 bundle with trivial determinant, we have $\cal E_2^{\vee}\cong \cal E_2$.  Hence $H^1(\cal E_2)=H^1(K_X\otimes \cal E_2)=0$, hence $H^0(\cal E_2\otimes\oo_C)=0$ for all $C\in\ls.$  This finishes the proof for case (1).

Now we deal with Case (2).  In this case $X=\P^1\times\P^1$ or $\widehat \P^2$.  We can write $X=\P(\oo_{\P^1}\oplus\oo_{\P^1}(e))$ with $e=0$ for $\P^1\times\P^1$ and $e=1$ for $\widehat \P^2$.  

Define the section class $\Xi:=G-eF/2$.  Then $\Xi^2=-e$ and $a\Xi+bF$ is ample if and only if $a,b>0$ and $b>ae$.  Choose three distinct points $x_1,x_2,x_3$ on $X$ such that no two of them lie on a divisor of class $F$ or $\Xi$.  Denote by $\I_{j}$ the ideal sheaf of $\{x_i,i\neq j\}$, hence we have $H^0(\I_j(F))=H^0(\I_j(\Xi))=0$ for all $1\leq j\leq 3$.  It is easy to compute that $\text{Ext}^1(\I_j(F),\oo_X(-F))\neq 0$.  Let $\cal E^i_2$ be a vector bundle lying in the following exact sequence.
\begin{equation}\label{de}0\ra\oo_X(-F)\ra \cal E^i_2\ra \I_j(F)\ra 0.\end{equation} 
There exists such vector bundle is because the Cayley Bacharach condition is fulfilled  by $H^0(K_X(2F))=0$.  
%\begin{LG} why is there such a sheaf. I think you should say that the Cayley Bacharach condition is fulfilled because $-K_X$ is ample.\end{LG}\begin{YY} Yes, I have added an explanation.\end{YY}
Lemma \ref{slop}, Remark \ref{apple} and Lemma \ref{three} below imply the statement for Case (2).  
This proves the proposition. 
\end{proof}

To deal with Case (2), we have the following three lemmas.
%\begin{LG} Maybe you should say that now you are in case (2)\end{LG}\begin{YY}Ok.\end{YY}
\begin{Lemma}\label{slop}If $e=1$, $\cal E^i_2$ is slope-stable for any polarization.  If $e=0$, $\cal E^i_2$ is slope-stable for the polarization $G+aF$ for $a\geq 1$.  
\end{Lemma}
\begin{proof}Choose any polarization $P=\Xi+\nu F$ with $\nu\in\mathbb{Q},\nu>e$.  Since $\cal E^i_2$ is locally free, it is enough to show that for any divisor $ S=a\Xi+b F$ with $a,b\in\mathbb{Z}$ such that $\<S,P\>\geq 0$, we have $\Hom(\oo_X(S),\cal E^i_2)=0$.  By (\ref{de}) it is enough to show $\Hom(\oo_X(S),\oo_{X}(-F))=0=H^0(\oo_X(S),\I_j(F))$.  

We have $\Hom(\oo_X(S),\oo_X(-F))=H^0(\oo_X(-a\Xi-(b+1)F))=0$ because $\<P,(-a\Xi-(b+1)F\><0$.  
If $\Hom(\oo_X(S),\I_j(F))=H^0(\I_j(-a\Xi-(b-1)F))\neq 0$, then $a\leq 0$, $b\leq 1$ and $-a\Xi-(b-1)F\ne 0$.  But $\<(a\Xi+bF),P\>=(\nu-e)a+b\geq 0$, hence we have $0>a\geq -1/(\nu-e)$ and $b=1$ or $a=b=0$.  

If $e=1$, then $H^0(\oo_X(k\Xi))=H^0(\oo_X(\Xi))\cong\C$ for all $k\geq 1$.   Hence $H^0(\I_j(k\Xi))=H^0(\I_j(F))=0$ for any $k$ because $\I_j$ is an ideal sheaf of two distinct points not lying on a divisor of class $F$ or $\Xi$.   Hence $\Hom(\oo_X(S),\I_j(F))=0$ and hence $\cal E^i_2$ is slope-stable for any polarization.

If $e=0$, then we ask $\nu\geq1$ and hence $a=-1$, $b=1$ or $a=b=0$.  But $H^0(\I_j(\Xi))=H^0(\I_j(F))=0$ and hence $\cal E^i_2$ is slope-stable. 
\end{proof}
\begin{Remark}\label{apple}In Lemma \ref{slop}, if $X=\P^1\times\P^1$ then $G=\Xi$ and $F$ are symmetric, hence we can always write a polarization as $G+aF$ with $a\geq 1$.  Hence for any polarization, we can construct slope-stable vector bundles $\cal E_2^i$. 
\end{Remark}
\begin{Lemma}\label{three}The restrictions of $s_{\cal E^i_2},1\leq i\leq3$ to $D_{\z}$ span $H^0(D_{\z},\z^2(2)|_{D_{\z}})$. 
\end{Lemma}
\begin{proof}It is enough to choose three distinct points $\G^1,\G^2,\G^3\in D_{\z}$ such that $s_{\cal E^i_2}(\G^j)\neq0$ if and only if $i=j$.  In other words, $H^0(\G^j\otimes \cal E^i_2)=0$ if and only if $i=j$.  

Recall that we have chosen three distinct points $x_1,x_2,x_3$.  Denote by $f_i$ the fiber passing through $x_i$, then $f_i\cong\P^1$ and $f_i\cap f_j=\emptyset $ for $i\neq j$.   Choose a smooth curve $C\in |-K_X|.$  Then we define $\G^i$ to be ($S$-equivalent to) $\oo_C\oplus\oo_{f_i}(-1)$, where $\oo_{f_i}(-1)\cong \oo_{\P^1}(-1)$.

By Lemma \ref{etwo} we see that $H^0(\cal E_2^i\otimes\oo_C)=0$ for $1\leq i\leq 3$.  Therefore $H^0(\cal E^i_2\otimes \G^j)= 0\Leftrightarrow H^0(\cal E_2^i\otimes\oo_{f_j}(-1))=0$.  
By (\ref{de}) we see that $\cal E_2^i|_{f_i}\cong \oo_{f_i}^{\oplus 2}$, while for $i\neq j$, $\cal E_2^i|_{f_j}\cong \oo_{f_i}(-1)\oplus \oo_{f_i}(1)$.  Hence $H^0(\G^j\otimes \cal E^i_2)=0$ if and only if $i=j$.  This proves the lemma.
\end{proof}

Let $W$ be a subset of $\big\{x_j^k\big\}_{1\leq j\leq d_L+1}^{1\leq k\leq n-2}$ such that $\#W=n-2$.  Let $\oo_W\cong \bigoplus_{x\in W}\oo_x$ be the structure sheaf of the subscheme $W$.
Choose a surjective map $h^i:\cal E_2^i\twoheadrightarrow\oo_W$ (one may choose $h^i$ to factor through $\cal E_2^i\twoheadrightarrow \cal E_2^i\otimes\oo_W\cong \oo_W^{\oplus 2}$).  
Let $\F_W^i$ be the kernel of $h^i$. Then we have the following exact sequence
\begin{equation}\label{grow}0\ra \F_W^i\ra \cal E_2^i\xrightarrow{h^i} \oo_W\ra 0.\end{equation}
%where $\oo_W\cong \bigoplus_{x\in W}\oo_x$ is the structure sheaf of the subscheme $W$.  
%\begin{LG} what is precisely the map $\E_2^i\to \oo_W$?\end{LG}\begin{YY}The surjective map $\E_2^i\ra\oo_W$ is certainly not unique and we just pick one to define $\F_W^i$, and any one will be ok.\end{YY}
It is easy to see that $\F_W^i$ is slope-stable and $s_{\F^i_W}=s_{\cal E^i_2}\times \pi^{*}t_{W}$.  Recall that $t_W$ is the section of $H^0(\ls,\oo_{\ls}(n-2))$ vanishing at curves passing through any point in $W$, and moreover $\{t_W\}_W$ spans $H^0(\ls,\oo_{\ls}(n-2))$.  By Lemma \ref{gent} and Proposition \ref{etwo} we have the following lemma.
\begin{Lemma}\label{song}The restriction of $\{s_{\F^i_{W}}\}^i_{W}$ to $D_{\z}$ spans $H^0(D_{\z},\z^2(n)|_{D_{\z}})$ for $n\geq 2$.
\end{Lemma}
By Proposition \ref{ned}, Lemma \ref{fog} and Lemma \ref{song}, we have the following proposition.  
\begin{Proposition}\label{case12}Let $X=\P^2$, $\P^1\times\P^1$ or $\widehat \P^2$.  Let $c^*=(0,L,\chi=0)$ with $L=-K_X$ or $L=-K_X+F$ with $F$ the fiber class for $X$ a Hirzebruch surface, and $c=(2,0,c_2)$.  Then for any polarization $H$ on $X$, the strange duality map
\[SD^g_{c,c^*}:H^0(M^X_H(c)^g,\D^g_{c,c^*})\ra H^0(M^X_H(c^*),\D_{c^*,c})\]
is surjective for all $c_2\geq 2$.
\end{Proposition}
\begin{proof}[Proof of Theorem \ref{vsdmain} for cases (1) and (2)]Combine Proposition \ref{nucon} and Proposition \ref{case12}.\end{proof}
\begin{Remark}\label{esd}In cases (1) and (2), there are still some conditions on the polarization $H$, which are required by Theorem \ref{p11t}.  However, by Theorem \ref{vanwall} we know that the generating function $\sum_{d>0}\chi(M^X_H(c),\D_{c,c^*})$ is essentially independent of $H$.  Hence we know that the strange duality map $SD_{c,c^*}^g$ is an isomorphism for any polarization for $c_2\gg0$.
\end{Remark}

\begin{flushleft}{\textbf{$\dagger$ Case (3).}}\end{flushleft}
In this case, we have $X=\P^2$ with $H$ the hyperplane class, $c^*=(0,2H,\chi=-1)$ and $c=(2,H,c_2\geq1)$.  As we have shown in the proof of Proposition \ref{nucon}, the map $\pi:M^X_H(c^*)\ra |2H|$ is an isomorphism, $\D_{c^*,c}\cong\oo_{|2H|}(c_2-1)$, $\cal T_{\P^2}(-1)$ is slope-stable of class $(2,H,1)$ and $s_{\cal T_{\P^2}(-1)}$ is nowhere vanishing.

Let $W$ be a subset of $\big\{x_j^k\big\}_{1\leq j\leq 6}^{1\leq k\leq c_2-1}$ such that $\#W=c_2-1$.  Let $\oo_W\cong \bigoplus_{x\in W}\oo_x$ be the structure sheaf of the subscheme $W$.
We then choose a surjective map $h:\cal T_{\P^2}(-1)\twoheadrightarrow\oo_W$ and construct a slope-stable sheaf $\F^H_W$ associated to $W$ as the kernel of $h$.  We have  
\begin{equation}\label{bow}0\ra \F^H_W\ra \cal T_{\P^2}(-1)\xrightarrow{h} \oo_W\ra 0.\end{equation}
%\begin{LG} how precisely is the map $ \cal T_{\P^2}(-1)\ra \oo_W$ given\end{LG}\begin{YY}The same as before, any surjective map will be ok.\end{YY} 
It is easy to see that $s_{\F^H_W}=s_{\cal T_{\P^2}(-1)}\times t_{W}=t_W$ and hence $\{s_{\F^H_W}\}_W$ spans $H^0(M(c^*),\D_{c^*.c})$.  Hence by Proposition \ref{ned} we have
\begin{Lemma}\label{case3}The map $SD_{c,c^*}$ is surjective for case (3).\end{Lemma}
\begin{proof}[Proof of Theorem \ref{vsdmain} for case (3)]Combine Proposition \ref{nucon} and Lemma \ref{case3}.
\end{proof}

\end{document}